\documentclass[a4paper,10pt, reqno]{amsart}

\usepackage{amsmath,amssymb,latexsym, amsfonts}
\usepackage{verbatim}
\usepackage{times}
\usepackage{mathbbol}
\usepackage{hyperref}
\usepackage{mathabx}
\usepackage{bbm}

\usepackage{stmaryrd}

\usepackage[usenames]{color}

\newcommand{\compactlist}[1]{\setlength{\itemsep}{0pt} \setlength{\parskip}{0pt} \setlength{\leftskip}{-0.#1em}}

\newcommand{\corm}[1]{C^{#1}_{\rm co}(U,N)}

\newcommand{\lie}{\mathcal{L}}
\newcommand{\BB}{B}

\newcommand{\bb}{{b}}

\newcommand{\DD}{{D}}

\newcommand{\dd}{{d}}
\newcommand{\sss}{{s}}
\newcommand{\ttt}{{t}}

\numberwithin{equation}{section}

\theoremstyle{plain}

\newtheorem{theorem}{Theorem}[section]
\newtheorem{proposition}[theorem]{Proposition}
\newtheorem{prop}[theorem]{Proposition}

\newtheorem{lem}[theorem]{Lemma}
\newtheorem{corollary}[theorem]{Corollary}

\theoremstyle{definition}
\newtheorem{definition}[theorem]{Definition}
\newtheorem{example}[theorem]{Example}

\newtheorem{rem}[theorem]{Remark}

\newcommand{\ahha}{{\scriptscriptstyle{A}}}

\newcommand{\dehhe}{{\scriptscriptstyle{D}}}

\newcommand{\ggii}{{\scriptscriptstyle{G}}}
\newcommand{\akka}{{\scriptscriptstyle{H}}}

\newcommand{\emme}{{\scriptscriptstyle{M}}}
\newcommand{\enne}{{\scriptscriptstyle{N}}}

\newcommand{\esse}{{\scriptscriptstyle{S}}}
\newcommand{\uhhu}{{\scriptscriptstyle{U}}}

\newcommand{\Q}{{\mathbb{Q}}}
\newcommand{\Z}{{\mathbb{Z}}}

\newcommand{\ga}{\alpha} 
\newcommand{\gb}{\beta}  
\newcommand{\gG}{\Gamma}
\newcommand{\gd}{\delta} 
\newcommand{\gD}{\Delta} 
\newcommand{\gve}{\varepsilon} 
  
\newcommand{\gvf}{\varphi}

\newcommand{\go}{\omega} 
\newcommand{\gO}{\Omega}

\newcommand{\gs}{\sigma} 
\newcommand{\gS}{\Sigma}
\newcommand{\gt}{\theta}

\newcommand{\cC}{{\mathcal C}}
\newcommand{\cD}{{\mathcal D}}

\newcommand{\cH}{{\mathcal H}}

\newcommand{\cL}{{\mathcal L}}
\newcommand{\cM}{{\mathcal M}}
\newcommand{\cN}{{\mathcal N}}

\newcommand{\cP}{{\mathcal P}}

\newcommand{\End}{\operatorname{End}}

\newcommand{\Hom}{\operatorname{Hom}}

\newcommand{\ad}{\operatorname{ad}}

\newcommand{\Tor}{{\rm Tor}}
\newcommand{\Ext}{{\rm Ext}}

\newcommand{\Cotor}{\operatorname{Cotor}}

\newcommand{\id}{{\rm id}}

\newcommand{\pr}{{\rm pr} \,}

\newcommand{\Sh}{{\rm Sh}}
\newcommand{\sh}{{\rm sh}}

\newcommand{\due}[3]{{}_{{#2 }} {#1}_{{ #3}}\,}    

\newcommand{\qttr}[5]{{}^{{#2 \!}}_{{#4 \!}} {#1}^{#3}_{{\! #5}}}    

\newcommand{\ubs}[2]{\underset{\scriptscriptstyle{#1}}{\underbrace{#2}}}

\newcommand{\cinfc}[1]{\cC_{\rm c}^\infty(#1)}                      
\newcommand{\cinf}[1]{\cC^\infty(#1)}                               
\newcommand{\pl}{\partial}

\newcommand{\rmref}[1]{{\rm (}\ref{#1}{\rm )}}

\newcommand{{\Hl}}{{H^{\ell}}} 
\newcommand{{\mHop}}{{m_{H^{\rm op}}}} 
\newcommand{{\Hop}}{{H^{\rm op}}} 
\newcommand{{\mUop}}{{m_{U^{\rm op}}}} 
\newcommand{{\mUopp}}{{m_{\scriptscriptstyle{U^{\rm op}}}}} 
\newcommand{{\Uop}}{{U^{\rm op}}}
\newcommand{{\mVop}}{{m_{V^{\rm op}}}} 
\newcommand{{\Vop}}{{V^{\rm op}}}  
\newcommand{{\Ae}}{{A^{\rm e}}}
\newcommand{{\Ue}}{{U^{\rm e}}}
\newcommand{{\He}}{{H^{\rm e}}}
\newcommand{{\Aop}}{{A^{\rm op}}}
\newcommand{{\Aope}}{({A^{\rm op}})^{\rm e}}
\newcommand{{\Aopl}}{{A^{\rm op}_\pl}}

\newcommand{{\Bop}}{{B^{\rm op}}}
\newcommand{{\Bope}}{({B^{\rm op}})^{\rm e}}
\newcommand{{\Bpl}}{{B_\pl}}

\newcommand{{\op}}{{{\rm op}}}
\newcommand{{\coop}}{{{\rm coop}}}
\newcommand{{\sop}}{{*^{\rm op}}}

\newcommand{\amoda}{A^{\rm e}\mbox{-}\mathbf{Mod}}                  %
\newcommand{\moda}{A^\mathrm{op}\mbox{-}\mathbf{Mod}}         %
\newcommand{\umod}{U\mbox{-}\mathbf{Mod}}                     
\newcommand{\modu}{U^\mathrm{op}\mbox{-}\mathbf{Mod}}         %
\newcommand{\yd}{{}^\uhhu_\uhhu\mathbf{YD}}                     
\newcommand{\ayd}{{}^\uhhu\mathbf{aYD}_\uhhu}

\newcommand{\ucomod}{U\mbox{-}\mathbf{Comod}}

\newcommand{\lact}{\smalltriangleright}                  
\newcommand{\ract}{\smalltriangleleft}
\newcommand{\blact}{\blacktriangleright}  
\newcommand{\bract}{\blacktriangleleft}

\newcommand{{\gog}}{{G \rightrightarrows G_0}}

\newcommand{{\rra}}{\rightrightarrows}

\newcommand{{\lra}}{\ \longrightarrow \ }
\newcommand{{\lla}}{\ \longleftarrow \ }
\newcommand{{\lma}}{\ \longmapsto \ }

\newcommand{{\bull}}{{\scriptscriptstyle{\bullet}}}
\newcommand{{\qqquad}}{{\quad\quad\quad}}
\newcommand{\Aopp}{{\scriptscriptstyle{\Aop}}}
\newcommand{\Aee}{{\scriptscriptstyle{\Ae}}}

\newcommand{\Hee}{{\scriptscriptstyle{\He}}}

\allowdisplaybreaks
\begin{document}

\title{Batalin-Vilkovisky algebra structures on 
${\rm (Co)}\Tor$ 
and Poisson bialgebroids}

\author{Niels Kowalzig}

\begin{abstract}
In this article, we extend our preceding studies on 
higher algebraic structures of (co)homology theories defined by
a left bialgebroid $(U,A)$. 
For a braided commutative Yetter-Drinfel'd algebra $N$, 
explicit expressions for the canonical Gerstenhaber algebra structure on
$\Ext_U(A,N)$ are given.
Similarly, if $(U,A)$ is a left Hopf algebroid where $A$ is an anti Yetter-Drinfel'd module over $U$, it is shown that 
the cochain complex computing
$\Cotor_U(A,N)$ defines a cyclic operad with multiplication and hence the groups $\Cotor_U(A,N)$
form a Batalin-Vilkovisky algebra. 
In the second part of this article, Poisson structures 
and the Poisson bicomplex for bialgebroids are introduced, 
which simultaneously generalise, for example, classical Poisson as well as cyclic homology. 
In case the bialgebroid $U$ is commutative, a Poisson structure on $U$ 
leads to a Batalin-Vilkovisky algebra structure on $\Tor_U(A,A)$. As an illustration, we show 
how this generalises the classical Koszul bracket on differential forms, and conclude by indicating how classical Lie-Rinehart bialgebras (or, geometrically, Lie bialgebroids) arise from left bialgebroids.
\end{abstract}

\address{Istituto Nazionale di Alta Matematica, P.le Aldo Moro 5, 00185 Roma, Italia}

\email{kowalzig@mat.uniroma2.it}

\keywords{Gerstenhaber algebra, Batalin-Vilkovisky algebra, operad, braided commutative Yetter-Drinfel'd algebra, noncommutative differential calculus, cyclic homology, Poisson homology, Hopf algebroid,  
Lie-Rinehart algebra, Koszul bracket}

\subjclass[2010]{16T05, 16E40; 16T15, 19D55, 58B34.
}

\maketitle

\tableofcontents

\section{Introduction}

\subsection{Gerstenhaber and Batalin-Vilkovisky algebras}
It is by now common knowledge that the cohomology or homology groups of a given mathematical object, although at first glance 
only a graded module over some base ring, often carry higher algebraic structures, such as products, brackets, and differentials. 
One of the earliest account for such structures on the cohomology groups of associative algebras is presumably 
Gerstenhaber's pioneer article \cite{Ger:TCSOAAR}, but in the meantime more general (co)homology theories such as for Lie-Rinehart algebras (Lie algebroids) \cite{Hue:LRAGAABVA,Hue:DBVAAFTLRA,Kos:EGAALBA,Xu:GAABVAIPG} or for Hopf algebras \cite{FarSol:GSOTCOHA, Kad:OTCCOAB, Men:BVAACCOHA, Men:CMCMIALAM, Tai:IHBCOIDHAAGCOTYP} as well as Hopf algebroids \cite{KowKra:BVSOEAT} have been investigated in this direction. What is more, some of these structures already appear on the (co)chain level \cite{TamTsy:NCDCHBVAAFC}, and
in an even more 
abstract spirit, analogous structures (up to homotopy) have been found for the cochain spaces or the cohomology of (certain) operads \cite{BraLaz:HBVAIPG, GalTonVal:HBVA, GerSch:ABQGAAD, GetJon:OHAAIIFDLS, GerVor:HGAAMSO,  MarShnSta:OIATAP, McCSmi:ASODHCC, Men:BVAACCOHA}. 

\begin{definition}
\label{golfoaranci1}
Let $k$ be a commutative ring.
\begin{enumerate}
\compactlist{99}
\item
A {\em Gerstenhaber algebra} 
over $k$ is a graded commutative $k$-algebra
$(V,\smallsmile)$  
$$
        V=\bigoplus_{p \in \mathbb{N}} V^p,\quad
        \ga \smallsmile \gb=(-1)^{pq}\gb \smallsmile \ga
        \in V^{p+q},\quad 
        \ga \in V^p,\gb \in V^q,
$$ 
with a graded Lie bracket 
$
        \{\cdot,\cdot\} : V^{p+1} \otimes_k V^{q+1} \rightarrow V^{p+q+1}
$ 
on the \emph{desuspension} 
$$
        V[1]:=\bigoplus_{p \in \mathbb{N}} V^{p+1}
$$
of $V$, 
for which all operators $\{\gamma,\cdot\}$ satisfy the graded Leibniz rule
$$
        \{\gamma,\ga \smallsmile \gb\}=
        \{\gamma,\ga\} \smallsmile \gb + (-1)^{pq} \ga \smallsmile
        \{\gamma,\gb\},\quad
        \gamma \in V^{p+1},\ga \in V^q.
$$ 
\item
A {\em Batalin-Vilkovisky} algebra is a Gerstenhaber algebra $V$ with a $k$-linear differential
$$
\BB: V^n \to V^{n-1}, \quad \BB\BB = 0
$$
 of degree $-1$
such that for all $\ga \in V^p$, $\gb \in V$
$$
\{\ga, \gb\} = (-1)^{p}\big(\BB(\ga \smallsmile \gb) - \BB(\ga) \smallsmile \gb - (-1)^p \ga \smallsmile \BB(\gb) \big).
$$
A Batalin-Vilkovisky algebra is also called an {\em exact} Gerstenhaber algebra and the differential $\BB$ is said to {\em generate} the Gerstenhaber bracket.  
\end{enumerate}
\end{definition}

Since we shall continuously deal with the desuspension mentioned above, it is convenient to introduce the notation
$$
|n| := n-1, \quad n \in \Z.
$$

On the other hand, as we will see in \S\ref{salice}, in some cases 
Gerstenhaber algebras come with a differential that, in contrast to the generating operator of a Batalin-Vilkovisky algebra, increases the degree:

\begin{definition}
\label{caffetommaseo}
A {\em differential Gerstenhaber algebra} is a Gerstenhaber algebra $V$ with a $k$-linear differential 
$$
\gd: V^n \to V^{n+1}, \quad \gd \gd = 0
$$
of degree $+1$ such that $\gd$ is a graded derivation of the cup product, {\em i.e.},
such that 
$$
\gd(\ga \smallsmile \gb) = \gd \ga \smallsmile \gb + (-1)^p \ga \smallsmile \gd \gb, \qquad  \ga \in V^p, \gb \in V,
$$
holds. It is called {\em strong differential} if $\gd$ is, additionally, a graded derivation of the Gerstenhaber bracket, that is, if
$$
\gd \{\ga, \gb\} = \{\gd \ga, \gb \} + (-1)^{|p|} \{ \ga, \gd \gb \}, \qquad  \ga \in V^p, \gb \in V,
$$
holds true.
\end{definition}

\subsection{Aims and objectives}
The principal aim of this paper is to investigate under what conditions the (co)homology groups
$$
\Ext_U(A,M), \ \Cotor_U(A,M), \ \mbox{and} \ \Tor^U(A,A)
$$
admit a Gerstenhaber resp.\ Batalin-Vilkovisky algebra structure, where $U$ is a left bialgebroid (a $\times_\ahha$-bialgebra) or a left Hopf algebroid (a $\times_\ahha$-Hopf algebra) over a possibly noncommutative $k$-algebra $A$. In \S\ref{prelim} we indicate the necessary details for this sort of generalisation of a $k$-bialgebra resp.\ Hopf algebra to noncommutative base rings.
Here, we only seize the occasion once again to point out 
that the rings governing most parts of classical homological algebra can all be described by such a structure.
As a consequence, our results apply to, for example, 
Hochschild and Lie-Rinehart (in particular Lie algebra, de Rham, Lie algebroid and Poisson) (co)homology, {\em i.e.}, give access to both algebra and geometry, but also to that of any Hopf algebra (which leads to, {\em e.g.}, group (co)homology) as well as to (\'etale) groupoid homology.

\subsection{Yetter-Drinfel'd algebras as coefficient modules for Gerstenhaber algebras} 
The aim of \S\ref{caprarola} is to give explicit expressions 
of the canonical Gerstenhaber algebra structures on (simplicial) cohomology and (coring) cohomology associated to a left bialgebroid $U$ and taking values in general coefficient modules: 
the left bialgebroid structure of $U$ leads not only to a monoidal structure on the categories $\umod$ and $\ucomod$ of left $U$-modules resp.\ left $U$-comodules, but also to one on $\yd$, the category of Yetter-Drinfel'd modules. 
Considering monoids in this latter category and with the help of a well-known result 
about Gerstenhaber structures in relation to the cohomology of operads with multiplication \cite{GerSch:ABQGAAD, McCSmi:ASODHCC, Men:BVAACCOHA}, we can prove:

\begin{theorem}
\label{main1}
If $N$ is a braided commutative Yetter-Drinfel'd algebra over a left bialgebroid $U$,
then 
$$
C^\bull(U,N) := \Hom_\Aopp\big((U^{\otimes_\Aopp \bull})_\ract, N  \big)
$$ 
defines an operad with multiplication.
Hence, 
$
H^\bull(U,N) := H(C^\bull(U,N), \gd)
$ 
carries the structure of a Gerstenhaber algebra.
\end{theorem} 
Here, $\gd: C^\bull(U,N) \to C^{\bull+1}(U,N)$ defines the canonical cochain complex that arises from the bar resolution of $A$. We refer to the main text for all details and in particular all notation used throughout this introductory section.
The theorem implies in particular that if 
$\due U \blact {}$ is projective as a left $A$-module, then $\Ext^\bull_U(A,N)$ is a Gerstenhaber algebra, and generalises not only relatively recent results \cite{Tai:IHBCOIDHAAGCOTYP,Men:CMCMIALAM} in bialgebra theory (where $A:=k$ is a commutative ring that is central in $U$) 
but also in bialgebroid theory \cite{KowKra:BVSOEAT} by introducing general coefficients.

Another cohomology theory attached to any bialgebroid is obtained by considering the cobar resolution of $A$, {\em i.e.}, by dealing with the coring cohomology. This leads to a cochain complex
$$
\gb: \corm \bull := (\due U \lact \ract)^{\otimes_\ahha  \bull} \otimes_\ahha N \to \corm {\bull+1},
$$ 
which again admits the same sort of higher algebraic structure. In \S\ref{acquedottofelice} we prove:

\begin{theorem}
\label{main2}
Let $N$ be a braided commutative Yetter-Drinfel'd algebra over a left bialgebroid $U$.
Then 
$
\corm \bull
$ 
defines an operad with multiplication.
Hence, the cohomology groups $H^\bull_{\rm co}(U,N) :=  H(\corm \bull , \gb)$ carry the structure of a Gerstenhaber algebra.
\end{theorem} 
In particular, if
$U_\ract$ is flat as a right $A$-module, then $\Cotor^\bull_U(A,N)$ is a Gerstenhaber algebra. 
Also this theorem is an extension of the bialgebra case 
known before (implicitly in \cite{GerSch:ABQGAAD} and rediscovered more recently in \cite{Kad:OTCCOAB}, see also \cite{Men:BVAACCOHA}) to bialgebroids ({\em i.e.}, to noncommutative base rings) and nontrivial coefficients.

For both cochain complexes mentioned above, we will give explicit expressions in \S\ref{responsabilitacivile} for the graded commutative product $\smallsmile$ and the bracket $\{\cdot,\cdot\}$ that belong to any Gerstenhaber algebra by defining (in the spirit of \cite{Ger:TCSOAAR})  {\em Gerstenhaber products} $\circ_i$ on $C^\bull(U,N)$ resp.\ $\corm \bull$.

\subsection{The Batalin-Vilkovisky algebra $\Cotor$}

If $U$ is not only a left bialgebroid but rather a left Hopf algebroid (a $\times_\ahha$-Hopf algebra) and if on top of that 
the base algebra $A$ carries a right $U$-action (which is both fulfilled if $U$ is, for example, a {\em full} Hopf algebroid), 
this equips the cochain spaces $C^\bull_{{\rm co}}(U,N)$ with the structure of a {\em cyclic operad with multiplication}. 
Since by Menichi's theorem \cite{Men:BVAACCOHA} any such cyclic operad with multiplication defines the structure of a Batalin-Vilkovisky algebra on the associated cohomology, we can prove:

\begin{theorem}
Let $N$ be a braided commutative Yetter-Drinfel'd algebra over a left Hopf algebroid $U$ and
assume that $A$ is an anti Yetter-Drinfel'd module. 
Then one can define a right $U$-action on $N$ such that $N$ together with its left $U$-comodule structure becomes an {\em anti} Yetter-Drinfel'd module as well, 
and if $N$ is moreover stable with respect to this action, 
then $\corm \bull$ defines a cyclic operad with multiplication. 
Hence, the cohomology groups $H^\bull_{\rm co}(U,N)$ carry the structure of a Batalin-Vilkovisky algebra.
\end{theorem} 

In particular, if  
$U_\ract$ is flat as a right $A$-module, then $\Cotor^\bull_U(A,N)$ is a Batalin-Vilkovisky algebra. This, once more, extends 
a result known in Hopf algebra theory \cite{Men:BVAACCOHA} not only to bialgebroids, but also to nontrivial coefficients. Moreover, we
confirm the conjecture in \cite[\S10]{Men:CMCMIALAM} that 
in case of a Hopf algebra over $k$ endowed with a modular pair $(\gd, \gs)$ in involution, one apparently 
{\em cannot} 
take $\qttr k \gs {}{} \gd$ unless the grouplike element $\gs$ is the unit element in the Hopf algebra, see \S\ref{caprarola} for details.

\subsection{Poisson structures for bialgebroids}
As is shown in \S\ref{ronciglione},
the definition of a {\em Poisson structure} or {\em (quasi-)triangular $r$-matrix} for a left bialgebroid $U$, that is, 
a $2$-cocycle $\gt \in C^2(U,A)$ that fulfills 
\begin{equation*}
\gt \circ_1 \gt = \gt \circ_2 \gt
\end{equation*}
generalises not only Poisson structures for (noncommutative) associative algebras and triangular $r$-matrices for Lie bialgebroids 
(and hence Poisson manifolds as well as skew-symmetric solutions of the classical Yang-Baxter equation in Lie bialgebra theory \cite{Dri:HSOLGLBATGMOCYBE}) but also the ring structure in an associative algebra or,
more general, the notion of operad multiplication for the operad $C^\bull(U,A)$ as given in \S\ref{palermo}.

We then define the differentials
\begin{eqnarray*}
b^\gt: C_n(U,M) \to  C_{n-1}(U,M), & (m,x) & \mapsto - \cL_\gt(m,x), \\
\gb^\gt: C^n(U,A) \to  C^{n+1}(U,A), & \gvf & \mapsto \{\gt, \gvf\}, 
\end{eqnarray*}
where $\cL_\gt$ is the generalised Lie derivative on left Hopf algebroids in \rmref{messagedenoelauxenfantsdefrance2} that was introduced in \cite{KowKra:BVSOEAT}. 
The triple $\big(C_\bull(U,M), b^\gt, B\big)$ can be shown to form a mixed complex, which allows for the definition of {\em cyclic Poisson homology}. 
This approach conceptually unites, for example, Hochschild with Poisson homology (resp.\ cyclic homology with cyclic Poisson homology), see \S\ref{fave}.

\subsection{The Batalin-Vilkovisky algebra $\Tor$}

In case $U$ is a {\em commutative} left Hopf algebroid, the shuffle product $\cdot \times \cdot$ defines the structure of a graded commutative algebra on the homology groups $H_\bull(U,A)$. In case $U$ is Poisson, this structure can be extended to that of a Batalin-Vilkovisky algebra. In \S\ref{castropretorio} we will prove:

\begin{theorem}
Let $U$ be a commutative Poisson left Hopf algebroid with triangular $r$-matrix $\gt$. 
Then there is a $k$-bilinear map
\begin{equation*}
\begin{split}
\{.,.\}_\gt: \ & C_p(U,A) \otimes C_q(U,A) \to C_{p+q-1}(U,A), \qquad p, q \geq 0,\\
& x \otimes y \mapsto (-1)^{|p|} b^\gt(x \times y) + (-1)^{p} b^\gt x \times y + x \times b^\gt y,  
\end{split}
\end{equation*}
which 
induces a Batalin-Vilkovisky algebra structure on $H_\bull(U,A)$.
\end{theorem}

Again, if $\due U \blact {}$ is projective as left $A$-module, this yields a bracket 
$$
\{.,.\}_\gt: \Tor^U_p(A,A) \otimes \Tor^U_q(A,A) \to \Tor^U_{p+q-1}(A,A).
$$
In our examples in \S\ref{pirano}, 
we illustrate how the Batalin-Vilkovisky structure on $\Tor^\bull_\uhhu(A,A)$ for a commutative Poisson left Hopf algebroid $U$ generalises the classical Koszul bracket on forms, {\em i.e.}, the Batalin-Vilkovisky algebra structure on the exterior algebra $\bigwedge_\ahha^\bull \! L^*$ of the dual of a Lie-Rinehart algebra $(A,L)$. 

We conclude by dealing with the case of how the idea of Lie-Rinehart bialgebras (Lie bialgebroids) induced by a Poisson bivector ({\em i.e.}, a triangular $r$-matrix)
transfers in complete analogy to $\Ext_U$ and $\Tor^U$: whereas a triangular Lie bialgebroid (in the sense of \cite{MacXu:LBAPG}) gives rise \cite{Kos:EGAALBA, Xu:GAABVAIPG} to a pair of strong differential Gerstenhaber algebras in duality  --- one of which is Batalin-Vilkovisky ---, 
in case of a commutative Poisson left Hopf algebroid 
both $H^\bull(U,A)$ and $H_\bull(U,A)$ are strong differential Gerstenhaber algebras as well, the latter being again 
Batalin-Vilkovisky.
\
\\
\
\\

\thanks{ {\bf Acknowledgements.}   \!
It is a pleasure to thank Fabio Gavarini and Ulrich Kr\"ahmer for inspiring 
discussions and helpful comments. 

This research was funded by an INdAM-COFUND Marie Curie grant.

\section{Preliminaries}
\label{prelim}
In this section we not only recall preliminaries on bialgebroids, Hopf
algebroids, (anti) Yetter-Drinfel'd modules and algebras, and (co)cyclic modules for bialgebroids --- mainly from our papers \cite{KowKra:CSIACT, KowKra:BVSOEAT} ---, 
but simultaneously introduce the notation and conventions used throughout the text. 
See \cite{Boe:HA} for more detailed information on bialgebroids and 
Hopf algebroids, and references to the original 
sources. 

\subsection{Bialgebroids}
\label{durchgegangen}
Throughout this paper, $A$ and $U$ are 
(unital associative) $k$-algebras, where $k$ is a commutative ground ring (sometimes of characteristic zero). 
As common, whenever an unadorned tensor product appears, it is meant to be over $k$.
Furthermore, we assume that there be 
a fixed $k$-algebra map
\begin{equation*}
\label{eta}
        \eta : \Ae := A \otimes_k \Aop \rightarrow 
        U.
\end{equation*}
This induces forgetful functors 
$$
\umod \rightarrow \amoda,\quad
\modu \rightarrow \amoda
$$ 
that turn left $U$-modules $N$ respectively right $U$-modules $M$ 
into $A$-bimodules with actions
\begin{equation}
\label{pergolesi}
        a \lact n \ract b:=\eta(a \otimes_k b)n,\quad
        a \blact m \bract b:=m \eta(b \otimes_k a),\quad
        a,b \in A,n \in N,m \in M.
\end{equation}
In particular, left and right multiplication in $U$ 
define $A$-bimodule structures of both these types on 
$U$ itself. 

Generalising the standard result for bialgebras (which 
is the case $A=k$), Schauenburg has proven 
\cite{Schau:DADOQGHA} that
the monoidal structures on $\umod$ for which the forgetful 
functor to $\amoda$ is strictly
monoidal (where $\amoda$ is monoidal via $\otimes_\ahha$) correspond to
what is known as \emph{(left) bialgebroid} (or $\times_\ahha$-bialgebra) structures
on $U$. We refer, {\em e.g.}, to our earlier paper 
\cite{KowKra:DAPIACT} for a detailed definition 
(which is due to Takeuchi \cite{Tak:GOAOAA}). Let us only
recall that a bialgebroid has a coproduct and a counit 
\begin{equation}
\label{cerveza}
        \Delta : U \rightarrow U_\ract \otimes_\ahha \due U \lact {}, \quad
        \varepsilon : U \rightarrow A,
\end{equation}
which turn $U$ into a coalgebra in $\amoda$.
Unlike for $A=k$, the counit $ \varepsilon $ is not necessarily a
ring homomorphism but only yields a left $U$-module structure on 
$A$ with action of $u \in U$ 
on $a \in A$ given by 
\begin{equation}
\label{stabatmater1}
u a := \varepsilon (u \bract a).
\end{equation}
Furthermore, 
$\Delta$ is required to corestrict to a map from $U$ to the
Sweedler-Takeuchi product $U
\times_\ahha U$, which is the $\Ae$-submodule of  $U \otimes_\ahha U$
whose elements $\sum_i u_i \otimes_\ahha v_i$ fulfil 
\begin{equation}
\label{tellmemore}
 \textstyle
\sum_i a \blact u_i \otimes_\ahha v_i
		  = \sum_i u_i \otimes_\ahha v_i \bract a, \
		  \forall a \in A.
\end{equation}
In the sequel, we will freely use Sweedler's notation  
$ \Delta (u)=:u_{(1)} \otimes_\ahha u_{(2)}$.

In the same paper \cite{Schau:DADOQGHA}, 
Schauenburg generalised
the notion of a Hopf algebra to the bialgebroid setting by introducing 
$\times_\ahha$-Hopf algebras which we usually refer to as \emph{left Hopf algebroids}. 
The crucial piece of additional structure on top of the bialgebroid one 
is the \emph{translation map}
\begin{equation}
\label{pm}
		  U \rightarrow 
		  {}_\blact U \otimes_\Aopp U_\ract, \quad u \mapsto u_+ \otimes_\Aopp u_-,
\end{equation}
where the right hand side is to be understood as a Sweedler-type notation, {\em i.e.}, indicating a sum. We will make permanent use of the following technical identities that hold for the map \rmref{pm}, 
see \cite[Proposition~3.7]{Schau:DADOQGHA}: 
\begin{proposition}
Let $U$ be a left Hopf algebroid over $A$. For all $u, v \in U$, $a,b \in A$ one has 
\begin{eqnarray}
\label{Sch1}
u_+ \otimes_\Aopp  u_- & \in 
& U \times_\Aopp U, \\
\label{Sch2}
u_{+(1)} \otimes_\ahha u_{+(2)} u_- &=& u \otimes_\ahha 1 \in U_\ract \otimes_\ahha {}_\lact U, \\
\label{Sch3}
u_{(1)+} \otimes_\Aopp u_{(1)-} u_{(2)}  &=& u \otimes_\Aopp  1 \in  {}_\blact U
\otimes_\Aopp  U_\ract, \\ 
\label{Sch4}
u_{+(1)} \otimes_\ahha u_{+(2)} \otimes_\Aopp  u_{-} &=& u_{(1)} \otimes_\ahha
u_{(2)+} \otimes_\Aopp  u_{(2)-},\\
\label{Sch5}
u_+ \otimes_\Aopp  u_{-(1)} \otimes_\ahha u_{-(2)} &=& 
u_{++} \otimes_\Aopp
u_- \otimes_\ahha u_{+-}, \\
\label{Sch6}
(uv)_+ \otimes_\Aopp  (uv)_- &=& u_+v_+
\otimes_\Aopp  v_-u_-, 
\\ 
\label{Sch7}
u_+u_- &=& s (\varepsilon (u)), \\
\label{Sch8}
\varepsilon(u_-) \blact u_+  &=& u, \\
\label{Sch9}
(s (a) t (b))_+ \otimes_\Aopp  (s (a) t (b) )_- 
&=& s (a) \otimes_\Aopp  s (b), 
\end{eqnarray}
where in \rmref{Sch3} we mean the Sweedler-Takeuchi product
\begin{equation*}
\label{petrarca}
		  U \times_\Aopp  U:=
		  \left\{\textstyle\sum_i u_i \otimes_\Aopp  v_i \in 
		  {}_\blact U \otimes_\Aopp  U_\ract\mid
		  \sum_i u_i \ract a \otimes_\Aopp  v_i=
		  \sum_i u_i \otimes_\Aopp  a \blact
		  v_i
		  \right\},
\end{equation*}
which is an algebra by factorwise multiplication, but with opposite 
multiplication on the second factor, and where in 
(\ref{Sch7}) and (\ref{Sch9}) we use the \emph{source} and \emph{target} maps 
\begin{equation}
\label{basmati}
        s,t : A \rightarrow U,\quad s(a):=\eta (a \otimes_k 1),\quad
        t(b):=\eta (1 \otimes_k b). 
\end{equation}
\end{proposition}

Beyond the obvious example of a left Hopf algebroid given by a Hopf algebra $H$ with antipode $S$, 
where the translation map is given by 
$$
h_+ \otimes h_- := h_{(1)} \otimes S(h_{(2)}), \qquad h \in H,
$$
we recall below three (by now) standard examples of left Hopf algebroids 
since they will be used as test cases throughout the text: the first one 
gives access to the Hochschild theory of associative algebras,
the second two to, {\em e.g.}, multivector fields and differential forms in differential geometry:

\begin{example}
\label{todi}
Recall from \cite{Schau:DADOQGHA} that 
$\Ae := A \otimes_k \Aop$ is for any $k$-algebra $A$ 
a left Hopf algebroid over 
$A$ with structure maps
$$
		  s(a) := a \otimes_k 1,\quad
		   t(b) := 1 \otimes_k b,\quad 
		  \gD(a \otimes_k b) := (a
		  \otimes_k 1) 
		  \otimes_\ahha (1 \otimes_k b),\quad
		  \gve(a \otimes_k b) := ab.
$$
The translation map is given by
$$
		  (a \otimes_k b)_+ \otimes_\Aopp 
		  (a \otimes_k b)_- := (a \otimes_k 1)
		  \otimes_\Aopp (b \otimes_k 1).
$$
\end{example}

\begin{example}
\label{castelliromani}
Let $(A,L)$ be a Lie-Rinehart algebra (geometrically, a Lie algebroid) over a commutative $k$-algebra $A$ and $V\!L$ be its universal enveloping algebra 
(see \cite{Rin:DFOGCA}). The left Hopf algebroid structure of $V\!L$ has been given in \cite{KowKra:DAPIACT}; 
as therein, we denote by the same symbols elements $a \in A$ and $X \in L$ and the corresponding generators in $V\!L$.
The maps $s = t$ are equal to the canonical injection $A \to V\!L$. On generators, the coproduct and the counit are given by 
\begin{equation}
\label{rashomon1}
\begin{array}{rclrcl}
\gD(X) &:=& X \otimes_\ahha 1 + 1 \otimes_\ahha X, & \qquad \gve(X) &:=& 0, \\
\gD(a) &:=& a \otimes_\ahha 1, &                 \qquad \gve(a) &:=& a,
\end{array}
\end{equation}
whereas the translation map is defined by
\begin{equation}
\label{rashomon2}
X_+ \otimes_\Aopp X_- := X \otimes_\Aopp 1 - 1 \otimes_\Aopp X, \qquad a_+ \otimes_\Aopp a_- := a \otimes_\ahha 1.
\end{equation}
By universality, these maps defined on generators can be extended to $V\!L$. If $(A,L) = (\cinf M, \gG^\infty(E))$ 
arises from a Lie algebroid $E \to \cM$ over a smooth manifold $\cM$, one can consider $V\!L$ as the space of $E$-differential operators on $\cM$ (see, for example, \cite{CanWei:GMFNCA}).
\end{example}

\begin{example}
\label{cambridge}
The third example (see \cite{KowPos:TCTOHA, NesTsy:DOSLADOHSSAIT, CalVdB:HCAAC}) is in some sense dual to the preceding one: let again $(A, L)$ be a Lie-Rinehart algebra and define the $A$-linear dual $J\!L := \Hom_\ahha(V\!L,A)$, 
the {\em jet space} of $(A,L)$. 
By duality, $J\!L$  
carries a commutative $\Ae$-algebra structure with product
\begin{equation}
\label{kaesekuchen1}
        (fg) (u) = 
        f(u_{(1)}) g(u_{(2)}), \qquad 
        f, g \in J\!L, \ u \in V\!L,
\end{equation}
the unit given by the counit $\gve$ of $V\!L$, 
and source and target maps given by 
\begin{equation}
\label{sarare}
        s(a)(u) := a \gve(u) = \varepsilon (au), \qquad 
        t(a)(u) := \gve(ua), \qquad 
        a \in A, u \in V\!L.
\end{equation}
The $\Ae$-ring $J\!L$ is complete with respect to 
the (topology defined by the) decreasing 
filtration whose degree $p$ part consists of those 
functionals that vanish on the $A$-linear span $(V\!L)_{\le p} \,{\subseteq}\,
V\!L$ of all 
monomials in up to $p$ elements of $L$. In case $L$ is finitely
generated projective over $A$,
Rinehart's generalised PBW theorem \cite[Thm.~3.1]{Rin:DFOGCA}
identifies $J\!L$ with the completed symmetric algebra 
of the $A$-module $L^* := \Hom_\ahha(L,A)$. Moreover, the filtration of $J\!L$ induces one
of $J\!L \otimes_\ahha J\!L$; if we denote by $J\!L
\hat\otimes_\ahha J\!L$ 
the completion, the product of $V\!L$ yields, as in \cite[Lem.~3.16]{KowPos:TCTOHA}, a coproduct
$\gD: J\!L \to J\!L \hat\otimes_\ahha J\!L$ by 
\begin{equation}
\label{kaesekuchen2}
        f(uv) =: \gD(f)(u \otimes_\Aopp v) = f_{(1)}(u f_{(2)}(v)).
\end{equation}
Along with this coproduct comes the counit of $J\!L$ given by evaluation on the unit element, that is,  
$f \mapsto f(1_{\scriptscriptstyle{V\!L}})$. 
These maps are part of a 
{\em complete (left) Hopf algebroid} structure on $J\!L$, see \cite[Appendix A]{Qui:RHT} for
complete Hopf algebras, its Hopf algebroid generalisation being
straightforward. 
Finally, extending the {\em Grothendieck connection} from $L$ to $V\!L$ defines a map
\begin{equation}
\label{kaesekuchen3}
        (Sf)(u) := \varepsilon (u_+ f(u_-)), \qquad u \in V\!L, f \in J\!L.
\end{equation}
With this map that may be called, as the notation suggests, the {\em antipode} of $J\!L$, the jet space 
is not only a left but a 
{\em full} complete Hopf algebroid 
in the sense of B\"ohm and
Szlach\'anyi \cite{Boe:HA}. A short computation yields that the antipode is an involution,
$
S^2 = \mathrm{id},
$
and the translation map
\rmref{pm} results as   
\begin{equation}
\label{apfelorangeingwersaft}
f_+ \hat\otimes_\Aopp f_- := f_{(1)} \hat\otimes_\Aopp S(f_{(2)}), 
\end{equation}
formally similar to the case of Hopf algebras.
Again, if $(A,L) = (\cinf \cM, \gG^\infty(E))$ 
arises from a Lie algebroid $E \to \cM$ over a smooth manifold $\cM$, 
the jet space $J\!L$ can be considered as the space of $E$-differential forms on the manifold $\cM$.
\end{example}

\subsection{Comodules, (co)module algebras, and (anti) Yetter-Drinfel'd modules}
\subsubsection{Comodules over left bialgebroids}
Recall, {\em e.g.}, from \cite{Boe:HA} that a left comodule for a left bialgebroid $U$ is a left comodule of the coring underlying $U$, {\em i.e.}, a left $A$-module $M$ and a left $A$-module map
$$
     \Delta_\emme:   M \rightarrow 
        U_\ract \otimes_\ahha {}_\blact  M, \quad
        m \mapsto m_{(-1)} \otimes_\ahha m_{(0)},
$$  
satisfying the usual coassociativity and counitality axioms. We denote the category of left $U$-comodules by $\ucomod$. 
On any $M \in \ucomod$ there is an induced {\em right} $A$-action given by 
$
ma := \gve(a \blact m_{(-1)})m_{(0)},
$
and $\Delta_\emme$ is then an $\Ae$-module morphism 
$ 
M \rightarrow 
        U_\ract \times_\ahha {}_\blact  M,
$
where $  U_\ract \times_\ahha {}_\blact  M$ is the $\Ae$-submodule of
$U_\ract \otimes_\ahha {}_\blact  M$ whose elements $\sum_i u_i
\otimes_\ahha m_i$ fulfil 
\begin{equation}
\label{auchnochnicht}
\textstyle
\sum_i a \blact u_i \otimes_\ahha m_i
		  = \sum_i u_i \otimes_\ahha m_i a, \
		  \forall a \in A.
\end{equation}
In particular, the $\Ae$-linearity reads
\begin{eqnarray}
\label{maotsetung}
		  \gD_\emme(amb) &\!\!\!\!=&\!\!\!\! 
		  a \lact  m_{(-1)} \bract b \otimes_\ahha
		  m_{(0)}, \quad \forall m \in M, \ a, b \in A.
\end{eqnarray}

For later use, let us mention that the base algebra $A$ itself is a left $U$-comodule with canonical coaction
\begin{equation}
\label{stabatmater2}
\gD_\ahha: A \to U \otimes_\ahha A \simeq U, \quad a \mapsto s(a),
\end{equation}
where $s$ is the source map from \rmref{basmati}.

\subsubsection{(Co)module algebras}
In order to introduce coefficients for the subsequent Gerstenhaber algebras, we will additionally need the subsequent concepts.
Similarly as for bialgebras, 
there exist the notions of monoid in the categories $\umod$ resp.\ $\ucomod$ of left modules resp.\ left comodules over a left bialgebroid $U$, with some particular attention to be paid to the underlying $A$-bimodule structures:

\begin{definition}[\cite{KadSzl:BAODTEAD}]
\label{LMR}
A {\em left $U$-module algebra} $M$
is a monoid in $\umod$. That is, $M \in \umod$ carries a canonical $A$-ring
structure with $A$-balanced multiplication 
$m \otimes_\ahha m' \mapsto m \cdot_\emme m'$ for $m, m' \in M$, 
and unit map $A \to M, \ a \mapsto a \lact 1_\emme
= 1_\emme \ract a$ such that for $u \in U, \ m, m' \in M$
\begin{equation}
\label{zilvesta}
u(m\cdot_\emme m') = (u_{(1)}m)\cdot_\emme (u_{(2)}m') \quad \mbox{and} \quad
u1_\emme = \gve(u) \lact 1_\emme = 1_\emme \ract \gve(u)
\end{equation}
holds.
\end{definition}

For example, the base algebra $A$ is a left $U$-module algebra with $U$-action given via the counit as in \rmref{stabatmater1}, but
$U$ itself usually is not. We remark that in case $U=\Ae$ an $\Ae$-module algebra is also called an {\em $A$-ring} or an {\em $A$-algebra}.

Observe in particular that with the induced $\Ae$-module
structure on $M$ given as in \rmref{pergolesi} and the $\Ae$-linearity of the coproduct, one has
\begin{equation}
\begin{array}{c}
\label{bilet}
a \lact (m \cdot_\emme m') = (a \lact m) \cdot_\emme m', 
\\
(m \cdot_\emme m') \ract a = m \cdot_\emme (m' \ract a), 
\end{array}
\end{equation}
and moreover
\begin{equation}
\label{normalny}
m \cdot_\emme (a \lact m') = (m \ract a) \cdot_\emme m'.
\end{equation}

Dually, we shall need the notion of a monoid in $\ucomod$, see, {\em e.g.}, \cite{BoeSte:CCOBAAVC}:
\begin{definition}
\label{LMC}
A {\em left $U$-comodule algebra} $N$
is a monoid in $\ucomod$. That is, $N \in \ucomod$ 
with coaction $\gD_\enne: n \mapsto n_{(-1)} \otimes_\ahha n_{(0)}$ 
moreover carries a canonical $A$-ring 
structure with $A$-balanced multiplication $n \otimes_A n' \mapsto n \cdot_\enne n'$ for $n, n' \in N$ 
such that
\begin{equation}
\label{corsasemplice}
\gD_\enne(1_\enne) = 1_\uhhu \otimes_\ahha 1_\enne, \quad \gD_\enne(n \cdot_\enne n') = n_{(-1)}n'_{(-1)} \otimes_\ahha n_{(0)} \cdot_\enne n'_{(0)}.
\end{equation}
\end{definition}
For example, both the base algebra $A$ with left 
$U$-coaction $a \mapsto s(a)$ as well as the ring $U$ itself by means of the coproduct $\gD$ are $U$-comodule algebras.
Observe that for a left $U$-comodule algebra relations with respect to the underlying left $\Ae$-module structure identical to those in \rmref{bilet} follow from \rmref{maotsetung}. 

Of course, one can also define comonoids in $\umod$ and $\ucomod$, respectively, but they are not needed in the sequel.

\subsubsection{(Anti) Yetter-Drinfel'd modules and Yetter-Drinfel'd algebras}

For a left bialgebroid $U$ there exists the notion of Yetter-Drinfel'd module (or {\em crossed bimodule}), {\em i.e.}, a module which is simultaneously a left comodule with a certain compatibility between action and coaction. For bialgebras, this concept goes back to \cite{Yet:QGAROMC}, 
whereas the bialgebroid version is due to \cite{Schau:DADOQGHA}:

\begin{definition}
A left $U$-module $N$ which is simultaneously a left comodule over a left bialgebroid $U$
is called a \emph{Yetter-Drinfel'd module (YD)} 
if the full $\Ae$-module structure 
${}_\lact N_\ract$ of the module coincides with that underlying 
the comodule, and if one has
\begin{equation}
\label{huhomezone1a}
		 (u_{(1)}n)_{(-1)} u_{(2)} \otimes_\ahha (u_{(1)}n)_{(0)}= 
		  u_{(1)} n_{(-1)} \otimes_\ahha u_{(2)} n_{(0)}.
\end{equation}
\end{definition}

The category $\yd$ of Yetter-Drinfel'd modules is monoidal with respect to the tensor product in \rmref{cerveza}, equipped with the diagonal module 
and the codiagonal comodule structure. We only mention as a side remark \cite[Prop.~4.4]{Schau:DADOQGHA} that $\yd$ is equivalent to the weak centre of the category $\umod$ with (pre)braiding
$$
\gs_{\enne, \enne'}: N_\ract \otimes_\ahha \due {N'} \lact {} \to {N'}_\ract \otimes_\ahha \due N \lact {}, 
\quad n \otimes_\ahha n' \mapsto n_{(-1)}n' \otimes_\ahha n_{(0)}. 
$$
The importance of the existence of this braiding for our later constructions of Gerstenhaber algebras is provided by the following notion (see, for example, \cite{BrzMil:BBAD}):

\begin{definition}
\label{fiocchi}
A {\em Yetter-Drinfel'd algebra} is a monoid in $\yd$, {\em i.e.}, an $A$-ring $N$ which is both a left $U$-module algebra and a left $U$-comodule algebra plus the compatibility condition \rmref{huhomezone1a} between action and coaction. 
A Yetter-Drinfel'd algebra is said to be {\em braided commutative} if the multiplication in $N$ is commutative with respect to $\gs_{\enne,\enne}$, that is,
\begin{equation}
\label{circolodegliartisti}
n \cdot_\enne n' = (n_{(-1)}n') \cdot_\enne n_{(0)}, \quad \mbox{for all} \ n, n' \in N.
\end{equation}
\end{definition}

Observe that \rmref{circolodegliartisti} is well defined by \rmref{normalny} as well as \rmref{maotsetung}.
Needless to say that the notion of braided commutativity is entirely independent of whether $N$ itself as an algebra is commutative: 
for example, the base algebra $A$ of any bialgebroid $U$ is always a braided commutative Yetter-Drinfel'd algebra by means of the canonical left $U$-action \rmref{stabatmater1} and left $U$-coaction \rmref{stabatmater2}.

If $U$ happens to be a {\em left Hopf algebroid}, one can give a sort of opposite notion of anti Yetter-Drinfel'd modules.
The following particular class of right modules which are also left comodules 
was introduced in \cite{HajKhaRanSom:SAYDM, JarSte:HCHARCHOHGE} for Hopf algebras and 
in \cite{BoeSte:CCOBAAVC} 
for left Hopf algebroids:
\begin{definition}
A right $U$-module left $U$-comodule $M$ over a left Hopf algebroid $U$
is called \emph{anti Yetter-Drinfel'd module (aYD)} 
if the full $\Ae$-module structure 
${}_\blact M_\bract$ of the module coincides with that underlying 
the comodule, and if one has
\begin{equation}
\label{huhomezone2}
		 (mu)_{(-1)} \otimes_\ahha (mu)_{(0)}= 
		  u_- m_{(-1)} u_{+(1)} \otimes_\ahha m_{(0)} u_{+(2)}
\end{equation}
for all $m \in M,u \in U$. Such a right module left comodule $M$ 
is called \emph{stable (SaYD)} if one has
$$
m_{(0)}m_{(-1)} = m.
$$
\end{definition}
Observe that the category $\ayd$ of anti Yetter-Drinfel'd modules over a left Hopf algebroid is {\em not} monoidal, 
not even in the Hopf algebra case, {\em i.e.}, for $A=k$.

\begin{rem}
Note that for a left Hopf algebroid $U$ and a Yetter-Drinfel'd module $N$, the compatibility condition \rmref{huhomezone1a} can be expressed as 
\begin{equation}
\label{huhomezone1b}
(un)_{(-1)} \otimes_\ahha (un)_{(0)}= 	  u_{+(1)} n_{(-1)} u_{-} \otimes_\ahha u_{+(2)} n_{(0)},
\end{equation}
showing more structural symmetry with respect to the anti Yetter-Drinfel'd case. 
However, this formulation obscures the fact that Yetter-Drinfel'd modules already exist on the bialgebroid level.
\end{rem}

\subsection{The para-(co)cyclic $k$-modules 
$C_\bull(U,M)$ and $C^\bull_{\rm co}(U,N)$}

The Gerstenhaber and Batalin-Vilkovisky algebras 
that we are going to study in this paper are 
obtained as the (co)simplicial (co)homology of 
para-(co)cyclic $k$-modules of the following form  
\cite{KowKra:CSIACT}:

\begin{proposition}
\label{kamille}
(1) For every $M \in \modu$ over a left bialgebroid $U$ 
there is a well-defined simplicial 
$k$-module structure on   
$$
        C_\bull(U,M) := 
        M \otimes_\Aopp  
        ({}_\blact U_ \ract)^{\otimes_\Aopp
          \bull}
$$
whose face maps in degree $n \geq 1$ are given by
\begin{equation*}
\!\!\!
\begin{array}{rcll}
        \dd_i(m,x)  
&\!\!\!\!\! =& \!\!\!\!\!\left\{ \!\!\!
\begin{array}{l}
        (m,u^1, \ldots,\varepsilon(u^n) \blact u^{n-1}),
\\
(m,\ldots,u^{n-i} u^{n-i+1},
\ldots,u^n) 
\\
(mu^1,u^2,\ldots,u^n) 
\end{array}\right.  & \!\!\!\!\!\!\!\!\!\!\!\! \,  \begin{array}{l} \mbox{if} \ i \!=\! 0, \\ \mbox{if} \ 1
\!  \leq \! i \!\leq\! n-1, \\ \mbox{if} \ i \! = \! n, \end{array} 
\medskip
\\
\medskip
&&&{\hspace*{-8.2cm}{\mbox{and vanish for $n =0$, that is, for elements in $M$. The degeneracies for $n \geq 0$ read:}}}
\\
\sss_j(m,x) &\!\!\!\!\! =&\!\!\!\!\!  \left\{ \!\!\!
\begin{array}{l} (m,u^1,  
\ldots,u^n,1)
\\
(m,\ldots,u^{n-j}, 1,  
u^{n-j+1},\ldots,u^n)  
\\
(m,1,u^1,\ldots,u^n) 
\end{array}\right.   & \!\!\!\!\!\!\!\!\!\!\!  \begin{array}{l} 
\mbox{if} \ j\!=\!0, \\ 
\mbox{if} \ 1 \!\leq\! j \!\leq\! n-1, \\  \mbox{if} \ j\! = \!n. \end{array} 
\end{array}
\end{equation*}
Here and in what follows, 
we denote elementary tensors in 
$C_\bull(U,M)$ by
$$
        (m,x):=(m,u^1,\ldots,u^n),\quad
m \in M,u^1,\ldots,u^n \in U.
$$
For a right $U$-module left $U$-comodule $M$ over 
a left {\em Hopf} algebroid $U$, the $k$-module $C_\bull(U,M)$  
becomes a para-cyclic $k$-module via   
\begin{equation*} 
\ttt(m,x) = 
(m_{(0)} u^1_+,u^2_+,\ldots,u^n_+,
u^n_- \cdots u^1_- m_{(-1)}).
\end{equation*}
This para-cyclic $k$-module 
is cyclic if $M$ is a stable anti Yetter-Drinfel'd module. 


(2) On the other hand, for $M' \in \ucomod$ for a left bialgebroid $U$, there is a well-defined cosimplicial $k$-module structure on 
$$
C^\bull_{\rm co}(U,M') \, := \,\, (\due U \lact \ract)^{\otimes_\ahha  \bull} \otimes_\ahha M',
$$
with cofaces in degree $n \geq 1$ given by
\begin{equation}
\!\! \begin{array}{rll}
\label{anightinpyongyang}
\gd_i(z, m') \!\!\!\!&= \left\{\!\!\!
\begin{array}{l} (1, u^1, \ldots
 , u^n , m')  
\\ 
(u^1 , \ldots , \gD (u^i) , \ldots
 , u^n , m')
\\
(u^1 , \ldots , u^n , m'_{(-1)} , m'_{(0)}) 
\end{array}\right. 
&   \begin{array}{l} \mbox{if} \ i=0, \\ \mbox{if} \
  1 \leq i \leq n, \\ \mbox{if} \ i = n + 1,  \end{array} 
\medskip
\\
\medskip
&&{\hspace*{-8cm}{\mbox{and for $n =0$, that is, on $M'$ by}}}
\\
\gd_j(m') \!\!\!\! &= \left\{ \!\!\!
\begin{array}{l}
		  (1, m')  \quad
\\
(m'_{(-1)} , m'_{(0)})  \quad 
\end{array}\right. &  
\begin{array}{l} \mbox{if} \ j=0, \\ \mbox{if} \
  j = 1.  \end{array} 
\medskip
\\
\medskip
&&{\hspace*{-8.1cm}{\mbox{The codegeneracies for $n \geq 1$ read, on the other hand,}}}
\\
\gs_i(z , m') \!\!\!\! 
&= (u^1 , \ldots ,
{\varepsilon} (u^{i+1}) , \ldots , u^n , m') & \ \  0 \leq i \leq n-1,
\end{array}
\end{equation}
and vanish on $M'$. 
Similarly as above, here and in what follows, we denote elementary tensors in $C^\bull_{\rm co}(U,M')$ by 
$$
(z,m'):=(u^1, \ldots ,  u^n, m'), \quad m' \in M', \ u^1, \ldots, u^n \in U, 
$$
if no confusion with the homology case can arise.
Again, for a right $U$-module left $U$-comodule $M'$ over a left {\em Hopf} algebroid $U$, the $k$-module $C^\bull_{\rm co}(U,M')$ becomes a para-cocyclic $k$-module by means of
\begin{equation}
\label{casadilivia}
\tau(z,m') = (u^1_{-(1)}u^2 , \ldots, u^1_{-(n-1)}u^n, u^1_{-(n)}m'_{(-1)}, m'_{(0)}u^1_+),  
\end{equation}
which is cocyclic if $M'$ is a stable anti Yetter-Drinfel'd module.
\end{proposition}
Recall that in the first case this means that the operators $(\dd_i,\sss_j,\ttt_n)$ satisfy all
the defining relations of a cyclic $k$-module in the sense of Connes 
(see, {\em e.g.}, \cite{Con:NCDG, Lod:CH}), 
except for the one that requires that 
$
\ttt^{n+1} = \id
$ 
on $C_n(U,M)$, 
which, as mentioned,
is only satisfied when $M$ is an SaYD module; analogous comments apply to the cohomology situation. 
The relation between the cyclic and the cocyclic module above as (a sort of) {\em cyclic duals} (as introduced by Connes \cite{Con:CCEFE} as well) is explained in \cite{KowKra:CSIACT}.
Although we shall need the full structure of the (co)cyclic modules, 
we are not going to study the cyclic (co)homology of these
objects, but rather their (co)simplicial (co)homology: 

\begin{definition}
\label{avviso}
For any bialgebroid $U$ and any $M \in \modu$, 
we denote the simplicial 
homology of $C_\bull(U,M)$, that is, the  
homology with respect to the
boundary map 
\begin{equation}
\label{appolloni1}
        \bb := \sum_{i=0}^n (-1)^i \dd_i  
  \end{equation}
by $H_\bull(U,M)$ and
call it the 
\emph{(Hochschild) homology of $U$ with
coefficients in $M$}.  Likewise, if $M' \in \ucomod$ is a left $U$-comodule, we denote the cosimplicial 
cohomology of $C^\bull_{\rm co}(U,M')$, {\em i.e.}, the  
cohomology with respect to the
coboundary map 
\begin{equation}
\label{appolloni2}
        \gb := \sum_{i=0}^{n+1} (-1)^i \gd_i  
  \end{equation}
by $H^\bull_{\rm co}(U,M')$ and
call it the 
\emph{(coring} or \emph{co-Hochschild) cohomology of $U$ with
coefficients in $M'$}.  
\end{definition}

Recall from \cite[Thm.~2.13]{KowPos:TCTOHA} that if $\due U \blact {}$ is projective as a left $A$-module, then 
$$
H_\bull(U,M) \simeq \Tor^U_\bull(M,A),
$$ 
and if $U_\ract$ is flat as a right $A$-module, we have 
\begin{equation}
\label{cotor}
H^\bull_{\rm co}(U,M') \simeq \Cotor^\bull_U(A,M').
\end{equation}

Mostly, we will work on the normalised complex $\bar{C}_\bull(U,M)$
of $C_\bull(U,M)$, meaning the quotient by the subcomplex spanned by the
images of the degeneracy maps of this simplicial 
$k$-module, {\em i.e.}, given by the cokernel of the degeneracy maps. Likewise, the normalised complex $\bar{C}^\bull_{\rm co}(U,M')$ of the cochain complex $C^\bull_{\rm co}(U,M')$ is obtained by dealing with the kernel of the codegeneracy maps.
We shall usually denote operators that descend from the
original complexes to these quotients by the same symbols if no
confusion can arise. 

On every para-cyclic $k$-module, 
one furthermore defines the {\em norm operator}, 
the {\em extra  degeneracy},
and the \emph{cyclic differential} 
\begin{equation}
\label{extra}
        \cN := \sum_{i=0}^n (-1)^{in} \ttt^i, \qquad         
        \sss_{-1} := \ttt \, \sss_n,\qquad
        \BB=(\mathrm{id}-\ttt) \, \sss_{-1} \, \cN,
\end{equation}
respectively.
Remember that $\BB$ coincides on the 
normalised complex 
$\bar C_\bull(U,M)$
with the map (induced by) $\sss_{-1} \, \cN$, so we 
take the liberty to denote the latter by 
$\BB$ as well, as we, in fact, will only consider the induced map on the
normalised complex. 

\subsection{Gerstenhaber algebras, operads with multiplication, and cyclic operads}
\label{responsabilitacivile}
In this section, we finally 
gather some well-known material about Gerstenhaber algebras and their relation to operads, 
as well as Batalin-Vilkovisky algebras and their relation to cyclic operads. 
We refer the interested reader to, for example, \cite{GerSch:ABQGAAD, LodVal:AO, MarShnSta:OIATAP, Men:BVAACCOHA} for more details on operad theory; here, we only need the basic definition in the formulation of Gerstenhaber-Schack \cite{GerSch:ABQGAAD} (termed ``strict unital comp algebra'' therein) plus one important consequence:

\begin{definition}
\label{moleskine}
A (non-$\gS$) {\em operad} in the category of $k$-modules is a sequence $\{O(n)\}_{n \geq 0}$ of $k$-modules with an identity element 
$\mathbb{1} \in O(1)$ together with $k$-bilinear operations $\circ_i: O(p) \otimes O(q) \to O(p+q-1)$
such that
\begin{eqnarray}
\label{danton}
\nonumber
\gvf \circ_i \psi &=& 0 \qquad \qquad \qquad \qquad \qquad \! \mbox{if} \ p < i \quad \mbox{or} \quad p = 0, \\
(\varphi \circ_i \psi) \circ_j \chi &=& 
\begin{cases}
(\varphi \circ_j \chi) \circ_{i+r-1} \psi \qquad \mbox{if} \  \, j < i, \\
\varphi \circ_i (\psi \circ_{j-i +1} \chi) \qquad \hspace*{1pt} \mbox{if} \ \, i \leq j < q + i, \\
(\varphi \circ_{j-q+1} \chi) \circ_{i} \psi \qquad \mbox{if} \ \, j \geq q + i,
\end{cases} \\
\nonumber
\gvf \circ_i \mathbb{1} &=& \mathbb{1} \circ_i \gvf \ \ = \ \ \gvf \! \qquad \quad \qquad \mbox{for} \ i \leq p,   
\end{eqnarray}
is fulfilled for any $\varphi \in O(p), \ \psi \in O(q)$, and $\chi \in O(r)$. 
The operad is called an {\em operad with multiplication} if there exists a {\em distinguished element} or {\em operad multiplication} $\mu \in O(2)$ and an element $e \in O(0)$ such that additionally
\begin{equation}
\label{distinguished, I said 1}
\begin{array}{rclrcl}
\mu \circ_1 \mu &=& \mu \circ_2 \mu, \\
\mu \circ_1 e &=& 
\mu \circ_2 e
\ \ = \ \ \mathbb{1} 
\end{array}
\end{equation}
holds.
\end{definition}

In the rest of this article, the term ``operad'' will always refer to a non-$\gS$ operad in the category of $k$-modules in the above sense.

Gerstenhaber algebra structures can be constructed, for example, 
by means of the notion of an operad with multiplication, as the following theorem shows:

\begin{theorem}[\cite{Ger:TCSOAAR, GerSch:ABQGAAD, McCSmi:ASODHCC}]
\label{customerscopy}
Each operad with multiplication gives rise to a cosimplicial $k$-module the cohomology of which is a Gerstenhaber algebra.
\end{theorem}

For later use, we give the necessary structure maps that constitute the proof of this theorem:
for any two cochains $\varphi \in O(p),\psi \in O(q)$, define
\begin{equation*}
        \varphi \bar\circ \psi := 
(-1)^{|p||q|}        
\sum^{p}_{i=1}
        (-1)^{|q||i|} \varphi \circ_i \psi \in O({p+q-1}), \qquad |n|:= n- 1,
\end{equation*}
and their \emph{Gerstenhaber bracket} by
\begin{equation}
\label{zugangskarte}
{\{} \varphi,\psi \}
:= \varphi \bar\circ \psi - (-1)^{|p||q|} \psi \bar\circ \varphi,
\end{equation}
whereas the graded commutative product, the {\em cup product}, is given as
\begin{equation}
\label{cupco}
        \varphi \smallsmile \psi = 
        (\mu \circ_1 \varphi) \circ_{p+1} \psi = 
        (\mu \circ_2 \psi) \circ_1 \varphi \in O(p+q).
\end{equation}
Finally, the cohomology mentioned in Theorem \ref{customerscopy} 
is defined with respect to the differential
\begin{equation*}
\label{erfurt}
        \delta \varphi = \{\mu,  \varphi\}.
\end{equation*}

We will frequently use Theorem \ref{customerscopy} in the next section. A sharpened version of this result is an analogous 
relation between Batalin-Vilkovisky algebras and {\em cyclic} operads established in \cite{Men:BVAACCOHA}, as we will recall below. 
The notion of cyclic operad goes back to \cite{GetKap:COACH}, see also \cite[p.~247--248]{MarShnSta:OIATAP}; 
the version we use here is due to \cite{Men:BVAACCOHA}:

\begin{definition}
A {\em cyclic} operad is a (non-$\gS$) operad $O$ equipped with $k$-linear maps $\tau_n: O(n) \to O(n)$ subject to
\begin{equation}
\label{superfluorescent1}
\begin{array}{rcll}
\tau_{|p+q|}(\gvf \circ_1 \psi) &=& \tau_q\psi \circ_q \tau_p \gvf, & \mbox{if} \ 1 \leq p, q, \\
\tau_{|p+q|}(\gvf \circ_i \psi) &=& \tau_p\gvf \circ_{i-1} \psi, & \mbox{if}  \ 0 \leq q \ \mbox{and} \  2 \leq i \leq p, \\
\tau_n^{n+1} &=& \id_{O(n)}, & \\
\tau_1 \mathbb{1} &=& \mathbb{1} &
\end{array}
\end{equation}
for every $\gvf \in O(p)$ and $\psi \in O(q)$. A {\em cyclic operad with multiplication} is simultaneously a cyclic operad and an operad with multiplication $\mu$ such that
\begin{equation}
\label{superfluorescent2}
\tau_2 \mu = \mu.
\end{equation}
\end{definition}

A crucial observation is now that Batalin-Vilkovisky algebras arise, for example, from cyclic operads with multiplication:

\begin{theorem}[\cite{Men:BVAACCOHA}]
\label{holl}
Each cyclic operad with multiplication 
gives rise to a cocyclic module of which the associated cyclic differential $\BB$ yields a generator for the Gerstenhaber bracket on the cohomology of the underlying cosimplicial $k$-module, turning it therefore into a Batalin-Vilkovisky algebra. 
\end{theorem}

\section{Gerstenhaber and Batalin-Vilkovisky algebra structures for bialgebroids}
\label{caprarola}
In this section, we will construct two operad structures with multiplication 
on two different cosimplicial modules attached to a left bialgebroid $U$ that compute, 
under suitable projectivity assumptions, the derived functors 
$\Ext_U$ and $\Cotor_U$, see below. The first one is the $\Aop$-linear dual to $C_\bull(U,N)$ from Proposition \ref{kamille}, whereas the second one is its cyclic dual, {\em i.e.}, the cosimplicial module $C^\bull_{\rm co}(U,N)$ introduced in the second part of the same proposition.

\subsection{$C^\bull(U,N)$ as an operad with multiplication}
\label{palermo}
As for the first one mentioned, define 
$$
C^p(U,N) := \Hom_\Aopp(U^{\otimes_\Aopp p}_\ract, N),   \qquad N \in \moda,
$$
which is, by duality, a cosimplicial module. 
The differential $\gd: C^\bull(U,N) \to C^{\bull+1}(U,N)$ is given by
\begin{small}
\begin{equation}
\label{spaetkauf}
\begin{split}
        \gd\varphi(u^1, \ldots, u^{p+1}) 
&:= u^1 \varphi(u^2, \ldots, u^{p+1}) 
+ \sum^{p}_{i=1} (-1)^i 
        \varphi(u^1, \ldots, u^i u^{i+1}, \ldots, u^{p+1}) \\
& \qquad + (-1)^{p+1} 
        \varphi(u^1, \ldots, \varepsilon(u^{p+1}) \blact u^p).
\end{split}
\end{equation}
\end{small}

We denote the cohomology of $C^\bull(U,N)$ 
by $H^\bull(U,N)$ and call this the \emph{(Hochschild) cohomology of $U$ with
  coefficients in $N$}. If $U_\ract$ is an $\Aop$-projective left
bialgebroid, then $H^\bull(U,N) \simeq \mathrm{Ext}^\bull_U(A,N)$, 
but in general we use the symbol $H^\bull(U,N)$ for the cohomology of the explicit cochain complex 
$\big(C^\bull(U,N), \gd\big)$. The {\em normalised complex} 
$\bar C^\bull(U,N)$ is given by the intersection of the kernels of the codegeneracies in the cosimplicial $k$-module $C^\bull(U,N)$.

The Gerstenhaber structure on $H^\bull(U,N)$ for the case $N:=A$ was already discussed in \cite[\S\S3.5--3.6]{KowKra:BVSOEAT}, 
we insert here general coefficients:
let $N$ be a left $U$-module (with action denoted by juxtaposition) which is 
simultaneously a left $U$-comodule with coaction $\gD_\enne: n \mapsto n_{(-1)} \otimes_\ahha n_{(0)}$ 
such that the underlying induced left $\Ae$-module structures coincide. Furthermore, 
assume that with respect to this left $\Ae$-module structure $N$ is an $A$-ring, {\em i.e.}, a monoid in $\amoda$ 
with multiplication denoted by $(n,n') \mapsto n \cdot_\enne n'$ in what follows. 
Observe that by these requirements in particular Eq.~\rmref{normalny} holds.
We then associate to any $p$-cochain $\varphi \in  C^p(U,N)$ the operator
\begin{equation}
\begin{array}{rcl}
\label{huetor1}
        \DD_\varphi^\enne: 
        U^{\otimes_\Aopp p} &\to& U \otimes_\ahha N, \\
        (u^1, \ldots, u^p) &\mapsto& 
        \varphi(u^1_{(1)}, \ldots, u^p_{(1)})_{(-1)} u^1_{(2)} \cdots u^p_{(2)} \otimes_\ahha \varphi(u^1_{(1)}, \ldots, u^p_{(1)})_{(0)},
\end{array}
\end{equation}
and this map is well-defined by \rmref{maotsetung} 
along with \rmref{tellmemore}.
For zero cochains, {\em i.e.}, elements in $N$, this map is given by the coaction $\gD_\enne$ of $N$. We introduce the notation
$$
\DD_\varphi(u^1, \ldots, u^p)_{(-1)} \otimes_\ahha \DD_\varphi(u^1, \ldots, u^p)_{(0)}
:= \DD_\varphi^\enne(u^1, \ldots, u^p).
$$
This enables us to define
$$
        \circ_i : C^p(U,N) \otimes C^q(U,N) \rightarrow 
        C^{|p+q|}(U,N),\quad
        i = 1, \ldots, p,
$$
by
\begin{equation}
\label{maxdudler}
\begin{split}
        &\ (\varphi \circ_i \psi)(u^1, \ldots, u^{|p+q|}) 
\\ &\ \quad 
:= \varphi(u^1_{(1)}, \ldots, u^{i-1}_{(1)}, 
        \DD_\psi(u^{i}, \ldots, u^{i+q-1})_{(-1)}, 
        u^{i+q}, \ldots, u^{|p+q|}) \\
&\hspace*{4cm} \cdot_\enne \big(u^1_{(2)} \cdots u^{i-1}_{(2)} \DD_\psi(u^{i}, \ldots, u^{i+q-1})_{(0)}\big), 
\end{split}
\end{equation}
which again is well-defined by \rmref{tellmemore}, \rmref{maotsetung}, but also \rmref{normalny}.
For zero cochains $n \in N$, we define $n \circ_i \psi = 0$ for all $i$ and
all $\psi$, whereas 
$$
\gvf \circ_i n :=  \varphi(u^1_{(1)}, \ldots, u^{i-1}_{(1)}, 
        n_{(-1)},u^{i}, \ldots, u^{p-1}) \cdot_\enne (u^1_{(2)} \cdots u^{i-1}_{(2)} n_{(0)}) \in C^{p-1}(U,N).
$$
 The {distinguished element}, {\em i.e.}, the {operad multiplication} \rmref{distinguished, I said 1} is here given by
\begin{equation}
\label{distinguished, I said 2}
        \mu := (\varepsilon \, m_\uhhu(\cdot,\cdot)) \lact 1_\enne \in C^2(U,N),
\end{equation}
where $m_\uhhu$ is the multiplication map of $U$. Furthermore, define
\begin{equation}
\label{habemuspapam1}
\mathbb{1} := \gve(\cdot) \lact 1_\enne \in C^1(U,N) \quad \mbox{and} \quad  
e := 1_\enne \in C^0(U,N).
\end{equation} 

The cup product \rmref{cupco} can then be expressed as
$$
(\gvf \smallsmile \psi)(u^1, \ldots, u^{p+q}) = \gvf(u^1_{(1)}, \ldots, u^{p}_{(1)}) \cdot_\enne \big(u^1_{(2)} \cdots u^{p}_{(2)} 
\psi(u^{p+1}, \ldots, u^{p+q})\big),
$$
where $\gvf \in C^p(U,N)$,  $\psi \in C^q(U,N)$, 
as a short computation reveals by means of \rmref{maxdudler} and \rmref{distinguished, I said 2} with the help of \rmref{zilvesta}, \rmref{tellmemore}, and \rmref{maotsetung}.

\begin{theorem}
\label{clairefontaine1a}
If $N$ is a braided commutative Yetter-Drinfel'd algebra over a left bialgebroid $U$,
then $C^\bull(U,N)$ with the structure given in \rmref{maxdudler}--\rmref{habemuspapam1} 
defines an operad with multiplication.
\end{theorem} 

\begin{rem}
As already mentioned, one can in particular take $N:=A$, the base algebra of the left bialgebroid itself, and in this case the operators \rmref{spaetkauf}--\rmref{habemuspapam1} coincide, by means of the canonical left action \rmref{stabatmater1} resp.\ coaction \rmref{stabatmater2}, with the operators given in \cite[\S3.5]{KowKra:BVSOEAT}.
On the other hand, observe that even for general coefficients $N$ one has for the differential \rmref{spaetkauf} the usual expression
$$
\gd \gvf = \{\mu, \gvf\},
$$ 
where the right hand side is defined as in \rmref{zugangskarte}. This follows from \rmref{tellmemore} and \rmref{maotsetung}.
\end{rem}

\begin{proof}[Proof of Theorem \ref{clairefontaine1a}]
Verifying the conditions in Definition \ref{moleskine} is essentially a direct computation, 
but we want to show at least 
two of the identities in \rmref{danton} in detail to illustrate where the various assumptions on $N$ in Definition \ref{fiocchi} of a braided commutative 
Yetter-Drinfel'd algebra appear; 
apart from that, 
the identities \rmref{distinguished, I said 1} are easy to check considering that for the distinguished element \rmref{distinguished, I said 2}
$$
D^\enne_\mu(u, v) = uv \otimes_\ahha 1_\enne, \quad \mbox{for all} \ u,v \in U,
$$ 
holds in case $N$ is a comodule algebra. 
Somewhat less obvious is to check that the identities \rmref{danton} are fulfilled: 
let $\gvf \in C^p(U,N)$, $\psi \in C^q(U,N)$, and $\chi \in C^r(U,N)$, along with $i \leq j < q + i$. Then
\begin{footnotesize}
\begin{equation*}
\begin{split}
(&(\gvf \circ_i \psi) \circ_j \chi)(u^1, \ldots, u^{||p+q|+r|}) \\
&=\big((\gvf \circ_i \psi)(u^1_{(1)}, \ldots, u^{j-1}_{(1)}, 
        \DD_\chi(u^{j}, \ldots, u^{j+r-1})_{(-1)}, 
        u^{j+r}, \ldots, u^{||p+q|+r|})\big) \\
&\hspace*{4cm} \cdot_\enne \big(u^1_{(2)} \cdots u^{j-1}_{(2)} \DD_\chi(u^{j}, \ldots, u^{j+r-1})_{(0)}\big) \\
&=\gvf\Big(u^1_{(1)}, \ldots, u^{i-1}_{(1)}, \DD_\psi(u^i_{(1)}, \ldots, u^{j-1}_{(1)}, \DD_\chi(u^{j}, \ldots, u^{j+r-1})_{(-1)}, u^{j+r}, \\
& \qquad\qquad\qquad \ldots, u^{i+q+r-2})_{(-1)}, u^{i+q+r-1},  \ldots, u^{||p+q|+r|}\Big) \\
& \qquad 
\cdot_\enne \Big(u^1_{(2)} \cdots u^{i-1}_{(2)} \DD_\psi(u^i_{(1)}, \ldots, u^{j-1}_{(1)}, 
        \DD_\chi(u^{j}, \ldots, u^{j+r-1})_{(-1)}, u^{j+r}, \ldots, u^{i+q+r-2})_{(0)}\Big) \\
& \qquad \cdot_\enne \Big(u^1_{(3)} \cdots u^{i-1}_{(3)}u^i_{(2)} \cdots u^{j-1}_{(2)} \DD_\chi(u^{j}, \ldots, u^{j+r-1})_{(0)}\Big) \\
&\overset{\scriptscriptstyle{\rmref{zilvesta}}}{=}
\gvf\Big(u^1_{(1)}, \ldots, u^{i-1}_{(1)}, \DD_\psi(u^i_{(1)}, \ldots, u^{j-1}_{(1)}, \DD_\chi(u^{j}, \ldots, u^{j+r-1})_{(-1)}, u^{j+r}, \ldots, u^{i+q+r-2})_{(-1)}, \\
& \qquad \qquad u^{i+q+r-1}, \ldots, u^{||p+q|+r|}\Big) 
\cdot_\enne \Big(u^1_{(2)} \cdots u^{i-1}_{(2)}
\big[\DD_\psi(u^i_{(1)}, \ldots, u^{j-1}_{(1)}, \DD_\chi(u^{j}, \ldots, u^{j+r-1})_{(-1)}, \\
& \qquad \qquad u^{j+r}, \ldots, u^{i+q+r-2})_{(0)} \cdot_\enne \big(u^i_{(2)} \cdots u^{j-1}_{(2)} \DD_\chi(u^{j}, \ldots, u^{j+r-1})_{(0)}\big)\big]\Big) \\
&\overset{\scriptscriptstyle{\rmref{huetor1}}}{=}
\gvf\Big(u^1_{(1)}, \ldots, u^{i-1}_{(1)}, \psi(u^i_{(1)}, \ldots, u^{j-1}_{(1)}, \DD_\chi(u^{j}_{(1)}, \ldots, u^{j+r-1}_{(1)})_{(-2)}, u^{j+r}_{(1)}, \ldots, u^{i+q+r-2}_{(1)})_{(-1)} \\
& \qquad \qquad u^i_{(2)} \cdots u^{j-1}_{(2)} \DD_\chi(u^{j}_{(1)}, \ldots, u^{j+r-1}_{(1)})_{(-1)} u^{j}_{(2)} \cdots u^{i+q+r-2}_{(2)}, \\
& \qquad \qquad u^{i+q+r-1}, \ldots, u^{||p+q|+r|}\Big) 
\cdot_\enne \Big(u^1_{(2)} \cdots u^{i-1}_{(2)}
\big[\psi(u^i_{(1)}, \ldots, u^{j-1}_{(1)}, \DD_\chi(u^{j}_{(1)}, \ldots, u^{j+r-1}_{(1)})_{(-2)}, \\
& \qquad \qquad u^{j+r}_{(1)}, \ldots, u^{i+q+r-2}_{(1)})_{(0)} \cdot_\enne \big(u^i_{(2)} \cdots u^{j-1}_{(2)} \DD_\chi(u^{j}_{(1)}, \ldots, u^{j+r-1}_{(1)})_{(0)}\big)\big]\Big) 
\end{split}
\end{equation*}
\begin{equation*}
\begin{split}
&\overset{\scriptscriptstyle{\rmref{huhomezone1b}, \rmref{Sch3}}}{=}
\gvf\Big(u^1_{(1)}, \ldots, u^{i-1}_{(1)}, \psi(u^i_{(1)}, \ldots, u^{j-1}_{(1)}, \DD_\chi(u^{j}_{(1)}, \ldots, u^{j+r-1}_{(1)})_{(-1)}, u^{j+r}_{(1)}, \ldots, u^{i+q+r-2}_{(1)})_{(-1)} \\
& \qquad \qquad \big(u^i_{(2)} \cdots u^{j-1}_{(2)} \DD_\chi(u^{j}_{(1)}, \ldots, u^{j+r-1}_{(1)})_{(0)}\big)_{(-1)} u^{i}_{(3)} \cdots u^{j-1}_{(3)}u^j_{(2)} \cdots u^{i+q+r-2}_{(2)}, \\
& \qquad \qquad u^{i+q+r-1}, \ldots, u^{||p+q|+r|}\Big) 
\cdot_\enne \Big(u^1_{(2)} \cdots u^{i-1}_{(2)}
\big[\psi(u^i_{(1)}, \ldots, u^{j-1}_{(1)}, \DD_\chi(u^{j}_{(1)}, \ldots, u^{j+r-1}_{(1)})_{(-1)}, \\
& \qquad \qquad u^{j+r}_{(1)}, \ldots, u^{i+q+r-2}_{(1)})_{(0)} \cdot_\enne \big(u^i_{(2)} \cdots u^{j-1}_{(2)} \DD_\chi(u^{j}_{(1)}, \ldots, u^{j+r-1}_{(1)})_{(0)}\big)_{(0)}\big]\Big) 
\\
&\overset{\scriptscriptstyle{\rmref{corsasemplice}}}{=} 
\gvf\Big(u^1_{(1)}, \ldots, u^{i-1}_{(1)}, \Big[\psi(u^i_{(1)}, \ldots, u^{j-1}_{(1)}, \DD_\chi(u^{j}_{(1)}, \ldots, u^{j+r-1}_{(1)})_{(-1)}, u^{j+r}_{(1)}, \ldots, u^{i+q+r-2}_{(1)}) \\
& \qquad \qquad \cdot_\enne \big(u^i_{(2)} \cdots u^{j-1}_{(2)} \DD_\chi(u^{j}_{(1)}, \ldots, u^{j+r-1}_{(1)})_{(0)}\big)\Big]_{(-1)} u^{i}_{(3)} \cdots u^{j-1}_{(3)}u^j_{(2)} \cdots u^{i+q+r-2}_{(2)}, \\
& \qquad \qquad u^{i+q+r-1}, \ldots, u^{||p+q|+r|}\Big) 
\cdot_\enne \Big(u^1_{(2)} \cdots u^{i-1}_{(2)}
\Big[\psi(u^i_{(1)}, \ldots, u^{j-1}_{(1)}, \DD_\chi(u^{j}_{(1)}, \ldots, u^{j+r-1}_{(1)})_{(-1)}, \\
& \qquad \qquad u^{j+r}_{(1)}, \ldots, u^{i+q+r-2}_{(1)}) \cdot_\enne \big(u^i_{(2)} \cdots u^{j-1}_{(2)} \DD_\chi(u^{j}_{(1)}, \ldots, u^{j+r-1}_{(1)})_{(0)}\big)\Big]_{(0)}\Big) 
\\
&= \gvf\big(u^1_{(1)}, \ldots, u^{i-1}_{(1)}, 
        (\psi \circ_{j-i+1} \chi)(u^{i}_{(1)}, \ldots, u^{i+r+q-2}_{(1)})_{(-1)}u^{i}_{(2)} \cdots u^{i+r+q-2}_{(2)}, \\ 
&\qquad\qquad        
u^{i+r+q-1}, \ldots, u^{||p+q|+r|}\big) \cdot_\enne \big(u^1_{(2)} \cdots u^{i-1}_{(2)} (\psi \circ_{j-i+1} \chi)(u^{i}_{(1)}, \ldots, u^{i+r+q-2}_{(1)})_{(0)}  
\\
&= \gvf\big(u^1_{(1)}, \ldots, u^{i-1}_{(1)}, 
        \DD_{\psi \circ_{j-i+1} \chi}(u^{i}, \ldots, u^{i+r+q-2})_{(-1)}, 
        u^{i+r+q-1}, \ldots, u^{||p+q|+r|}) \\
&\hspace*{4cm} \cdot_\enne \big(u^1_{(2)} \cdots u^{i-1}_{(2)} \DD_{\psi \circ_{j-i+1} \chi}(u^{i}, \ldots, u^{i+r+q-2})_{(0)}\big)  \\
&= (\gvf \circ_i (\psi \circ_{j-i+1} \chi))(u^1, \ldots, u^{||p+q|+r|})
\end{split}
\end{equation*}
\end{footnotesize}
holds. 
The condition \rmref{circolodegliartisti} appears, {\em e.g.}, when computing for $j \geq q+i$ that  
\begin{footnotesize}
\begin{equation*}
\begin{split}
(&(\gvf \circ_i \psi) \circ_j \chi)(u^1, \ldots, u^{||p+q|+r|}) \\
&\ \ =\big((\gvf \circ_i \psi)(u^1_{(1)}, \ldots, u^{j-1}_{(1)}, 
        \DD_\chi(u^{j}, \ldots, u^{j+r-1})_{(-1)}, 
        u^{j+r}, \ldots, u^{||p+q|+r|})\big) \\
&\hspace*{4cm} \cdot_\enne \big(u^1_{(2)} \cdots u^{j-1}_{(2)} \DD_\chi(u^{j}, \ldots, u^{j+r-1})_{(0)}\big) \\
&\ \ =\gvf\big(u^1_{(1)}, \ldots, u^{i-1}_{(1)}, \DD_\psi(u^i_{(1)}, \ldots, u^{i+q-1}_{(1)})_{(-1)}, u^{i+q}_{(1)}, \ldots, u^{j-1}_{(1)}, \DD_\chi(u^{j}, \ldots, u^{j+r-1})_{(-1)}, \\
& \qquad \qquad
u^{j+r}, \ldots, u^{||p+q|+r|}\big) 
\cdot_\enne \big(u^1_{(2)} \cdots u^{i-1}_{(2)} \DD_\psi(u^i_{(1)}, \ldots, u^{i+q-1}_{(1)})_{(0)}\big) \\
& \qquad \qquad 
\cdot_\enne \big(u^1_{(3)} \cdots u^{i-1}_{(3)}u^i_{(2)} \cdots u^{j-1}_{(2)}
        \DD_\chi(u^{j}, \ldots, u^{j+r-1})_{(0)}\big) \\
&\overset{\scriptscriptstyle{\rmref{zilvesta}}}{=}
\gvf\big(u^1_{(1)}, \ldots, u^{i-1}_{(1)}, \DD_\psi(u^i_{(1)}, \ldots, u^{i+q-1}_{(1)})_{(-1)}, u^{i+q}_{(1)}, \ldots, u^{j-1}_{(1)}, \DD_\chi(u^{j}, \ldots, u^{j+r-1})_{(-1)}, \\
& \qquad \qquad
u^{j+r}, \ldots, u^{||p+q|+r|}\big) 
\cdot_\enne \Big(u^1_{(2)} \cdots u^{i-1}_{(2)} \Big(\DD_\psi(u^i_{(1)}, \ldots, u^{i+q-1}_{(1)})_{(0)} \\
& \qquad \qquad 
\cdot_\enne \big(u^i_{(2)} \cdots u^{j-1}_{(2)}
        \DD_\chi(u^{j}, \ldots, u^{j+r-1})_{(0)}\big)\Big)\Big) \\
&\overset{\scriptscriptstyle{\rmref{circolodegliartisti}}}{=}
\gvf\big(u^1_{(1)}, \ldots, u^{i-1}_{(1)}, \DD_\psi(u^i, \ldots, u^{i+q-1})_{(-2)}, u^{i+q}_{(1)}, \ldots, u^{j-1}_{(1)}, \DD_\chi(u^{j}, \ldots, u^{j+r-1})_{(-1)}, \\
& \qquad \qquad
u^{j+r}, \ldots, u^{||p+q|+r|}\big) 
\cdot_\enne \big(u^1_{(2)} \cdots u^{i-1}_{(2)} \big(\DD_\psi(u^i, \ldots, u^{i+q-1})_{(-1)} \\
& \qquad \qquad 
\big(u^{i+q}_{(2)} \cdots u^{j-1}_{(2)}
        \DD_\chi(u^{j}, \ldots, u^{j+r-1})_{(0)}\big) \cdot_\enne \DD_\psi(u^i, \ldots, u^{i+q-1})_{(0)} \big)\big) \\
&\overset{\scriptscriptstyle{\rmref{zilvesta}}}{=}
\gvf\big(u^1_{(1)}, \ldots, u^{i-1}_{(1)}, \DD_\psi(u^i, \ldots, u^{i+q-1})_{(-2)}, u^{i+q}_{(1)}, \ldots, u^{j-1}_{(1)}, \DD_\chi(u^{j}, \ldots, u^{j+r-1})_{(-1)}, \\
& \qquad \qquad
u^{j+r}, \ldots, u^{||p+q|+r|}\big) 
\cdot_\enne \big(u^1_{(2)} \cdots u^{i-1}_{(2)} \DD_\psi(u^i, \ldots, u^{i+q-1})_{(-1)} \\
& \qquad \qquad 
u^{i+q}_{(2)} \cdots u^{j-1}_{(2)}
        \DD_\chi(u^{j}, \ldots, u^{j+r-1})_{(0)}\big) \cdot_\enne \big(u^1_{(3)} \cdots u^{i-1}_{(3)}  \DD_\psi(u^i, \ldots, u^{i+q-1})_{(0)} \big) \\
&=\big((\gvf \circ_{j-q+1} \chi)(u^1_{(1)}, \ldots, u^{i-1}_{(1)}, 
        \DD_\psi(u^{i}, \ldots, u^{i+q-1})_{(-1)}, 
        u^{i+q}, \ldots, u^{||p+q|+r|})\big) \\
&\hspace*{4cm} \cdot_\enne \big(u^1_{(2)} \cdots u^{i-1}_{(2)} \DD_\psi(u^{i}, \ldots, u^{i+q-1})_{(0)}\big) \\
\\
&=((\gvf \circ_{j-q+1} \chi) \circ_i \psi)(u^1, \ldots, u^{||p+q|+r|})
\end{split}
\end{equation*}
\end{footnotesize}
is true. To check the remaining identities in \rmref{danton} is left to the reader.
\end{proof}

By Theorem \ref{customerscopy} one deduces at once:

\begin{corollary}
\label{fundbuero}
If $N$ is a braided commutative Yetter-Drinfel'd algebra,
the cohomology groups $H^\bull(U,N)$ carry the structure of a Gerstenhaber algebra. In particular, if 
$\due U \blact {}$ is projective as a left $A$-module, then $\Ext^\bull_U(A,N)$ is a Gerstenhaber algebra.
\end{corollary}

\begin{example}
\label{figlieditalia}
For $U = \Ae$ with the bialgebroid structure given as in Example \ref{todi} and in case $N:=A$, the Gerstenhaber algebra structure on 
$\Ext_\Ae(A,A)$ is (for $A$ projective over $k$) the classical one given in \cite{Ger:TCSOAAR}, as already mentioned in \cite[\S7]{KowKra:BVSOEAT}:
one has an isomorphism
\begin{equation}
\label{yamamay}
C^\bull(\Ae, A) \to C^\bull(A,A) := \Hom_k(A^{\otimes \bull}, A), \quad \gvf \mapsto \tilde{\gvf},
\end{equation}
where the right hand side refers to the standard Hochschild cochain complex 
and $ \tilde \varphi $ is defined by 
$$
        \tilde\gvf(a_1 \otimes_k \cdots \otimes_k a_n) 
        :=       \gvf\big((a_1 \otimes_k 1),  \ldots, (a_n \otimes_k 1)\big) 
$$
so that 
$$
        \gvf\big((a_1 \otimes_k b_1), \ldots, (a_n \otimes_k b_n)\big) =
        \tilde\gvf(a_1 \otimes_k \cdots \otimes_k a_n)b_n \cdots b_1. 
$$
It is easy to see that the respective induced isomorphism on cohomology is moreover one of Gerstenhaber algebras, and in particular one has
$$
\widetilde{\phi \circ_i \psi} = \tilde{\phi} \circ_i^{\scriptscriptstyle{G}} \tilde{\psi}, \qquad i = 1, \ldots, p,
$$ 
for all $\gvf \in C^p(\Ae,A)$, $\psi \in C^q(\Ae, A)$, where the right hand side 
\begin{equation}
\label{maifest}
 (\tilde{\gvf} \circ_i^{\scriptscriptstyle{G}} \tilde{\psi})(a_1, \ldots, a_{|p+q|}) :=  
\tilde{\gvf}\big(a_1, \ldots, a_{i-1}, \tilde{\psi}(a_i, \ldots, a_{i+|q|}), a_{i+q}, \ldots, a_{|p+q|}\big)
\end{equation}
are the classical insertion operations found by Gerstenhaber \cite{Ger:TCSOAAR}.

However, we want to underline that already in [{\em loc.~cit.}, p.~287], coefficients 
were introduced by considering (what turns out to be) an $A$-ring $N$ with an $A$-bimodule map $\phi: N \to A$ such that 
\begin{equation}
\label{camisanegra}
\phi(n)n' = n \cdot_\enne n' = n \phi(n'). 
\end{equation}
We now show how these coefficients are examples of our general construction: 
an $A$-ring $N$ is, as said before, by definition an $\Ae$-module ring and if $N$ is also a comodule, 
the Yetter-Drinfel'd condition \rmref{huhomezone1a} is automatically fulfilled by \rmref{maotsetung} and the $\Ae$-linearity of the coproduct on $\Ae$. 
Hence, every $A$-ring which also is an $\Ae$-comodule algebra is automatically a Yetter-Drinfel'd algebra. 
If a morphism $\phi \in \Hom_\Ae(N,A)$ now fulfills \rmref{camisanegra}, it is clear that 
$n \mapsto (\phi(n) \otimes_k 1_\ahha) \otimes_\ahha 1_\enne$ defines an $\Ae$-coaction which gives $N$ the structure of a braided commutative Yetter-Drinfel'd algebra. 
\end{example}

\begin{example}
Another classical example fits in this theory as follows: 
let $H$ be a Hopf algebra over a field $k$ with antipode $S$, and denote the Hochschild cohomology of $H$ as an algebra with values in $H$ itself by $H^\bull_{\scriptscriptstyle{\rm alg}}(H,H) := \Ext^\bull_\Hee(H,H)$, which, as mentioned in the preceding Example \ref{figlieditalia}, 
classically carries a Gerstenhaber algebra structure \cite{Ger:TCSOAAR}. 
Now, as follows from \cite[Thm.~VIII.3.1]{CarEil:HA}, one has a vector space isomorphism 
\begin{equation}
\label{mauerpark}
\Ext^\bull_\Hee(H,H) \simeq \Ext^\bull_\akka(k,\ad(H)).
\end{equation}
Here, $\ad(H)$ is $H$ as vector space but with left $H$-action on it 
given by the adjoint action $\ad(h)h' := h_{(1)} h' S(h_{(2)})$ for all $h,h' \in H$.  
It is a straightforward check that $\ad(H)$ equipped with this action, the left $H$-coaction given by the coproduct in $H$, and the ring structure given by the multiplication in $H$ forms 
a braided commutative Yetter-Drinfel'd algebra, hence $\Ext^\bull_\akka(k,\ad(H))$ is also a Gerstenhaber algebra by Corollary \ref{fundbuero}. 

We correspondingly want to show that \rmref{mauerpark} is also an isomorphism of Gerstenhaber algebras:
recall from, {\em e.g.}, \cite{FenTsy:HACHOQG, Kra:OTHCOQHS} that $\Ext^\bull_\akka(k,\ad(H))$ can be computed by means of the standard Hochschild cochain complex $C^\bull(H, H) := \Hom_k(H^{\otimes \bull}, H)$, but with coboundary
\begin{equation*}
\begin{split}
d\gvf(h^1, \ldots, h^{n+1}) &:= \ad(h^1) \gvf(h^2, \ldots, h^{n+1}) + \sum^n_{i=1} (-1)^i \gvf(h^1, \ldots, h^ih^{i+1}, \ldots, h^{n+1}) \\
& \quad + (-1)^{n+1} \gvf(h^1, \ldots, h^n)\gve(h^{n+1}).
\end{split}
\end{equation*}
Now, by means of the $k$-linear isomorphism
$$
\xi: C^\bull(H,H) \to C^\bull(H,H), \quad (\xi(\gvf))(h^1, \ldots, h^n) := \gvf(h^1_{(1)}, \ldots, h^n_{(1)})h^1_{(2)} \cdots h^n_{(2)}
$$
with inverse
$$
(\xi^{-1}(\psi))(v^1, \ldots, v^n) := \psi(v^1_{(1)}, \ldots, v^n_{(1)})S(v^1_{(2)}, \ldots, v^n_{(2)}),
$$
one proves \cite[\S2.4]{Kra:OTHCOQHS} that $\gd \xi = \xi d$ with respect to the Hochschild coboundary $\gd$ from \rmref{spaetkauf} (adapted to this example, {\em i.e.}, for $A :=k$ and $N := H$) and hence $\big(C^\bull(H,H), \gd\big) \simeq \big(C^\bull(H,H), d\big) $ as cochain complexes. Moreover, one has with respect 
to Gerstenhaber's insertion operations \rmref{maifest} for all $\gvf \in C^p(H,H)$, $\psi \in C^q(H,H)$, and $i = 1, \ldots, p$, 
by means of the properties of a Hopf algebra and since $k$ is central in $H$
\begin{footnotesize}
\begin{equation*}
\begin{split}
&(\xi^{-1}(\xi \gvf \circ_i^{\scriptscriptstyle{G}} \xi\psi))(h^1, \ldots, h^{|p+q|}) \\
&= \big((\xi \gvf \circ_i^{\scriptscriptstyle{G}} \xi\psi))(h^1_{(1)}, \ldots, h^{|p+q|}_{(1)})\big)S(h^1_{(2)} \cdots h^{|p+q|}_{(2)}) \\
&\overset{\scriptscriptstyle{\rmref{maifest}}}{=} 
\big(\xi \gvf\big)\big(h^1_{(1)}, \ldots, h^{i-1}_{(1)}, (\xi\psi)(h^i_{(1)}, \ldots, h^{i+|q|}_{(1)}), 
h^{i+q}_{(1)}, \ldots, h^{|p+q|}_{(1)}\big)S(h^1_{(2)} \cdots 
h^{|p+q|}_{(2)}) 
\\
&\overset{\scriptscriptstyle{}}{=} 
\gvf\big(h^1_{(1)}, \ldots, h^{i-1}_{(1)}, \psi(h^i_{(1)}, \ldots, h^{i+|q|}_{(1)})_{(1)} h^i_{(2)} \cdots h^{i+|q|}_{(2)}, 
h^{i+q}_{(1)}, \ldots, h^{|p+q|}_{(1)}\big) h^1_{(2)} \cdots h^{i-1}_{(2)} 
\\
& \qquad 
\psi(h^i_{(1)}, \ldots, h^{i+|q|}_{(1)})_{(2)} h^i_{(3)} \cdots h^{i+|q|}_{(3)} 
h^{i+q}_{(2)} \cdots h^{|p+q|}_{(2)}
S(h^1_{(3)} \cdots h^{i-1}_{(3)} h^i_{(4)} \cdots h^{i+|q|}_{(4)} h^{i+q}_{(3)} \cdots
h^{|p+q|}_{(3)}) 
\\
&\overset{\scriptscriptstyle{}}{=} 
\gvf\big(h^1_{(1)}, \ldots, h^{i-1}_{(1)}, \psi(h^i_{(1)}, \ldots, h^{i+|q|}_{(1)})_{(1)} h^i_{(2)} \cdots h^{i+|q|}_{(2)}, 
h^{i+q}, \ldots, h^{|p+q|}\big) 
\\
& \qquad \quad
\cdot_\akka h^1_{(2)} \cdots h^{i-1}_{(2)}  \psi(h^i_{(1)}, \ldots, h^{i+|q|}_{(1)})_{(2)} 
S(h^1_{(3)} \cdots h^{i-1}_{(3)}) \\
&\overset{\scriptscriptstyle{}}{=} 
\gvf\big(h^1_{(1)}, \ldots, h^{i-1}_{(1)}, \psi(h^i_{(1)}, \ldots, h^{i+|q|}_{(1)})_{(1)} h^i_{(2)} \cdots h^{i+|q|}_{(2)}, 
h^{i+q}, \ldots, h^{|p+q|}\big) 
\\
& \qquad \quad
\cdot_\akka \big(\ad(h^1_{(2)} \cdots h^{i-1}_{(2)})  \psi(h^i_{(1)}, \ldots, h^{i+|q|}_{(1)})_{(2)}\big) \\
&\overset{\scriptscriptstyle{\rmref{huetor1}}}{=} 
\gvf\big(h^1_{(1)}, \ldots, h^{i-1}_{(1)}, D_\psi(h^i, \ldots, h^{i+|q|})_{(1)}, 
h^{i+q}, \ldots, h^{|p+q|}\big) 
\\
& \qquad \quad
\cdot_\akka \big(\ad(h^1_{(2)} \cdots h^{i-1}_{(2)})  D_\psi(h^i, \ldots, h^{i+|q|})_{(2)}\big) \\
&\overset{\scriptscriptstyle{\rmref{maxdudler}}}{=}
(\gvf \circ_i \psi)(h^1, \ldots, h^{|p+q|}),
\end{split}
\end{equation*}
\end{footnotesize}
where, at certain places, we denoted the multiplication in $H$ by $\cdot_\akka$ 
to better illustrate the analogy to \rmref{maxdudler}.
Hence, as claimed,
the vector space 
isomorphism $\xi$ induces an isomorphism of Gerstenhaber algebras on cohomology.
\end{example}

\begin{example}
If $H$ is a $k$-bialgebra and $V$ simultaneously an $H$-module algebra and a left $H$-comodule, then $V \# H$ is a left bialgebroid over $V$ if and only if $V$ is a braided commutative Yetter-Drinfel'd algebra over $H$, see \cite[Thm.~5.1]{Lu:HAAQG} or \cite[Thm.~4.1]{BrzMil:BBAD}.
In this case, both $\Ext_H(k,V)$ and $\Ext_{V \# H}(V,V)$ are Gerstenhaber algebras 
as $V$ is automatically a braided commutative Yetter-Drinfel'd algebra over $V \# H$, 
see the comment below Definition \ref{fiocchi}. 
For example, if $H$ is a finite dimensional Hopf algebra over a field with bijective antipode 
and $H^*$ its $k$-linear dual, then the Heisenberg double $\cH(H^*)$ is a braided commutative Yetter-Drinfel'd algebra over the Drinfel'd double $\cD(H)$, see \cite{Sem:HDHBSAABCYDMAOTDD}; hence, $\Ext_{\cH(H^*) \# \cD(H)}(\cH(H^*), \cH(H^*))$ and $\Ext_{\cD(H)}(k, \cH(H^*))$ are Gerstenhaber algebras. 
On the other hand, the Yetter-Drinfel'd algebra structure on $\cH(H^*)$ over $\cD(H)$ arises (see [{\em op.~cit.}] again) from the construction of $\cH(H^*)$ as a braided product of ${H^*}_\coop$ and $H$, which both are braided commutative Yetter-Drinfel'd algebras over $\cD(H)$ as well. In particular,
both $\Ext_{\cD(H)}(k, H)$ and $\Ext_{\cD(H)}(k, {H^*}_\coop)$ carry Gerstenhaber algebra structures, which 
can be transferred to certain Gerstenhaber-Schack cohomology groups $H_{\scriptscriptstyle{\rm GS}}(.,.)$: for two Hopf bimodules (or {\em tetramodules}) $M, N$, one has an isomorphism \cite{Tai:CTOHBACP}
$$
\Ext^\bull_{\cD(H)}(M^{\scriptscriptstyle{\rm coinv}} , N^{\scriptscriptstyle{\rm coinv}}) \simeq H_{\scriptscriptstyle{\rm GS}}(M,N).
$$ 
Applying this to our situation above means 
$
\Ext_{\cD(H)}(k, {H^*}_\coop) \simeq  H_{\scriptscriptstyle{\rm GS}}(H,H \otimes {H^*}_\coop), 
$
and therefore yields a Gerstenhaber algebra structure on 
$$
H_{\scriptscriptstyle{\rm GS}}(H,H \otimes {H^*}_\coop) \simeq H_{\scriptscriptstyle{\rm GS}}(H,\End_k(H)). 
$$
Observe here that the Gerstenhaber algebra structure on $H_{\scriptscriptstyle{\rm GS}}(H,H)$ obtained in \cite{FarSol:GSOTCOHA} by similar arguments is trivial, as shown in \cite{Tai:IHBCOIDHAAGCOTYP}.

We leave all details of this example paragraph to a future extension, also in the hope of discovering genuinely new examples which do not resemble or generalise classical Gerstenhaber algebra structures.
\end{example}

\subsection{$C^\bull_{\rm co}(U,N)$ as a cyclic operad with multiplication}
\label{acquedottofelice}

Let again $N$ be a left $U$-module (with action denoted by juxtaposition) 
which is simultaneously a left $U$-comodule with coaction $\gD_\enne: n \mapsto n_{(-1)} \otimes_\ahha n_{(0)}$ 
such that the underlying induced left $\Ae$-module structures coincide. 
As before, we assume that with respect to this left $\Ae$-module structure $N$ is an $A$-ring with multiplication denoted by $(n,n') \mapsto n \cdot_\enne n'$. Observe once more that by these requirements in particular Eq.~\rmref{normalny} holds.
We then define, as a generalisation of \cite[p.~65]{GerSch:ABQGAAD},
$$
\circ_i : \corm p \otimes_k \corm q \rightarrow 
        \corm {|p+q|}, \quad i = 1, \ldots, p,
$$
by
\begin{equation}
\label{maxdudler2}
\begin{split}
(&u^1, \ldots, u^p, n) \circ_i (v^1, \ldots, v^q, n') \\
&:= (u^1, \ldots, u^{i-1}, u^{i}_{(1)}v^1, \ldots, u^{i}_{(q)}v^q, (u^{i}_{(q+1)}n')_{(-p+i)}u^{i+1}, \\
&\qquad \qquad \quad 
\ldots, (u^{i}_{(q+1)}n')_{(-1)}u^p,  (u^{i}_{(q+1)} n')_{(0)} \cdot_\enne n),
\end{split}
\end{equation}
which 
is well-defined by \rmref{maotsetung}, \rmref{tellmemore}, \rmref{auchnochnicht}, and \rmref{normalny}.
For zero cochains, that is, elements $n \in N$, we define $n \circ_i (v^1, \ldots, v^q, n') = 0$ for all $i$ and
all $(v^1, \ldots, v^q, n') \in \corm q$, whereas 
\begin{equation*}
\begin{split}
(&u^1, \ldots, u^p, n) \circ_i n' \\
&:=  (u^1, \ldots, u^{i-1}, (u^{i}n')_{(-p+i)}u^{i+1}, \ldots, (u^{i}n')_{(-1)}u^p,  (u^{i}n')_{(0)} \cdot_\enne n) \in \corm {p-1}.
\end{split}
\end{equation*}
The {distinguished element}, {\em i.e.}, the {operad multiplication} \rmref{distinguished, I said 1} is here
\begin{equation}
\label{distinguished, I said 3}
        \mu := (1_\uhhu, 1_\uhhu, 1_\enne) \in \corm 2,
\end{equation}
and also set
\begin{equation}
\label{habemuspapam2}
\mathbb{1} := (1_\uhhu, 1_\enne) \in \corm 1 \quad \mbox{along with} \quad  e := 1_\enne \in \corm 0. 
\end{equation}
With \rmref{maxdudler2} and \rmref{distinguished, I said 3}, the cup product \rmref{cupco} explicitly becomes
$$
(u^1, \ldots, u^p, m) \smallsmile (v^1, \ldots, v^q, n) = (u^1, \ldots, u^p, m_{(-q)}v^1, \ldots, m_{(-1)}v^q, m_{(0)} \cdot_\enne n). 
$$
Observe that the differential \rmref{appolloni2} formed by the cofaces in \rmref{anightinpyongyang} can be expressed as
$$
\gd(z,n)  = \{\mu, (z,n)\},
$$ 
where the right hand side is again defined as in \rmref{zugangskarte} and we used the abbreviation 
$(z,n) := (u^1, \ldots, u^p, n)$. 

\begin{theorem}
\label{clairefontaine1b}
Let $N$ be a braided commutative Yetter-Drinfel'd algebra over a left bialgebroid $U$.
Then $\corm \bull$ with the structure given in \rmref{maxdudler2}--\rmref{habemuspapam2} yields an operad with multiplication.
\end{theorem} 

\begin{proof}
The proof relies on straightforward computations analogous to those in the proof of Theorem \ref{clairefontaine1a}, which is why it is omitted.
\end{proof}

By Theorem \ref{customerscopy} one deduces at once:
 
\begin{corollary}
\label{capodistria}
With the assumptions of Theorem \ref{clairefontaine1b} on $N$, the cohomology groups $H^\bull_{\rm co}(U,N)$ carry the structure of a Gerstenhaber algebra.
In particular, if
$U_\ract$ is flat as a right $A$-module, then $\Cotor^\bull_U(A,N)$ is a Gerstenhaber algebra.
\end{corollary}

\begin{rem}
\label{tariffaregionalelazio2}
Taking for $N$ the base algebra $A$, the conditions mentioned in the preceding proposition can, as said before, always be fulfilled for the canonical $U$-action and $U$-coaction \rmref{stabatmater1} resp.\ \rmref{stabatmater2} (and the formulae notably simplify: the cup product, for example, just becomes the tensor product). 
However, this is not the case if one twists the coaction $a \mapsto s(a)\gs$ on $A$ by a grouplike element $\gs \in U$ (for the underlying $A$-coring of $U$): unless $\gs = 1$, the base algebra $A$ is already not a comodule algebra any more, not even for a bialgebra (where $A=k$). 
This seems to correspond to Menichi's conjecture in \cite[\S10]{Men:CMCMIALAM} that $C^\bull_{\rm co}(H, {}^\gs k_\gd)$, where $H$ is a Hopf algebra over a commutative ring $k$ and $(\gd, \gs)$ is a modular pair in involution, does {\em not} give an operad with multiplication unless $\gs = 1$.
\end{rem}

\subsection{Batalin-Vilkovisky algebra structures on $\Cotor$}
In this section, we want to investigate how the operad with multiplication from Theorem \ref{clairefontaine1b} can be given the structure of a cyclic operad with multiplication; or, equivalently, under which conditions the Gerstenhaber algebra in Corollary \ref{capodistria} becomes a Batalin-Vilkovisky algebra. In particular, as seen in Proposition \ref{kamille} and Eq.~\rmref{casadilivia}, respectively, whenever a cyclic operation needs to be defined, the left $U$-comodule $N$ has to be given the structure of a {\em right} $U$-action. This can be obtained if $U$ is not merely a left bialgebroid but rather a left Hopf algebroid and if moreover the base algebra $A$ is itself a right $U$-module. One can then show:

\begin{theorem}
\label{clairefontaine2}
Let $N \in \yd$ be a braided commutative Yetter-Drinfel'd algebra over a left Hopf algebroid $U$.
Assume that $A$ carries an ($\Ae$-balanced) right $U$-action 
$A \otimes_\Aee \due U \lact \ract \to A,\ a \otimes_\Aee u \mapsto a \Yleft
u$, and define the corresponding {\em right character}
\begin{equation}
\label{rosti}
\pl: U \to A, \quad \pl u := 1_\ahha \Yleft u, \quad \mbox{for all} \ u \in U.
\end{equation} 
If $A$ with this right action and the canonical coaction \rmref{stabatmater2} fulfills \rmref{huhomezone2}, {\em i.e.}, is anti Yetter-Drinfel'd, then the left $U$-comodule $N$ equipped with the right $U$-action  
\begin{equation}
\label{viaappia}
N \otimes U \to N, \quad n \otimes u \mapsto nu:= (u_-n) \ract \pl u_+
\end{equation}
is also an {\em anti} Yetter-Drinfel'd module. 
If $N$ is moreover stable
with respect to \rmref{viaappia} and its given left $U$-comodule structure, 
then $\corm \bull$ with the structure given in \rmref{maxdudler2}--\rmref{habemuspapam2} and the cocyclic operator in \rmref{anightinpyongyang} yields a cyclic operad with multiplication. 
\end{theorem}

\begin{rem}
The fact that the YD module $N$ equipped with the right $U$-action \rmref{viaappia} becomes an aYD if $A$ itself is aYD 
is due to the more general fact that the tensor product (over $A$) of an aYD module with a YD module yields an aYD module again (in other words, $\ayd$ forms a module category over $\yd$). 
Let us also remark that a straightforward check proves that if the base algebra $A$ of a left Hopf algebroid $U$ is aYD, then $(U, A, \Aop)$ defines a {\em full} Hopf algebroid
 with involutive antipode in the sense of B\"ohm-Szlach\'anyi (see, {\em e.g.}, \cite{Boe:HA} for the precise definition), the antipode (and its inverse) given by 
$$
Su := u_- \ract \pl u_+, \qquad \forall \ u \in U.
$$ 
One could also be tempted to think that starting directly with aYD modules simplifies matters; however, since there is no corresponding monoidal category, 
there is no such thing as a monoid in that category, which in turn would be necessary for the Gerstenhaber algebra structure.   
\end{rem}

\begin{proof}[Proof of Theorem \ref{clairefontaine2}]
The first claim that $N \in \yd$ equipped with the action \rmref{viaappia} becomes aYD is proven as follows:
if $A$ is aYD w.r.t.\ the mentioned right action and the left $U$-coaction $\gD_\ahha$ from \rmref{stabatmater2}, then one has
$$
s(a \Yleft u) \otimes_\ahha 1_\ahha = \gD_\ahha(a \Yleft u) = u_- s(a) u_{+(1)} \otimes_\ahha (1_\ahha \Yleft u_{+(2)}), 
$$
hence
$$
s(\pl u) = (u_- u_{+(1)}) \ract \pl u_{+(2)}. 
$$
Therefore,
\begin{equation*}
\begin{split}
\gD_\enne(nu) &\overset{\phantom{\scriptscriptstyle\rmref{viaappia}}}{=} (u_-n)_{(-1)} \bract \pl u_+ \otimes_\ahha (u_-n)_{(0)} \\ 
&\overset{\phantom{\scriptscriptstyle\rmref{viaappia}}}{=} 
(\pl u_{++(2)} \blact (u_-n)_{(-1)}) u_{+-} u_{++(1)} \otimes_\ahha  (u_-n)_{(0)} \\
&\overset{\scriptscriptstyle\rmref{huhomezone1b}}{=}  
u_{-+(1)} n_{(-1)} u_{--} u_{+-} u_{++(1)} \otimes_\ahha  (u_{-+(2)} n_{(0)}) \ract \pl u_{++(2)} \\
&\overset{\scriptscriptstyle\rmref{Sch5}}{=}  
u_{-(1)+(1)} n_{(-1)} u_{-(1)-} u_{-(2)} u_{+(1)} \otimes_\ahha  (u_{-(1)+(2)} n_{(0)}) \ract \pl u_{+(2)} \\
&\overset{\scriptscriptstyle\rmref{Sch3}}{=}  
u_{-(1)} n_{(-1)} u_{+(1)} \otimes_\ahha  (u_{-(2)} n_{(0)}) \ract \pl u_{+(2)} \\
&\overset{\scriptscriptstyle\rmref{Sch5}}{=}  
u_{-} n_{(-1)} u_{++(1)} \otimes_\ahha  (u_{+-} n_{(0)}) \ract \pl u_{++(2)} \\
&\overset{\scriptscriptstyle\rmref{Sch4}}{=}  
u_{-} n_{(-1)} u_{+(1)} \otimes_\ahha  (u_{+(2)-} n_{(0)}) \ract \pl u_{+(2)+}, 
\end{split}
\end{equation*}
and the last line is precisely the aYD condition \rmref{huhomezone2} for the right $U$-action \rmref{viaappia} on the left $U$-comodule $N$.

To prove that  $\corm \bull$ forms a cyclic operad with multiplication, we have to verify the conditions \rmref{superfluorescent1} and \rmref{superfluorescent2}: we only prove the first identity, the rest being either similar or obvious. One computes for $p, q \geq 1$:
\begin{footnotesize}
\begin{equation*}
\begin{split}
&\tau((u^1, \ldots, u^p, n) \circ_1 (v^1, \ldots, v^q, n')) \\
&\overset{\scriptscriptstyle\rmref{maxdudler2}}{=} 
\tau\big(u^{1}_{(1)}v^1, \ldots, u^{1}_{(q)}v^q, (u^{1}_{(q+1)}n')_{(-p+1)}u^{2}, 
\ldots, (u^{1}_{(q+1)}n')_{(-1)}u^p,  (u^{1}_{(q+1)} n')_{(0)} \cdot_\enne n\big) 
\\
&\overset{\scriptscriptstyle\rmref{anightinpyongyang}}{=}   
\Big(v^1_{-(1)}u^{1}_{(1)-(1)}u^1_{(2)}v^2, \ldots, v^1_{-(q-1)}u^{1}_{(1)-(q-1)}u^{1}_{(q)}v^q, v^1_{-(q)}u^{1}_{(1)-(q)}(u^{1}_{(q+1)}n')_{(-p+1)}u^{2}, 
\\
&\qquad \qquad \quad 
\ldots, v^1_{-(|q+p|-1)}u^{1}_{(1)-(|q+p|-1)} 
(u^{1}_{(q+1)}n')_{(-1)}u^p,  \\
&\qquad \qquad \qquad 
v^1_{-(|q+p|)}u^{1}_{(1)-(|q+p|)} \big((u^{1}_{(q+1)} n')_{(0)} \cdot_\enne n\big)_{(-1)}, 
\big((u^{1}_{(q+1)} n')_{(0)} \cdot_\enne n\big)_{(0)} u^1_{(1)+} v^1_+\Big) 
\\
%
&\overset{\scriptscriptstyle\rmref{corsasemplice}, \rmref{huhomezone1b}}{=}  
\Big(v^1_{-(1)}u^{1}_{(1)-(1)}u^1_{(2)}v^2, \ldots, v^1_{-(q-1)}u^{1}_{(1)-(q-1)}u^{1}_{(q)}v^q, 
\\
&\qquad \qquad \quad 
v^1_{-(q)}u^{1}_{(1)-(q)}u^{1}_{(q+1)+(1)}n'_{(-p)}u^1_{(q+1)-(1)}u^{2}, 
\\
&\qquad \qquad \quad 
\ldots, v^1_{-(|q+p|-1)}u^{1}_{(1)-(|q+p|-1)} 
u^{1}_{(q+1)+(p-1)}n'_{(-2)}u^1_{(q+1)-(p-1)}u^p,  \\
&\qquad \qquad \qquad 
v^1_{-(|q+p|)}u^{1}_{(1)-(|q+p|)} u^{1}_{(q+1)+(p)} n'_{(-1)} u^1_{(q+1)-(p)} n_{(-1)},
 \\
&\qquad \qquad \quad 
(u^{1}_{(q+1)+(p+1)} n'_{(0)} \cdot_\enne n_{(0)}) u^1_{(1)+} v^1_+\Big)
\\
&\overset{\scriptscriptstyle\rmref{Sch4}}{=}  
\Big(v^1_{-(1)}u^{1}_{+(1)-(1)}u^1_{+(2)}v^2, \ldots, v^1_{-(q-1)}u^{1}_{+(1)-(q-1)}u^{1}_{+(q)}v^q, 
\\
&\qquad \quad
v^1_{-(q)}u^{1}_{+(1)-(q)}u^{1}_{+(q+1)}n'_{(-p)}u^1_{-(1)}u^{2}, 
\\
&\qquad  
\ldots, v^1_{-(|q+p|-1)}u^{1}_{+(1)-(|q+p|-1)} 
u^{1}_{+(|q+p|)}n'_{(-2)}u^1_{-(p-1)}u^p,  \\
&\qquad \qquad
v^1_{-(|q+p|)}u^{1}_{+(1)-(|q+p|)} u^{1}_{+(q+p)} n'_{(-1)} u^1_{-(p)} n_{(-1)},
(u^{1}_{+(q+p+1)} n'_{(0)} \cdot_\enne n_{(0)}) u^1_{+(1)+} v^1_+\Big)
\\
&\overset{\scriptscriptstyle\rmref{Sch3}}{=}  
\Big(v^1_{-(1)}v^2, \ldots, v^1_{-(q-1)}v^q, 
v^1_{-(q)}n'_{(-p)}u^1_{-(1)}u^{2}, 
\ldots, v^1_{-(|q+p|-1)}n'_{(-2)}u^1_{-(p-1)}u^p,  
\\
&\qquad \qquad \qquad 
v^1_{-(|q+p|)} n'_{(-1)} u^1_{-(p)} n_{(-1)},
(u^{1}_{+(2)} n'_{(0)} \cdot_\enne n_{(0)}) u^1_{+(1)} v^1_+\Big)
\\
&\overset{\scriptscriptstyle{\rmref{viaappia}, \rmref{zilvesta}}}{=}  
\Big(v^1_{-(1)}v^2, \ldots, v^1_{-(q-1)}v^q, 
v^1_{-(q)}n'_{(-p)}u^1_{-(1)}u^{2}, 
\ldots, v^1_{-(|q+p|-1)}n'_{(-2)}u^1_{-(p-1)}u^p,  
\\
&\qquad \qquad \qquad 
v^1_{-(|q+p|)} n'_{(-1)} u^1_{-(p)} n_{(-1)},
\big(u^1_{+(1)-(1)}u^{1}_{+(2)} n'_{(0)} \cdot_\enne (u^1_{+(1)-(2)} n_{(0)})\ract \pl u^1_{+(1)+}\big) v^1_+\Big)
\\ 
&\overset{\scriptscriptstyle\rmref{Sch5}, \rmref{Sch3}}{=}  
\Big(v^1_{-(1)}v^2, \ldots, v^1_{-(q-1)}v^q, 
v^1_{-(q)}n'_{(-p)}u^1_{-(1)}u^{2}, 
\ldots, v^1_{-(|q+p|-1)}n'_{(-2)}u^1_{-(p-1)}u^p,  
\\
&\qquad \qquad \qquad 
v^1_{-(|q+p|)} n'_{(-1)} u^1_{-(p)} n_{(-1)},
\big(n'_{(0)} \cdot_\enne (u^1_{+-} n_{(0)})\ract \pl u^1_{++}\big) v^1_+\Big)
\\ 
&\overset{\rmref{viaappia}, \rmref{circolodegliartisti}}{=}  
\Big(v^1_{-(1)}v^2, \ldots, v^1_{-(q-1)}v^q, 
v^1_{-(q)}n'_{(-p-1)}u^1_{-(1)}u^{2}, 
\ldots, v^1_{-(|q+p|-1)}n'_{(-2)}u^1_{-(p-1)}u^p,  
\\
&\qquad \qquad \qquad 
v^1_{-(|q+p|)} n'_{(-1)} u^1_{-(p)} n_{(-1)}, \big(v^1_{+-}\big(n'_{(-1)}(n_{(0)}u^1_+) \cdot_\enne n'_{(0)}\big)\big) \ract \pl v^1_{++} \Big)
\end{split}
\end{equation*}
\begin{equation*}
\begin{split}
&\overset{\scriptscriptstyle\rmref{Sch5}}{=}  
\Big(v^1_{-(1)}v^2, \ldots, v^1_{-(q-1)}v^q, 
v^1_{-(q)}n'_{(-p-1)}u^1_{-(1)}u^{2}, 
\ldots, v^1_{-(|q+p|-1)}n'_{(-2)}u^1_{-(p-1)}u^p,  
\\
&\qquad \qquad \qquad 
v^1_{-(|q+p|)} n'_{(-1)} u^1_{-(p)} n_{(-1)}, \big(v^1_{-(q+p)}\big(n'_{(-1)}(n_{(0)}u^1_+) \cdot_\enne n'_{(0)}\big)\big) \ract \pl v^1_{+} \Big)
\\
&\overset{\scriptscriptstyle\rmref{zilvesta}, \rmref{bilet}}{=}  
\Big(v^1_{-(1)}v^2, \ldots, v^1_{-(q-1)}v^q, 
v^1_{-(q)}n'_{(-p-1)}u^1_{-(1)}u^{2}, 
\ldots, v^1_{-(|q+p|-1)}n'_{(-2)}u^1_{-(p-1)}u^p,  
\\
&\qquad \qquad \qquad 
v^1_{-(|q+p|)} n'_{(-1)} u^1_{-(p)} n_{(-1)}, (v^1_{-(q+p)}n'_{(-1)}(n_{(0)}u^1_+)) \cdot_\enne (v^1_{-(q+p+1)}  n'_{(0)}) \ract \pl v^1_{+} \Big)
\\
&\overset{\scriptscriptstyle\rmref{Sch5}, \rmref{viaappia}}{=}  
\Big(v^1_{-(1)}v^2, \ldots, v^1_{-(q-1)}v^q, 
v^1_{-(q)}n'_{(-p-1)}u^1_{-(1)}u^{2}, 
\ldots, v^1_{-(|q+p|-1)}n'_{(-2)}u^1_{-(p-1)}u^p,  
\\
&\qquad \qquad \qquad 
v^1_{-(|q+p|)} n'_{(-1)} u^1_{-(p)} n_{(-1)}, \big(v^1_{-(q+p)}n'_{(-1)}(n_{(0)}u^1_+) \big) \cdot_\enne n'_{(0)}v^1_+\Big)
\\
&\overset{\scriptscriptstyle\rmref{maxdudler2}}{=}  
(v^1_{-(1)}v^2, \ldots, v^1_{-(q-1)} v^q, v^1_{-(q)} n'_{(-1)}, n'_{(0)}v^1_{+} ) \\
& \qquad \qquad 
\circ_q 
(u^1_{-(1)}u^2, \ldots, u^1_{-(p-1)} u^p, u^1_{-(p)} n_{(-1)}, n_{(0)}u^1_{+})
\\
&\overset{\scriptscriptstyle\rmref{anightinpyongyang}}{=}  
\tau(v^1, \ldots, v^q, n') \circ_q \tau(u^1, \ldots, u^p, n),
\end{split}
\end{equation*}
\end{footnotesize}
which concludes the proof of the theorem.
\end{proof}

By Theorem \ref{holl} one then immediately has:

\begin{corollary}
\label{clairefontaine2a}
With the assumptions of Theorem \ref{clairefontaine2} on $N$ and $A$, the cohomology groups $H^\bull_{\rm co}(U,N)$ carry the structure of a Batalin-Vilkovisky algebra.
In particular, if  
$U_\ract$ is flat as a right $A$-module, then $\Cotor^\bull_U(A,N)$ is a Batalin-Vilkovisky algebra.
\end{corollary}

\subsection{Maps of Gerstenhaber and Batalin-Vilkovisky algebras}
In this section, we discuss how the Batalin-Vilkovisky algebra $\Cotor^\bull_{V\!L}(A,A_\pl)$ over the universal enveloping algebra $V\!L$ of a Lie-Rinehart algebra $(A,L)$ as in Example \ref{castelliromani} is related to the classical Batalin-Vilkovisky algebra structure on the exterior algebra $\bigwedge^\bull_\ahha \! L$ as discussed in \cite{Hue:LRAGAABVA}.

\subsubsection{The generalised Schouten bracket}
For an arbitrary 
Gerstenhaber algebra $V^\bull$, the pair $(V^0, V^1)$ of its degree zero and degree one part forms a Lie-Rinehart algebra \cite[p.~67]{GerSch:ABQGAAD}. The forgetful functor $V^\bull \to (V^0, V^1)$ from Gerstenhaber algebras to Lie-Rinehart algebras has as a left adjoint (see \cite[Thm.~5]{GerSch:ABQGAAD}) given by the map
$(V^0,V^1) \to \bigwedge^\bull_{\scriptscriptstyle V^0} \! V^1$, where the exterior algebra $\bigwedge^\bull_{\scriptscriptstyle V^0} \! V^1$ of $V^1$ over $V^0$ is equipped with the Schouten
bracket:
\begin{footnotesize}
\begin{equation}
\label{labicana0}
\begin{split}
&[X_1 \wedge  \cdots \wedge X_p, Y_1 \wedge \cdots \wedge Y_q] \\
&= \sum^p_{i=1} \sum^q_{j=1} (-1)^{i+j + |p||q|} [X_i, Y_j] 
\wedge X_1 \wedge \cdots \wedge \widehat{X}_i \wedge \cdots \wedge X_p  
\wedge Y_1 \wedge \cdots \wedge \widehat{Y}_j \wedge \cdots \wedge Y_q.  
\end{split}
\end{equation}
\end{footnotesize}
Here the symbol $\widehat{\phantom{x}}$ denotes omission, as usual. Correspondingly, for every Gerstenhaber algebra $(V^\bull, \smallsmile)$ there is a universal map of Gerstenhaber algebras,
\begin{equation}
\label{labicana}
\textstyle\bigwedge^\bull_{\scriptscriptstyle V^0} \! V^1 \to V^\bull, \quad 
 X_1 \wedge \cdots \wedge X_p \mapsto  X_1 \smallsmile \cdots \smallsmile X_p.  
\end{equation}

\subsubsection{Batalin-Vilkovisky algebra structures and Lie-Rinehart algebras}
Assume for the rest of this section that $(A,L)$ is a Lie-Rinehart algebra in which $L$ is projective as an $A$-module. 
By a direct computation or by applying \cite[Thm.~2.13]{KowPos:TCTOHA} to cocommutative left bialgebroids, one has
\begin{equation*}
\begin{array}{rcl}
\Cotor^0_{V\!L}(A,A) &=& A,   \\
\Cotor^1_{V\!L}(A,A) &=& P(V\!L) = L,  
\end{array}
\end{equation*}
where $P(V\!L)$ denotes the set of primitive elements of $V\!L$, and where the last equation in the second line is a consequence of the Milnor-Moore theorem for cocommutative bialgebroids (see, {\em e.g.}, \cite{MoeMrc:OTUEAOALA}).
Since $\Cotor^\bull_{V\!L}(A,A)$ is a Gerstenhaber algebra as seen in Corollary \ref{clairefontaine2a}, 
one consequently has a canonical morphism $\textstyle\bigwedge^\bull_\ahha \! L \to \Cotor^\bull_{V\!L}(A,A)$ of Gerstenhaber algebras as generally given by the universal map \rmref{labicana}.

On the other hand, in \cite{Hue:LRAGAABVA} 
it is shown that there is a bijective correspondence between right $V\!L$-module structures on $A$ and operators of square zero that generate the Schouten bracket \rmref{labicana0}: if $(a,u) \mapsto a \Yleft u$ for all $u \in V\!L$ and $a \in A$ is a right 
$V\!L$-action on $A$ and 
$\pl: V\!L \to A, \ \pl u := 1_\ahha \Yleft u$ 
its associated generalised right character analogous to \rmref{rosti}, the map 
$b_L:\textstyle\bigwedge^\bull_\ahha \! L\to\textstyle\bigwedge^{\bull-1}_\ahha \! L$ 
defined by
\begin{equation*}
\begin{split}
b_L(X_1\wedge\cdots\wedge X_n):=&\sum_{i=1}^n(-1)^{i+1} 
\pl(X_i) X_1\wedge\cdots\wedge\hat{X}_i\wedge\cdots\wedge X_n\\
&+\sum_{i<j}(-1)^{i+j}[X_i,X_j]\wedge X_1\wedge\cdots\wedge\hat{X}_i
\wedge\cdots\wedge\hat{X}_j\wedge\cdots\wedge X_n
\end{split}
\end{equation*}
generates the Schouten bracket \rmref{labicana0} and has the property $b_L^2 = 0$, see \cite[Thm.~1]{Hue:LRAGAABVA}.

The following theorem is a generalisation of \cite[Prop.~65 \& Thm.~66]{Men:CMCMIALAM}
from Lie algebras to Lie-Rinehart algebras (or Lie algebroids):

\begin{theorem}
Let $\Q \subseteq k$ and $(A,L)$ be a Lie-Rinehart algebra with $L$ projective as an $A$-module.
Then the universal map $\bigwedge_\ahha^\bull \! L \to \Cotor_{V\!L}^\bull(A,A)$ of Gerstenhaber algebras arising from \rmref{labicana} is an isomorphism. 
If $A$ is moreover a right $V\!L$-module, then this morphism is a map of Batalin-Vilkovisky algebras.
\end{theorem}
\begin{proof}
This follows at once by using \rmref{cotor} along with applying \cite[Thm.~3.13]{KowPos:TCTOHA}:
there it is proven that
the antisymmetrisation map
\begin{equation}
\label{luzzi}
X_1 \wedge \cdots \wedge X_n \mapsto
\frac{1}{n!}\sum_{\gs \in S_n} (-1)^\gs (X_{\gs(1)}, \ldots, X_{\gs(n)})
\end{equation}
defines a morphism of
mixed complexes
\begin{equation*}
{\rm Alt}:\left( \textstyle\bigwedge^\bull_\ahha \! L,0,b_L\right)\to
\left(C^\bull_{\rm co}(V\!L),\gb,B\right),
\end{equation*}
where the right hand side 
refers to the complex defined by \rmref{anightinpyongyang} together with the differential \rmref{appolloni2} and the analogue of \rmref{extra} for cocyclic modules. On cohomology, this morphism induces 
a natural isomorphism 
between $\textstyle\bigwedge^\bull_\ahha \! L$  and $\Cotor^\bull_{V\!L}(A,A)$. Since the universal map \rmref{labicana} of Gerstenhaber algebras coincides, when descending to cohomology, with the antisymmetrisation map \rmref{luzzi} as the cup product \rmref{cupco} becomes simply the tensor product for $N=A$, the first claim follows. The second statement follows by the first along with the equation ${\rm Alt} \circ b_L = B \circ {\rm Alt}$ mentioned just above. 
\end{proof}

\section{Poisson Bialgebroids and their (co)homology}
\label{vipiteno}

Definition \ref{shakta} below connects the idea of a Poisson structure to that of a distinguished element, {\em i.e.}, can be understood as a generalised operad multiplication for the operad $C^\bull(U,A)$ as given in \S\ref{palermo}. The examples in this section will show that this approach conceptually unites, for example, Hochschild with Poisson homology (resp.\ cyclic homology with cyclic Poisson homology).

\subsection{The noncommutative calculus structure on left Hopf algebroids}
\label{nmd}
For later use, we need to recall from \cite[\S4]{KowKra:BVSOEAT}
the Hopf algebroid generalisation of the {\em cap product} and
\emph{Lie derivative} along with its properties; see [{\em loc.~cit.}] for all details and proofs in this subsection. These operators together with the cyclic differential $B$ from \rmref{extra} form a {\em noncommutative differential calculus} in the sense of Nest-Tsygan (see \cite{NesTsy:OTCROAA}, {\em cf.} also \cite{TamTsy:NCDCHBVAAFC}), which was the main point in \cite{KowKra:BVSOEAT}.

\begin{definition}[\cite{KowKra:BVSOEAT}]
Let $U$ be a left Hopf algebroid and $M$ a right $U$-module left $U$-comodule such that the induced left $A$-actions on $M$ coincide, and let  $\varphi \in C^p(U,A)$ be a $p$-cochain.
\begin{enumerate}
\compactlist{99}
\item
The {\em cap product} $\iota_\gvf := \gvf \smallfrown \cdot$ of  
$\varphi$ with 
$(m, x) \in C_n(U,M)$ is defined by
\begin{equation*}
\label{alles4}
\gvf \smallfrown (m,x) :=  (m, u^1, \ldots, u^{n-p-1}, 
        \varphi(u^{n-|p|}, \ldots, u^n) \blact u^{n-p}),
\end{equation*}
where as in Proposition~\ref{kamille} the abbreviation 
$(m,x) :=(m, u^1, \ldots, u^n)$ 
is used.
\item
The {\em Lie derivative} 
$$
        \lie_\varphi : C_n(U,M) \rightarrow 
        C_{n-|p|}(U,M)
$$
along $\gvf$ in degree $n$ with $p < n+1$ is defined to be 
\begin{equation}
\label{messagedenoelauxenfantsdefrance2}
        \lie_\varphi := 
        \sum^{n-|p|}_{i=1} 
        (-1)^{\theta^{n,p}_{i}} 
        \ttt^{n - |p| - i} \, \DD'_\varphi \, \ttt^{i+p} 
        + 
        \sum^{p}_{i=1} 
        (-1)^{\xi^{n,p}_{i}} 
        \ttt^{n-|p|} \, \DD'_\varphi \, \ttt^i,
\end{equation}
where the signs are given by
$        \theta^{n,p}_{i} 
        := |p|(n-|i|) 
$
 and        
$
\xi^{n,p}_{i} 
        := n|i| + |p|.
$
In case $p = n+1$, set
$$
        \lie_\varphi := 
        (-1)^{|p|}\iota_\varphi \, \BB,
$$
and for $p > n+1$, we define 
$\lie_\varphi := 0$. 
\end{enumerate}
\end{definition}
To simplify future reference, we call the first sum in \rmref{messagedenoelauxenfantsdefrance2} the {\em untwisted part} and the second sum the {\em twisted part} of $\lie_\gvf$. As a shorthand, we will write
\begin{equation}
\label{travaux}
\lie_\gvf = \lie^{\scriptscriptstyle{\rm untw}}_\gvf + \lie^{\scriptscriptstyle{\rm tw}}_\gvf. 
\end{equation}

We list a few useful facts about the triple of operators $(\lie_\gvf, \iota_\gvf, B)$; 
for simplicity, let $M$ be an SaYD module in the following theorem and let $[.,.]$ denote the graded commutator in all what follows.

\begin{theorem}[\cite{KowKra:BVSOEAT}]
\label{feinefuellhaltertinte}
Let $M$ be an SaYD module. Then the  triple $(C_\bull(U,M),\bb,\smallfrown)$ is a left DG module over 
$(C^\bull(U,A),\delta,\smallsmile)$, {\em i.e.}, for any cochain $\varphi \in C^\bull(U,A)$ 
\begin{equation}
\label{mulhouse2}
        \iota_\varphi \, \iota_\psi = \iota_{\varphi \smallsmile \psi} 
      \quad \mbox{and} \quad  [\bb, \iota_\varphi] = \iota_{\gd\varphi}
\end{equation} 
holds, where $\smallsmile$ is given by \rmref{cupco}.
On the other hand, the Lie derivative $\lie$ defines a DG 
Lie algebra representation of $(C^\bull(U,M)[1], \{.,.\})$: for another cochain $\psi \in C^\bull(U,A)$, 
we have, as operators on $C_\bull(U,M)$,
\begin{equation}
\label{weimar}
[\lie_\varphi, \lie_\psi] = \lie_{\{\varphi, \psi\}},
\end{equation}
where the bracket on the right hand side is the Gerstenhaber bracket
\rmref{zugangskarte}.
Moreover,
\begin{equation}
\label{alles1}
[\bb, \lie_\varphi] + \lie_{\gd\varphi}= 0. 
\end{equation}
If furthermore 
$\varphi \in C^p(U,A)$, $\psi \in C^q(U,A)$ are any two co{\em cycles}, 
the induced maps 
$$
\lie_\varphi: H_{\bull}(U,M)  \to H_{\bull-|p|}(U,M) \quad \mbox{and} \quad \iota_\psi:  H_{\bull}(U,M)  \to H_{\bull-q}(U,M) 
$$
are well-defined operators that turn $H_\bull(U,M)$ into a module over the Gerstenhaber algebra $H^\bull(U,A)$, that is, they
satisfy
\begin{equation*}
\label{radicale1}
[\iota_\psi, \lie_\varphi] = \iota_{\{\psi, \gvf\}}.
\end{equation*}
Finally, for a cocycle $\varphi \in \bar C^p(U,A)$ the {\em Cartan-Rinehart homotopy formula}
$$
\lie_\gvf = [B,\iota_\gvf]
$$
holds on $H_{\bull}(U,M)$.
\end{theorem}

With the help of the homotopy formula, we can easily 
prove
that for any cochain $\varphi \in \bar C^\bull(U,A)$ one has on the normalised complex $\bar 
C_\bull(U,M)$ 
\begin{equation}
\label{alles2}
        [\lie_\varphi, \BB] = 0.
\end{equation}

\begin{rem}
\label{hornet}
One can also obtain a homotopy formula more generally on the chain resp.\ cochain level. In this case, one has to apply a ``cyclic correction'' to the cap product $\iota_\gvf$ by an operator $S_\gvf: C_n(U,M) \to C_{n-p+2}(U,M)$ to take the full cyclic bicomplex into account; see \cite{KowKra:BVSOEAT} for all details (or \cite{Rin:DFOGCA, Get:CHFATGMCICH, NesTsy:OTCROAA} for the example of associative algebras). Setting $\mathbf{I}_\gvf := \iota_\gvf + S_\gvf$ for the ``cyclic cap product'' on the cyclic bicomplex and ${\bf B} := B + b$ for its differential, the Cartan-Rinehart homotopy formula for any cochain $\gvf \in \bar C^p(U,M)$ 
then reads
$$
\lie_\gvf = [{\bf B}, {\bf I}_\gvf] - {\bf I}_{\gd \gvf}.
$$
All the preceding statements can be moreover relaxed to the case where $M$ is merely a right $U$-module left $U$-comodule with compatible left $A$-actions instead of being an SaYD module. Since we shall not deal with this situation, we refer to \cite[\S4]{KowKra:BVSOEAT} 
for the details of this more general construction.
\end{rem}

\noindent {\it Special cases.}
We conclude this section by listing explicit expressions for the Lie derivative in two special situations: 
in case $M$ is SaYD, the Lie derivative along $\gvf \in C^p(U,A)$ can be written on $C_n(U,M)$ explicitly as 
\begin{footnotesize}
\begin{equation*}
\label{encarnit1}
\begin{split}
\lie_\gvf(m, x) = &\sum_{i=1}^{n-|p|} (-1)^{\theta^{n,p}_{i}} 
\big(m,u^1,\ldots, \DD_\gvf(u^i,\ldots, u^{i+|p|}), \ldots, u^n\big) \\
&\quad  + \sum^{|p|}_{i=0} (-1)^{\xi^{n,p}_{i+1}} \big(m_{(0)}u^1_{+(2)} \cdots u^{i}_{+(2)}, u^{i+1}_+, \ldots, u^{n-p+i}_+, \\
& \qquad \qquad \qquad \gvf(u^{n-|p|+i+1}_+, \ldots, u^n_+, u^n_- \cdots u^1_- m_{(-1)}, u^1_{+(1)}, \ldots, u^{i}_{+(1)}) \blact u^{n-|p|+i}_+\big).
\end{split}
\end{equation*}
\end{footnotesize}
In \S\ref{castropretorio} we will deal with a {\em commutative} left Hopf algebroid $U$. 
Unlike for general left Hopf algebroids, we have in this case canonical coefficients: 
here, the base algebra $M \! := \! A$ is automatically an SaYD module 
using the canonical left action \rmref{stabatmater1} as right action, along with the canonical left coaction \rmref{stabatmater2}. 
In this case, the Lie derivative simplifies to:
\begin{footnotesize}
\begin{equation}
\label{encarnita2}
\begin{split}
\lie_\gvf(x) &= \sum_{i=1}^{n-|p|} (-1)^{\theta^{n,p}_{i}} 
\big(u^1,\ldots, \DD_\gvf(u^i,\ldots, u^{i+|p|}), \ldots, u^n\big) \\
&\quad  + \sum^{|p|}_{i=0} (-1)^{\xi^{n,p}_{i+1}} \big(u^{i+1}_+, \ldots, u^{n-p+i}_+, 
\\& \qquad \qquad \qquad 
\gvf(u^{n-|p|+i+1}_+, \ldots, u^n_+, u^n_- \cdots u^1_-, u^1_{+}, \ldots, u^{i}_{+}) \blact u^{n-|p|+i}_+\big).
\end{split}
\end{equation}
\end{footnotesize}

\subsection{Poisson Bialgebroids}
\label{ronciglione}
Having recalled the technical machinery above, we are in a position to introduce the Poisson theory for bialgebroids:

\begin{definition}
\label{shakta}
A {\em (quasi-)triangular $r$-matrix} or {\em Poisson structure} for a left bialgebroid $U$ is a $2$-cocycle $\gt \in C^2(U,A)$ that fulfills 
\begin{equation}
\label{impregnante1}
\gt \circ_1 \gt = \gt \circ_2 \gt.
\end{equation}
A left bialgebroid $U$ is called {\em Poisson bialgebroid} if there is a triangular
$r$-matrix $\gt \in C^2(U,A)$. 
\end{definition}

Note that (if $k$ has characteristic different from two) Eq.~\rmref{impregnante1} is equivalent to 
\begin{equation}
\label{impregnante2}
\gt \bar\circ \gt = 0 = \{\gt, \gt\},
\end{equation}
as follows from the grading of the Gerstenhaber bracket. The terminology ``triangular $r$-matrix'' will be motivated in \S\ref{salice}.
One could, of course, also define Poisson structures for {\em right} bialgebroids but for shortage in terminology, we shall always mean left ones when speaking about Poisson structures.
At times, we denote a Poisson bialgebroid by $U^\gt$ if we want to underline the dependence of a certain construction or 
structure from the triangular $r$-matrix $\gt$.
Observe that every bialgebroid allows for at least one such triangular $r$-matrix given by the operad multiplication, {\em i.e.}, the distinguished element $\mu$ in \rmref{distinguished, I said 2} (for $M:=A$), which we will refer to as the {\em trivial} one.

\subsection{The Poisson bicomplex}
\label{fave}
Let $U$ be a Poisson bialgebroid with triangular $r$-matrix $\gt$ and $M$ an 
SaYD module.
Define
the operators
\begin{eqnarray}
\label{monti}
b^\gt: C_n(U,M) \to  C_{n-1}(U,M), & (m,x) & \mapsto - \cL_\gt(m,x), \\
\label{celio}
\gb^\gt: C^n(U,A) \to  C^{n+1}(U,A), & \gvf & \mapsto \{\gt, \gvf\}. 
\end{eqnarray}

These operators could be, of course, defined for any $2$-cochain but the crucial property here is:

\begin{lem}
If $\gt$ is a triangular $r$-matrix, one has
$$
b^\gt b^\gt = 0 \quad \hbox{as well as} \quad \gb^\gt \gb^\gt =0.
$$
\end{lem}
\begin{proof}
The first equation is immediately seen with \rmref{impregnante2} and the property \rmref{weimar} of the Lie derivative. 
The second equation follows again from \rmref{impregnante2} and the graded Jacobi identity for the Gerstenhaber bracket.
\end{proof}

Hence $\big(C_\bull(U,M), b^\gt\big)$ and $\big(C^\bull(U,A), \gb^\gt\big)$ are chain resp.\ cochain complexes, and 
we can define:

\begin{definition}
The homology of $\big(C_\bull(U,M), b^\gt\big)$ will be called {\em Poisson homology} of the bialgebroid $U$ with values in $M$ 
and denoted by $H^\gt_\bull(U,M)$. 
In a similar way, the {\em Poisson cohomology} $H_\gt^\bull(U,A)$ with coefficients in the base algebra $A$ is the cohomology of $\big(C^\bull(U,A), \gb^\gt\big)$.
\end{definition}

\begin{rem}
In view of Theorem \ref{clairefontaine1a}, one could, for any braided commutative Yetter-Drinfel'd algebra $N$, also introduce {\em $N$-valued Poisson structures} on bialgebroids and consequently also consider the Poisson cohomology $H^\bull_\gt(U,N)$ with coefficients in such a braided commutative Yetter-Drinfel'd algebra $N$. Also, as already stated in Remark \ref{hornet}, one could relax the assumptions on $M$ as to be a right $U$-module left $U$-comodule with compatible left $A$-action and proceed in the spirit of \cite{KowKra:BVSOEAT}. To keep the exposition simple we shall refrain from pursuing both these generalisations, and leave them to a future project.
\end{rem}

As follows from \rmref{alles2}, one has for every triangular $r$-matrix $\gt$ the identity 
\begin{equation}
\label{trieste}
b^\gt B + B b^\gt = 0
\end{equation}
on the normalised complex $\bar C_\bull(U, M)$,
and therefore:

\begin{prop}
\label{c'est-a-dire}
The triple $\big(\bar C_\bull(U,M), b^\gt, B\big)$ forms a mixed complex.
\end{prop}

This enables us to define:

\begin{definition}
We call the cyclic homology of the mixed complex $\big(\bar C_\bull(U,M), b^\gt, B\big)$, 
{\em i.e.}, the homology of its total complex, the {\em cyclic Poisson homology}
of the Poisson bialgebroid $U^\gt$ and denote it by $HC^\gt_\bull(U,M)$.
\end{definition}

\begin{example} ({\em Hochschild and cyclic homology for left Hopf algebroids and associative algebras})
\label{erstes}
Let $\gt:=\mu$ be the distinguished element
\rmref{distinguished, I said 2}, {\em i.e.}, the trivial triangular $r$-matrix. Then, as shown in 
\cite[Eq.~(3.11) \& Lemma 4.18]{KowKra:BVSOEAT}, one has
\begin{equation}
\label{siena}
\begin{array}{rcccl}
b^\mu &=& -\lie_\mu &=& b, \\
\gb^\mu &=& \{\mu, \cdot\} &=& \gb,
\end{array}
\end{equation} 
where $b$ and $\gb$ are as in \rmref{appolloni1} and \rmref{appolloni2}. 
Hence, one reproduces the simplicial ({\em i.e.}, Hochschild) homology (with coefficients in $M$) resp.\ cohomology (with coefficients in $A$) for the left bialgebroid $U$, and the mixed complex related of Proposition \ref{c'est-a-dire} is the one arising from the cyclic module in the first part of Proposition \ref{kamille}.

In particular, for the case $U = \Ae$, this leads to the well-known Hochschild (co)homology 
of an associative algebra $A$, and the relation to the operators $b^\mu$ and $\gb^\mu$ for this case were already noticed in, {\em e.g.}, \cite{Tsy:FCFC, GerSch:ABQGAAD}. 
In both cases, {\em i.e.}, for general $U$ or in the example $U=\Ae$, the cyclic Poisson homology is then simply the cyclic homology of $U$ resp.\ the classical cyclic homology for associative algebras (cf.~\cite{Con:NCDG, FeiTsy:AKT}).
\end{example}

\begin{example} ({\em Poisson (co)homology for associative algebras})
\label{zweites}
In
the case $U = \Ae$ for a not necessarily commutative associative $k$-algebra $A$, a triangular 
$r$-matrix $\pi \in C^2(\Ae,A)$ can, as mentioned before, be seen as a conventional Hochschild $2$-cocycle $\tilde\pi$ by means of the isomorphism 
$ C^\bull(\Ae, A) \to C^\bull(A,A)$
in \rmref{yamamay}. 
Such a Hochschild $2$-cocycle $\tilde\pi$ with the property $\{\tilde{\pi}, \tilde{\pi}\}_\ggii = 0$ was named {\em noncommutative Poisson structure} in \cite{Xu:NCPA}, where we write $\{.,.\}_\ggii$ for the classical Gerstenhaber structure on Hochschild cohomology as mentioned 
in Example \ref{figlieditalia}.
 Using \rmref{yamamay} again, one obtains
$$
\widetilde{\gb^\pi(f)} = \tilde{\gb}^{\tilde{\pi}}(\tilde{f}), \quad \forall f \in C^\bull(\Ae, A),
$$
where $\tilde{\gb}^{\tilde{\pi}} := \{\tilde{\pi}, .\}_\ggii$ 
is the differential introduced in \cite{Xu:NCPA} that defines {\em noncommutative Poisson cohomology}. 

The differential $b^\pi$, in turn, leads in this context to the {\em noncommutative Poisson homology} defined in \cite[\S4.1]{NeuPflPosTan:HOFDOPELG}: similar to \rmref{yamamay}, there is an isomorphism
$$
C_\bull(\Ae,A) \simeq C_\bull(A,A),
$$
where the right hand side is the standard Hochschild chain complex,
and using this along with \rmref{yamamay}, one obtains analogously to the considerations in \cite[\S7.1]{KowKra:BVSOEAT} the operator
\begin{eqnarray*}
       b^{\tilde\pi}(a_0 \otimes_k \cdots \otimes_k a_n) &\!\!\!\!\!=&\!\!\!\!\! 
\sum^{n-1}_{i=0} (-1)^{n-i}  
a_0 \otimes_k \cdots \otimes_k {\tilde\pi}\big(a_{i} \otimes_k a_{i+1}\big) \otimes_k \cdots \otimes_k a_n \\
&& \quad +
{\tilde\pi}(a_{n} \otimes_k a_0) \otimes_k a_{1} \otimes_k \cdots \otimes_k a_{n-1},  
\end{eqnarray*}
which is the noncommutative Poisson boundary in \cite{NeuPflPosTan:HOFDOPELG}.

In both cases, for a Poisson manifold $\cP$ and the commutative algebra $A:= \cinfc \cP $ of compactly supported smooth functions on $\cP$, 
these two definitions lead to the differential geometric notions of Poisson cohomology resp.\ homology, as introduced
 by Lichnerowicz \cite{Lic:LVDPELADLA} resp.\ Koszul \cite{Kos:CDSNEC} and Brylinski \cite{Bry:ADCFPM}. 
The resulting mixed complex $\big(C^\bull(A,A), b^{\tilde{\pi}}, d \big)$, where $d$ is the de Rham differential for forms, was introduced in {\cite[\S1.3.4]{Bry:ADCFPM}}, and the corresponding cyclic homology of Poisson manifolds was further discussed in, {\em e.g.}, 
\cite{FerIbaLeo:TCSSFPM, Pap:HAAVDP, Vdb:NCHOSTDQS}.
\end{example}

\begin{example}
In \cite{Hue:PCAQ}, Poisson homology for a commutative Poisson algebra $A$ was introduced somewhat differently: if $\{.,.\}$ is a Poisson structure on a commutative algebra $A$,  
then the pair $(A, \gO^1_{A|k})$, 
where $\gO^1_{A|k}$ denotes the K\"ahler differentials over $A$, 
can be given the structure of a Lie-Rinehart algebra depending on $\{.,.\}$. Its Lie-Rinehart homology is 
then the Poisson homology of $A$. In view of \cite{KowPos:TCTOHA, KowKra:CSIACT}, this amounts to considering the cyclic (co)homology of the left $A$-bialgebroid given by the universal enveloping algebra $V(\gO^1_{A|k})$, 
which, in turn, depends on the right $(A,\gO^1_{A|k})$-module structure (see \cite{Hue:PCAQ} for the definition) given by 
$a \otimes b du \mapsto \{ab, u\}$ 
for $a,b,u \in A$. Hence, Poisson homology can be introduced in (at least) 
two ways arising from two different left bialgebroids $(A, \Ae)$ and $\big(A, V(\gO^1_{A|k})\big)$, respectively;
however, the latter
approach, in contrast to Example \ref{zweites}, does not include the case of  noncommutative Poisson algebras (as there is no notion of a Lie-Rinehart algebra over a noncommutative base algebra).
\end{example}

\begin{example}
Another way of arriving at the differential geometric notion of Poisson cohomology is by considering the bialgebroid given by the jet space $J\!L$ for a Lie-Rinehart algebra $(A,L)$ as in Example \ref{cambridge}, taking for $L$ the sections of the tangent bundle of a Poisson manifold. This situation will be discussed at length 
in \S\ref{pirano} and \S\ref{salice} below.
\end{example}

\subsection{Batalin-Vilkovisky algebra structures on commutative Poisson bialgebroids}
\label{castropretorio}

In this section, we assume that $U$ be a {\em commutative} Poisson left Hopf algebroid and for the sake of simplicity that
$M:=A$. 
In this case, one can prove that beyond the canonical Gerstenhaber 
algebra structure
 on $H^\bull(U,A)$, there is also one on $H_\bull(U,A)$, which is moreover a Batalin-Vilkovisky algebra. 
This essentially follows by an application to the bialgebroid case of Koszul's classical result in \cite{Kos:CDSNEC}. In there, Koszul considered a graded commutative algebra $S = \oplus_{p \geq 0} S^p$ 
(over a field the characteristic of which is different from two) 
with unit $1 \in S^0$. Moreover, let there be a differential operator $D \in \End(S)$ of at most second order and odd degree $r$, which vanishes on scalars ({\em i.e.}, $D(1) = 0$), and the square of which is supposed to be again of at most second order. 
In such a situation, the bracket
\begin{equation}
\label{romoletto}
\begin{split}
\{.,.\}_\dehhe: S^p \otimes S^q &\to S^{p+q+r}, \\
x \otimes y &\mapsto (-1)^p D(xy) + (-1)^{|p|} D(x)y - (-1)^{p|q|} D(y)x  
\end{split}
\end{equation}
generated by $D$ yields a Gerstenhaber algebra structure on $S$ and, by the very construction, even that of a Batalin-Vilkovisky algebra (see [{\em op.~cit.}]); for a recent contribution with respect to formality issues of such a bracket  see, {\em e.g.}, \cite{FioMan:FOKBADOHPM}, or \cite{BraLaz:HBVAIPG} for issues related to higher Koszul brackets and homotopy Batalin-Vilkovisky structures.

The necessary ingredients to apply this general fact to the homology groups
 $H_\bull(U,A)$ of a commutative Poisson left Hopf algebroid are the shuffle product and the Lie derivative in \rmref{messagedenoelauxenfantsdefrance2}.

\subsubsection{The shuffle product for commutative bialgebroids}
Recall from, {\em e.g.}, \cite[\S4.2]{Lod:CH} that the shuffle product on Hochschild chains leads to an {\em inner} shuffle product map 
provided that 
the algebra in question is commutative. We give here a version slightly adapted to the case of the simplicial module $C_\bull(U,A)$. 
Generally, the shuffle product map could be defined on any left bialgebroid $U$ provided the base algebra $A$ is central in $U$ (which happens, for example, if $U$ is a bialgebra over $A$ or if $U$ is a commutative left bialgebroid).

Hence, let $U$ be a commutative left bialgebroid over $A$ for the rest of this section. Although this means for $A$ to be commutative as well and therefore $A = \Aop$, we 
stick to the notation $U \otimes_\Aopp U := \due U {\blact} {} \otimes_{\Aopp} U_\ract$ to distinguish the various tensor products.
For $p, q \geq 1$, define
\begin{equation}
\label{shuffle1}
\begin{split}
& \cdot \times \, \cdot = \sh_{pq}: C_p(U,A) \otimes C_q(U,A)  \to 
C_{p+q}(U,A), \\
& (u^1, \ldots, u^p) \times (u^{p+1}, \ldots, u^{p+q}) :=  \sum_{\gs \in \Sh(p,q)} (-1)^\gs (u^{\gs^{-1}(1)},  \ldots, u^{\gs^{-1}(p+q)}),  
\end{split}
\end{equation}
where, as usual,
$$
\Sh(p,q) := \{ \gs \in S_{p+q} \mid \gs(1) < \ldots < \gs(p); \gs(p+1) < \ldots < \gs(p+q) \}
$$
is the set of $(p,q)$-shuffles in the symmetric group $S_{p+q}$. 
Additionally, set $\sh_{00} := m_\ahha$, the multiplication in $A$, and 
\begin{equation}
\label{shuffle0}
\begin{array}{rclrcl}
\sh_{p0}:  C_p(U,A) \otimes A &\!\!\!\! \to&\!\!\!\!  C_{p}(U,A), & 
(u^1, \ldots, u^p) \otimes a &\!\!\!\! \mapsto&\!\!\!\!  (u^1 \ract a, \ldots, u^p), \\
\sh_{0q}:  A \otimes C_q(U,A) &\!\!\!\! \to&\!\!\!\!  C_q(U,A), &
a \otimes (u^1, \ldots, u^q)  &\!\!\!\! \mapsto&\!\!\!\!  (u^1, \ldots, a \blact u^q), 
\end{array}
\end{equation}
but note that the second line is actually redundant, {\em i.e.}, $\sh_{p0} = \sh_{0p}$ since $U$ is commutative.
If we let 
$$
\sh := \sum_{{p, q \geq 0} \atop {p+q=n}} \sh_{pq}: \big(C_p(U,A) \otimes C_q(U,A)\big)_n \to C_{n}(U,A)
$$
be the sum of the shuffle products for $p+q=n$, a straightforward computation proves, analogous to the Hochschild case for associative algebras (see, {\it e.g.}, \cite[p.~312]{McL:H}), 
that the Hochschild boundary is a graded derivation of the shuffle product and that $\sh$ is a map of complexes:

\begin{lem}
\label{vincoli1}
For $x \in C_p(U,A)$, $y \in C_\bull(U,A)$, one has
\begin{equation}
\label{vincoli2}
b(x \times y) =  bx \times y + (-1)^p x \times by.  
\end{equation}
The map $\sh$ therefore is a map of complexes of degree $0$, {\em i.e.}, $[b,\sh]=0$.
Hence, the induced map 
$$
\cdot \times \, \cdot: H_p(U, A) \otimes H_q(U,A) \to H_{p+q}(U,A)
$$
establishes the structure of a graded commutative algebra on $H_\bull(U,A)$.
\end{lem}

\subsubsection{The $\Tor^U_\bull(A,A)$ groups as Batalin-Vilkovisky algebra}
In order to obtain a Gerstenhaber algebra structure on $\Tor^U_\bull(A,A)$, we will add to the shuffle product from the preceding subsection a bracket that is generated by the Lie derivative using a Poisson structure: to be able to do this, the left bialgebroid needs to carry the additional structure of a left Hopf algebroid:

\begin{theorem}
\label{barberini}
Let $U$ be a commutative Poisson left Hopf algebroid with triangular $r$-matrix $\gt$. 
Then there is a $k$-bilinear map
\begin{equation}
\label{freni&frizioni1}
\begin{split}
\{.,.\}_\gt: \ & C_p(U,A) \otimes C_q(U,A) \to C_{|p+q|}(U,A),  \qquad p, q \geq 0, \\
& x \otimes y \mapsto (-1)^{|p|} b^\gt(x \times y) + (-1)^{p} b^\gt x \times y + x \times b^\gt y,  
\end{split}
\end{equation}
where $C_\bull(U,A)$ is seen as graded commutative algebra by means of the shuffle product \rmref{shuffle1}, which 
induces a Batalin-Vilkovisky algebra structure 
\begin{equation}
\label{freni&frizioni2}
\{.,.\}_\gt: \  
H_p(U,A) \otimes H_q(U,A) \to H_{p+q-1}(U,A) 
%
\end{equation}
on homology.
\end{theorem}

\begin{rem}
If $\due U \blact {}$ is projective as left $A$-module, this yields a bracket 
$$
\{.,.\}_\gt: \Tor^U_p(A,A) \otimes \Tor^U_q(A,A) \to \Tor^U_{p+q-1}(A,A).
$$
\end{rem}

\begin{proof}[Proof of Theorem \ref{barberini}]
This will be proven basically by applying Koszul's result with respect to the Lie derivative \rmref{encarnita2} and the shuffle product \rmref{shuffle1}. 
From what was said around Eq.~\rmref{romoletto}, it is clear that it suffices to show that $\lie_\theta$, which is of odd degree if $\theta \in C^{\rm \scriptscriptstyle{even}}(U,A)$, 
and which vanishes on $C_0(U,A) =A$ by definition, is a differential operator of degree $2$ on the graded commutative algebra $\big(H_\bull(U,A), \times)$. 
Following Koszul \cite[\S1]{Kos:CDSNEC} we call, as in the case of an ungraded algebra, 
a differential operator $\mathcal{D} \in \End(S)$ acting on a graded commutative (unital) algebra $S$ of {\em second order} if
\begin{equation}
\label{tipa1}
m_\esse(\mathcal{D} \otimes \id)\big(dx \, dy \, dz\big) = 0.
\end{equation}
Here, $m_\esse$ is the multiplication in $S$ and $dx$ are the K\"ahler differentials in $\gO^1_{S|k}$, {\em i.e.}, we set $dx : = x 
\otimes 1 - 1 \otimes x$ as a map $d: S \to I/I^2 \simeq \gO^1_{S|k}$ (with the isomorphism suppressed), 
where $I$ is the ideal in $S \otimes S$ defined as the kernel of $m_\esse$, and $S \otimes S$ becomes a commutative graded algebra by factorwise multiplication. 
Explicitly, Eq.~\rmref{tipa1} means for any $x \in S^p, y \in S^q$, and $z \in S^r$
\begin{equation}
\label{tipa2}
\begin{split}
\mathcal{D}(xyz) &= \mathcal{D}(xy)z + (-1)^{p(q+r)} \mathcal{D}(yz)x + (-1)^{r(p+q)} \mathcal{D}(zx)y \\
&\quad - \mathcal{D}(x)yz + (-1)^{|p(q+r)|} \mathcal{D}(y)zx + (-1)^{|r(p+q)|} \mathcal{D}(z)xy + \mathcal{D}(1)xyz,
\end{split}
\end{equation}
which we will verify now for the Lie derivative along $\gt$:

\begin{lem}
For a commutative Poisson left Hopf 
algebroid $U^\gt$, the Lie derivative $\lie_\gt$ is a second-order differential operator on $H_\bull(U,A)$.
\end{lem}
\begin{proof}
First, observe
that the unit for the shuffle product \rmref{shuffle1} is given as in \rmref{shuffle0}, {\em i.e.}, by elements in $C_0(U,A) = A$, and that $\lie_\gt$ vanishes on $A$ by definition, hence 
$\lie_\gt(1_{C_0(U,A)}) = 0$. Second, using the terminology of \rmref{travaux}, it is immediate 
from the explicit form of $\lie_\gt$ in \rmref{encarnita2} that the first summand $\lie^{\scriptscriptstyle{\rm untw}}_\gt$ independently fulfils \rmref{tipa2}. 
Hence, to prove the lemma it is enough to show the same property independently for the twisted part $\lie^{\scriptscriptstyle{\rm tw}}_\gt$, 
that is, that for $x \in C_p(U,A), y \in C_q(U,A)$, and $z \in C_r(U,A)$, {\em when passing to homology,}
\begin{equation}
\label{deewangee}
\begin{split}
\lie^{\scriptscriptstyle{\rm tw}}_\gt(x \times y \times z) 
&= \lie^{\scriptscriptstyle{\rm tw}}_\gt
(x \times y) \times z + (-1)^{p(q+r)} \lie^{\scriptscriptstyle{\rm tw}}_\gt
(y \times z) \times x \\
&\quad  + (-1)^{r(p+q)} \lie^{\scriptscriptstyle{\rm tw}}_\gt
(z \times x) \times y - \lie^{\scriptscriptstyle{\rm tw}}_\gt
(x) \times y \times z \\
&\quad + (-1)^{|p(q+r)|} \lie^{\scriptscriptstyle{\rm tw}}_\gt
(y) \times z \times x + (-1)^{|r(p+q)|} \lie^{\scriptscriptstyle{\rm tw}}_\gt
(z) \times x \times y
\end{split}
\end{equation}
holds. This consists in a direct (but sufficiently tedious) verification, we only indicate the main steps. 
To start with, we list a few properties that are needed: 
for any $2$-cocycle $\gt$ one has,
for a commutative left bialgebroid,
 from \rmref{spaetkauf} 
 the identity
\begin{equation}
\label{portaportese1}
\gve(u)\gt(v,w)  - \gt(uv,w) + \gt(u,vw) - \gve(w)\gt(u, v) = 0
\end{equation}
for any $u, v, w \in U$.
Furthermore, dealing with the induced maps descending on homology, 
one deduces from \rmref{mulhouse2} that for a cocycle $\gt$
\begin{equation}
\label{portaportese2}
b \iota_\gt = \iota_\gt b =0
\end{equation}
is fulfilled on homology.
Apart from that, for any
$(u^1, \ldots, u^n) \in C_n(U,A)$ we can give the following expression for the Hochschild differential \rmref{appolloni1} followed by the extra degeneracy \rmref{extra}:
\begin{equation}
\label{portaportese3}
\begin{split}
s_{-1}b(u^1, \ldots, u^n) &= 
(u^1_+, \ldots, u^{n-1}_+, u^n_+ u^n_- u^{n-1}_- \cdots u^1_-) \\
&\quad + \sum^{n-1}_{i=1} (-1)^i (u^1_+, \ldots, u^{n-i}_+ u^{n-i+1}_+ , \ldots, u^{n}_+, u^{n}_- \cdots u^1_-) \\
&\quad + (-1)^n (u^2_+, \ldots, u^{n}_+, u^1_+ u^1_- u^{2}_- \cdots u^n_-), 
\end{split}
\end{equation}
as follows from \rmref{Sch6}, \rmref{Sch7}, \rmref{Sch9}, and \rmref{appolloni1}; in particular, 
the entire expression \rmref{portaportese3} equals zero when descending to homology. 
Considering the fact that  
\begin{equation}
\label{urmila}
\lie^{\scriptscriptstyle{\rm tw}}_\gt = - \iota_\gt s_{-1} + (-1)^{n+1} \iota_\gt s_{-1} t
\end{equation}
on elements of length $n$ in this situation, we 
will use Eqs.~\rmref{portaportese1}--\rmref{portaportese3} to rewrite the terms in $\lie^{\scriptscriptstyle{\rm tw}}_\gt(x \times y \times z)$: denote
$x := (u^1, \ldots, u^p)$, $y := (v^1, \ldots, v^q)$, $z := (w^1, \ldots, w^r)$, and  introduce the notation
$$
(x_{[+]}, x_{[-]}) := s_{-1}(x) = (u^1_+, \ldots, u^p_+, u^p_- \cdots u^1_-).
$$
Consider now one of the terms in the shuffle product $x \times y \times z$; for example, without loss of generality,
the element 
$(u^1, v^1, \ldots, v^q, u^2, \ldots, u^p, w^1, \ldots, w^r) = (u^1, y , u^2, \ldots, u^p, z)$. By commutativity of $U$, we can then compute, observing that $u \ract a = a \blact u$ in this case, 
\begin{footnotesize}
\begin{equation}
\label{rewriting1}
\begin{split}
&\iota_\gt  s_{-1} t (u^1, y, u^2, \ldots, u^p, z) 
\overset{\scriptscriptstyle{\phantom{\rmref{portaportese1}, \rmref{Sch8}}}}{=} 
\big(y_{[+]}, u^2_+, \ldots, u^p_+, z_{[+]} \ract \gt(y_{[-]} z_{[-]} u^1_- \cdots u^p_-, u^1_+)\big) 
\\
&\overset{\scriptscriptstyle{\rmref{portaportese1}, \rmref{Sch8}}}{=}
\ubs{=:(1)}{\big(y_{[+]}, u^2_+, \ldots, u^p_+, z_{[+]} \ract \gt(y_{[-]} z_{[-]} , u^1_+  u^1_- \cdots u^p_-)\big)} \\
& \qqquad\qqquad +  \ubs{=:(2)}{\big(y, u^2_+, \ldots, u^p_+, z \ract \gt(u^1_- \cdots u^p_-, u^1_+)\big)} \\
&\qqquad\qqquad  \ubs{=:(3)}{-\big(\gve(u^1_+) \blact y_{[+]}, u^2_+, \ldots, u^p_+, z_{[+]} \ract \gt(y_{[-]} z_{[-]}, u^1_- \cdots u^p_-)\big)}.
\end{split}
\end{equation}
\end{footnotesize}
Note that $(2)$ is already one of the terms of $ - \lie^{\scriptscriptstyle{\rm tw}}_\gt
(x) \times y \times z$ in \rmref{deewangee}, the correct sign turning out if one takes the sign of the shuffle for the element $(u^1, y , u^2, \ldots, u^p, z)$ into account.
By \rmref{portaportese2}, one furthermore has, using \rmref{Sch8} again,
\begin{footnotesize}
\begin{equation}
\label{rewriting2}
\begin{split}
&(3) =  \ubs{=:(4)}{-\big(u^1_+v^1_+, \ldots, v^q_+, u^2_+, \ldots, 
u^p_+, z_{[+]} \ract \gt(v^1_- \cdots v^q_- z_{[-]}, u^1_- \cdots u^p_-)\big)} \\
& + \ubs{=:(5)}{\sum^{|q|}_{i=1} (-1)^{i-1} \big(u^1_+, v^1_+, \ldots, v^i_+v^{i+1}_+, \ldots v^q_+, u^2_+, \ldots, 
u^p_+, z_{[+]} \ract \gt(v^1_- \cdots v^q_- z_{[-]}, u^1_- \cdots u^p_-)\big)} 
\\
& +  \ubs{=:(6)}{(-1)^{|q|} \big(u^1_+, v^1_+, \ldots, v^{q-1}_+, v^q_+u^2_+, u^3_+, \ldots, 
u^p_+, z_{[+]} \ract \gt(v^1_- \cdots v^q_- z_{[-]}, u^1_- \cdots u^p_-)\big)} \\
& + \ubs{=:(7)}{\sum^{|p|}_{j=2} (-1)^{j+q} \big(u^1_+, y_{[+]}, u^2_+, \ldots, u^j_+u^{j+1}_+, \ldots,
u^p_+, z_{[+]} \ract \gt(y_{[-]} z_{[-]}, u^1_- \cdots u^p_-)\big)} \\
& + \ubs{=:(8)}{(-1)^{p+q} \big(u^1_+, y_{[+]}, u^2_+, \ldots, u^{p-1}_+, u^p_+ w^1_+, w^2_+, \ldots, w^r_+ \ract \gt(y_{[-]} w^1_- \cdots w^r_-, u^1_- \cdots u^p_-)\big)} \\
& + \ubs{=:(9)}{\sum^{|r|}_{k=1} (-1)^{k+p+q} \big(u^1_+, y_{[+]}, u^2_+, \ldots, 
u^p_+, w^1_+, \ldots, w^{k}_+ w^{k+1}_+, \ldots, w^r_+ \ract \gt(y_{[-]} w^1_- \cdots w^r_-, u^1_- \cdots u^p_-)\big)} \\
& + \ubs{=:(10)}{(-1)^{r+p+q} \big(u^1_+, y_{[+]}, u^2_+, \ldots, 
u^p_+, w^1_+, \ldots, w^{r-1}_+ \ract \gt(y_{[-]} w^r_+ 
w^r_- \cdots w^1_-, u^1_- \cdots u^p_-)\big)} \\
& + \ubs{=:(11)}{(-1)^{|r+p+q|} \big(u^1_+, y_{[+]}, u^2_+, \ldots, 
u^p_+, w^1_+, \ldots, w^{r-1}_+ \ract \gt(w^r_+,  y_{[-]} w^1_- \cdots w^r_- u^1_- \cdots u^p_-)\big)} \\
& + \ubs{=:(12)}{(-1)^{r+p+q} \big(u^1, y_{[+]}, u^2, \ldots, 
u^p, w^1_+, \ldots, w^{r-1}_+ \ract \gt(w^r_+,  y_{[-]} w^1_- \cdots w^r_-)\big)}.
\end{split}
\end{equation}
\end{footnotesize}
We see that $(11) = (-1)^{r+p+q+1} \iota_\gt s_{-1}(u^1, y , u^2, \ldots, u^p, z)$, 
that is, rewriting the second term in \rmref{urmila} for each summand in the shuffle product $x \times y \times z$ on the left hand side of \rmref{deewangee} 
cancels the first summand in \rmref{urmila} so that we do not need to care about these terms in the following. Also, $(12)$ is one of the summands in $(-1)^{p(q+r)} \lie^{\scriptscriptstyle{\rm tw}}_\gt
(y \times z) \times x$; again, the right sign turns out if one takes the sign of the respective shuffle into account. 

As for the term $(1)$, we want to use the fact that \rmref{portaportese3} equals zero
on homology: 
to this end, 
compute, as in \rmref{rewriting1}, 
the second summand $\iota_\gt s_{-1} t$ of the twisted part of the Lie derivative for the element
$ 
(v^2, \ldots, v^q, x, z, v^1)
$
in the shuffle product. The terms that appear analogous to those in $(7)$ above produce the middle sum in \rmref{portaportese3}. 

However, 
the missing first summand in \rmref{portaportese3} so that the sum of 
all terms equal zero on homology {\em cannot} be directly produced by rewriting the terms $\iota_\gt s_{-1} t$ in the twisted part of the Lie derivative on any element in the shuffle product, but only by the following steps: similarly as in \rmref{rewriting1}, compute
\begin{footnotesize}
\begin{equation*}
\label{rewriting3}
\begin{split}
&\iota_\gt  s_{-1} t (y, u^1, \ldots, u^{p-1}, z, u^p) \\
&\quad=
\ubs{=:(13)}{\big(v^2_+, \ldots, v^q_+, u^1_+, \ldots, u^{p-1}_+, z_{[+]}, u^p_+ \ract \gt(y_{[-]} z_{[-]} , v^1_+  u^1_- \cdots u^p_-)\big)} \\
& \qquad +  \ubs{=:(14)}{\big(y, u^2_+, \ldots, u^p_+, z \ract \gt(u^1_- \cdots u^p_-, v^1)\big)} \\
&\qquad  \ubs{=:(15)}{-\big(\gve(v^1_+) \blact v^2_+, \ldots, v^q_+, u^1_+, \ldots, u^{p-1}_+, z_{[+]}, u^p_+ \ract \gt(y_{[-]} z_{[-]}, u^1_- \cdots u^p_-)\big)}.
\end{split}
\end{equation*}
\end{footnotesize}
Rewrite now the term $(15)$ in the spirit of \rmref{rewriting2}: 
since we are, for the moment, not interested in the terms that multiply elements outside the argument of $\gt$, {\em i.e.}, those analogous to $(4)$--$(9)$ in \rmref{rewriting2}, for the sake of simplicity we only write down the last three terms:
\begin{footnotesize}
\begin{equation*}
\label{rewriting4}
\begin{split}
(15) =  \ldots 
 &+ \ubs{=:(16)}{(-1)^{r+p+q} \big(y_{[+]}, u^1_+, \ldots, u^{p-1}_+, z_{[+]} \ract \gt(u^p_+ y_{[-]} z_{[-]}, u^1_- \cdots u^p_-)\big)} \\
& + \ubs{=:(17)}{(-1)^{|r+p+q|} \big(y_{[+]}, u^1_+, \ldots, u^{p-1}_+, z_{[+]} \ract \gt(u^p_+, y_{[-]} z_{[-]} u^1_- \cdots u^p_-)\big) } \\
& + \ubs{=:(18)}{(-1)^{r+p+q} \big(y_{[+]}, u^1, \ldots, u^{p-1}, z_{[+]} \ract \gt(u^p, y_{[-]} z_{[-]})\big) }.
\end{split}
\end{equation*}
\end{footnotesize}
Now $(16)$ above can be further rewritten as
\begin{footnotesize}
\begin{equation*}
\begin{split}
(16)&= \ubs{=:(19)}{(-1)^{r+p+q} \big(y_{[+]}, u^1_+, \ldots, u^{p-1}_+, z_{[+]} \ract \gt(y_{[-]} z_{[-]}, u^1_- \cdots u^p_- u^p_+)\big)} 
\\&\qqquad\qqquad 
+ \ubs{=:(20)}{(-1)^{r+p+q}\big(y, u^1_+, \ldots, u^{p-1}_+, z \ract \gt(u^p_+, u^1_- \cdots u^p_-)\big)}
\\&\qqquad\qqquad 
 \ubs{=:(21)}{-(-1)^{|r+p+q|}\big(y_{[+]}, u^1, \ldots, u^{p-1}, z_{[+]} \ract \gt(y_{[-]} z_{[-]}, u^p)\big)}.
\end{split}
\end{equation*}
\end{footnotesize}
While $(19)$ is the missing summand that cancels 
(when respecting the sign) with $(1)$ and the terms deriving from $(v^2, \ldots, v^q, x, z, v^1)$ on homology in the sense of \rmref{portaportese3}, whereas $(20)$ is another term of $ - \lie^{\scriptscriptstyle{\rm tw}}_\gt
(x) \times y \times z$, we have to deal with $(18)$ and $(21)$ which are not of the form of the terms in \rmref{deewangee}. Again by using \rmref{Sch8}, \rmref{portaportese1}, and the commutativity of $U$, one sees after some straightforward 
intermediate steps that
\begin{footnotesize}
\begin{equation*}
\begin{split}
&(-1)^{r+p+q} \big[(18)+(21)\big] \\
&= \ubs{=:(22)}{\big(y_{[+]}, u^1_+, \ldots, u^{p-1}_+, z \ract \gt(u^p_+, y_{[-]}u^p_- \cdots u^1_-)\big)} 
- \ubs{=:(23)}{\big(y_{[+]}, u^1_+, \ldots, u^{p-1}_+, z \ract \gt(y_{[-]}u^p_- \cdots u^1_-, u^p_+)\big)} \\
&\ + \ubs{=:(24)}{\big(y, u^1_+, \ldots, u^{p-1}_+, z_{[+]} \ract \gt(u^p_+, z_{[-]}u^p_- \cdots u^1_-)\big)} 
- \ubs{=:(25)}{\big(y, u^1_+, \ldots, u^{p-1}_+, z_{[+]} \ract \gt(z_{[-]}u^p_- \cdots u^1_-, u^p_+)\big)}.
\end{split}
\end{equation*}
\end{footnotesize}
Again, while $(22)$ and $(24)$ are terms that appear in the 
explicit expression of 
$\lie^{\scriptscriptstyle{\rm tw}}_\gt(x \times y) \times z$ and 
$(-1)^{r(p+q)} \lie^{\scriptscriptstyle{\rm tw}}_\gt
(z \times x) \times y$, respectively, the terms 
$(23)$ and $(25)$ 
do not appear 
in \rmref{deewangee}. One now proceeds 
recursively with these terms as in \rmref{rewriting1} to shift the tensor factors in the order in which they 
appear in the expression in \rmref{rewriting2}. We underline that this is only possible 
if one descends to homology as the rewriting \rmref{rewriting2} only holds due to \rmref{portaportese2}.

Proceeding with the same steps as above with respect to all other terms in the shuffle product $x \times y \times z$, it is moreover clear 
by a combinatorial argument that one produces all missing terms in \rmref{deewangee}. For example, one checks in a fashion analogous to proving the fact that the Hochschild differential is a graded derivation of the product as in Lemma \ref{vincoli1} 
that $(4)$, $(6)$, and $(8)$ cancel; we leave these remaining steps to the reader.
\end{proof}
Applying this lemma to Koszul's result in the way mentioned below the bracket \rmref{romoletto} proves Theorem \ref{barberini}.
\end{proof}

We will end this subsection by dealing with the two canonical examples:

\subsubsection{The zero bracket on Hochschild homology for associative algebras}
Let $A$ be a commutative associative $k$-algebra, and recall the context of Exs.~\ref{todi}, \ref{figlieditalia}, and \ref{erstes} in which $U=\Ae$. Then, if $A$ is projective over $k$, the groups $\Tor^\Ae_\bull(A,A) = HH_\bull(A)$ yield the classical Hochschild homology
of $A$ with values in $A$, and the bracket \rmref{freni&frizioni1} for the operad multiplication \rmref{distinguished, I said 2} by means of \rmref{siena} reads 
\begin{equation*}
\{x,y\}_\mu = 
(-1)^{|p|} b(x \times y) + (-1)^{p} b(x) \times y + x \times b(y).  
\end{equation*}
As mentioned in \rmref{vincoli2}, this equals zero already on the chain level and therefore the Gerstenhaber bracket vanishes.

\subsubsection{The Koszul bracket on forms}
\label{pirano}
Recall from \cite{MacXu:LBAPG} that a {\em triangular $r$-matrix} or {\em Poisson bivector} for a Lie-Rinehart algebra $(A,L)$ is an element $\pi \in \bigwedge^2_\ahha \! L$ with the property $[\pi, \pi] =0$, where $[.,.]$ denotes the classical Schouten-Nijenhuis bracket specified in \rmref{labicana0}. If $L$ is finitely generated $A$-projective, Koszul \cite{Kos:CDSNEC} proved that the exterior algebra $\bigwedge^\bull_\ahha \! L^*$ of $L^* := \Hom_\ahha(L,A)$ is a Batalin-Vilkovisky algebra with bracket
\begin{equation}
\label{elis}
\{\go,\eta\}_\pi = (-1)^p \big(\mathsf{L}_\pi(\go \wedge \eta) - \mathsf{L}_\pi(\go) \wedge \eta - (-1)^p \go \wedge \mathsf{L}_\pi(\eta) \big),
\end{equation}
for $\eta \in \textstyle{\bigwedge^\bull_\ahha \! L}$, $\go \in \textstyle{\bigwedge^p_\ahha \! L}$,  $p \geq 0$, and
which for $1$-forms $\ga, \gb \in \bigwedge^1_\ahha \! L^* = L^*$ becomes
the customary formula (which, according to \cite{Kos:PMLAMCAS}, appeared in \cite{AbrMar:FOM} for the first time)
$$
[\ga, \gb]_\pi = \mathsf{L}_{\pi^\#(\ga)}(\gb) - \mathsf{L}_{\pi^\#(\gb)}(\ga) - \mathsf{d}\mathsf{i}_\pi(\ga \wedge \gb).
$$
Here, by $\pi^\#$ we mean the map $\pi^\#: L^* \to L, 
\ \pi^\#(\ga)(\gb) := \pi(\ga,\gb)$, along with the classical operations $(\mathsf{L}, \mathsf{i}, \mathsf{d})$ of Lie derivative, contraction, and de Rham differential between forms and fields ({\em cf.}, for example, \cite[\S6.4]{KowKra:BVSOEAT} for the concrete form of these operators used here). 

In order to connect this Gerstenhaber bracket to our Gerstenhaber bracket from \rmref{freni&frizioni2}, we need to apply the construction in Theorem \ref{barberini} to the commutative left Hopf algebroid given by the jet space $J\!L$ mentioned in Example \ref{cambridge}. 
To this end, we briefly recall from \cite{KowPos:TCTOHA, KowKra:BVSOEAT} some facts that allow to apply the preceding results 
to complete Hopf algebroids such as $J\!L$: to have the structure maps ({\em e.g.}, those that define the cyclic module structure) well-defined, 
completed tensor products need to be used in the chain complex $C_\bull(J\!L,M)$. 
Similarly, in the definition of an SaYD module the
coaction should be given by a map 
$M \rightarrow J\!L \hat\otimes_\ahha M$. 
Dually, $C^\bull(J\!L,A)$ needs to be defined as
$\Hom_\Aop^\mathrm{cont}({J\!L^{\hat\otimes_\Aopp \bull}}_\ract, A)$,
where $\mathrm{cont}$ means that the cochains 
have to be continuous ($A$
being discrete), as only the operators assigned to these cochains will
be well-defined on the completed tensor products. 

Now, there is a
morphism of chain complexes  
\begin{equation}
\label{schreibwarenhandlung1}
        F: \big(\bar C_\bull(J\!L,A), \bb\big) \to 
        \big(\Hom_\ahha(\textstyle\bigwedge^\bull_\ahha \!L,A), 0\big),
\end{equation}
which, in degree $n >0 $, is given by
$$
        F(f^1, \ldots, f^n)(X^1 \wedge \cdots \wedge X^n) 
:= (-1)^n \big(Sf^1 \wedge \cdots \wedge Sf^n\big)(X^1, \ldots, X^n),
$$
while it is the identity on $A$ in degree $n = 0$.
As
$C_\bull(J\!L,A)$ is defined via completed tensor
products, we have
\begin{equation*}
        C_n(J\!L,A) \simeq \varprojlim
        \Hom_\ahha\big((V\!L^{\otimes_\ahha n})_{\leq p},A \big),  
  \end{equation*}
where $(V\!L^{\otimes_\ahha n})_{\leq p}$ is the degree $p$ part of
the filtration induced by that of $V\!L$.
That $F$ is well-defined on the normalised complex 
$\bar C_\bull(J\!L,A)$ follows 
since degenerate chains vanish under $F$.
When $L$ is finitely generated projective
over $A$, the wedge product of multilinear forms provides an
isomorphism 
$
                \textstyle\bigwedge^\bull_\ahha \!L^* \rightarrow 
                \Hom_\ahha(\textstyle\bigwedge^\bull_\ahha \!L,A) 
$
that we usually suppress in the sequel. 
In this case, one defines the map
$$
        F'(\ga^1 \wedge \cdots \wedge \ga^n) := 
        \sum_{\gs \in S_n}(-1)^\gs \big(\pr^*\ga^{\gs(1)}, 
        \ldots, \pr^*\ga^{\gs(n)}\big)
$$
for $\ga^1,\ldots,\ga^n \in L^*$, where $\pr: V\!L \to L$ denotes the projection on $L$ resulting
from Rinehart's PBW theorem \cite[Thm.~3.1]{Rin:DFOGCA}, and proves
\begin{equation}
\label{acquapanna}
        FF'=n! \,\mathrm{id}_{\scriptscriptstyle \bigwedge^n_\ahha
          \!L^{\scriptscriptstyle \ast}}.
\end{equation}
Finally, one has, dual to \rmref{schreibwarenhandlung1}, a
morphism
\begin{equation*}
F^*: \big(\!\textstyle\bigwedge^\bull_\ahha\!L, 0\big)  \to 
(\bar C^\bull(J\!L,A),\delta ) 
\end{equation*}
of cochain complexes explicitly given as 
$$
X^1 \wedge \cdots \wedge X^n \mapsto \big\{ (f^1, \ldots, f^n) \mapsto  (-1)^n \!\! \sum_{\gs \in S_n} (-1)^\gs (Sf^1)(X^{\gs(1)}) \cdots (Sf^n)(X^{\gs(n)}) \big\}.
$$
With these preparations at hand 
we can give the relation between the classical Koszul bracket on forms and the Gerstenhaber bracket \rmref{freni&frizioni2}:

\begin{prop}
If $L$ is finitely generated projective over $A$ and $\pi \in \bigwedge_\ahha^2 \! L$ 
is a triangular $r$-matrix, then for any $\go \in \bigwedge^p_\ahha \! L^*$ and $\eta \in \bigwedge_\ahha^q \! L^*$ the identity
\begin{equation}
\label{mesopotamia}
2 \, |p+q|! \, \{\go,\eta\}_{\pi} = F\big(\{F'\go, F'\eta\}_{F^*\pi}\big)
\end{equation}
holds.
\end{prop}
\begin{proof}
The proof
resembles, to some extent, the computations performed in the proof of \cite[Prop.~6.5]{KowKra:BVSOEAT} 
(but note that in [{\em op.~cit.},  Eqs.~(6.18)--(6.20)] 
the factors $(n+1)$, $(n-1)$, resp.\ $n$ should rather read $(n+1)!$, $(n-1)!$, resp.\ $n!$): 
first of all, it follows from [{\em loc.~cit.}] and \cite[Thm.~1.4]{Cal:FFLA} that if $\pi$ is a triangular $r$-matrix for the classical Schouten-Nijenhuis bracket, then $F^*\pi$ is one with respect to the Gerstenhaber bracket constructed in \S\ref{palermo}.

Assume that $\go := \ga^1 \wedge \cdots \wedge \ga^p$ and  $\eta := \gb^1 \wedge \cdots \wedge \gb^q$ for $p,q \geq 1$ (if either $p$ or $q$ is zero, the proof is analogous, but simpler). 
We then have for the second summand of the right hand side in \rmref{mesopotamia} by means of \rmref{freni&frizioni1} and observing that $\lie_{F^*\pi}$ is an operator of degree $-1$:
\begin{footnotesize}
\begin{equation*}
\begin{split}
& F \big(\cL_{F^*\pi} F'(\ga^1 \wedge \cdots \wedge \ga^p) \times F'(\gb^{1} \wedge \cdots \wedge \gb^{q}) \big)\big(X^1 \wedge \cdots \wedge X^{|p+q|}\big) \\
& =\frac{|p+q|!}{|p|! \, q!} F \big(\cL_{F^*\pi} F'(\ga^1 \wedge \cdots \wedge \ga^p) \otimes_\ahha F'(\gb^{1} \wedge \cdots \wedge \gb^{q}) \big)\big(X^1 \wedge \cdots \wedge X^{|p+q|}\big) \\
& =\frac{|p+q|!}{|p|!} \big(F \cL_{F^*\pi} F'(\ga^1 \wedge \cdots \wedge \ga^p) \wedge \gb^{1} \wedge \cdots \wedge \gb^{q}) \big)\big(X^1, \ldots, X^{|p+q|}\big), 
\end{split}
\end{equation*}
\end{footnotesize}
where we used \rmref{acquapanna} in the last line.
Hence, 
if we proved that
$
F \cL_{F^*\pi} F'(\ga^1 \wedge \cdots \wedge \ga^p) = 2 \, |p|! \mathsf{L}_\pi (\ga^1 \wedge \cdots \wedge \ga^p),
$
it is clear that we obtain the second summand of the left hand side in \rmref{mesopotamia}, as given in \rmref{elis}.
Writing  
the Poisson bivector as $\pi:= \pi^1 \wedge \pi^2 \in \bigwedge^2_\ahha \! L$, we therefore compute with \rmref{encarnita2}, \rmref{kaesekuchen1}--\rmref{apfelorangeingwersaft}, \rmref{Sch1}, 
and the commutativity of $A$, along with $St = s$ and $S^2 =\id$:
\begin{footnotesize}
\begin{equation*}
\begin{split}
& F \big(\cL_{F^*\pi} F'(\ga^1 \wedge \cdots \wedge \ga^p)\big)\big(X^1 \wedge \cdots \wedge X^{p-1}\big) \\
& = 
F \Big(
\sum_{\gs \in S_{p}} (-1)^\gs \Big(\sum^{|p|}_{i=1} (-1)^{n-|i|} \big((\ga^{\gs(1)} \pr), \ldots,  (\ga^{\gs(i-1)} \pr), \\
& \qquad 
F^*\pi\big((\ga^{\gs(i)} \pr)_{(1)}, (\ga^{\gs(i+1)} \pr)_{(1)}\big) \lact (\ga^{\gs(i)} \pr)_{(2)} (\ga^{\gs(i+1)} \pr)_{(2)}, (\ga^{\gs(i+2)} \pr), \ldots, 
(\ga^{\gs(p)} \pr)\big) \\
& \quad - \big((\ga^{\gs(1)} \pr)_+, \ldots,  (\ga^{\gs(p-2)} \pr)_+, \\
& \qqquad\quad 
F^*\pi\big((\ga^{\gs(p)} \pr)_+, (\ga^{\gs(p)} \pr)_{-} \cdots (\ga^{\gs(1)} \pr)_{-}\big) \blact (\ga^{\gs(p-1)} \pr)_+\big) \\
&\quad +(-1)^{|p|} \big((\ga^{\gs(2)} \pr)_+, \ldots,  (\ga^{\gs(p-1)} \pr)_+, \\
& \qqquad \quad 
F^*\pi\big((\ga^{\gs(p)} \pr)_{-} \cdots (\ga^{\gs(1)} \pr)_{-}, (\ga^{\gs(p)} \pr)_+ \big) \blact (\ga^{\gs(p)} \pr)_+\big) 
\Big)
\Big)
\Big(X^1 \wedge \cdots \wedge X^{p-1}\Big) 
\\
& = 
2 \, |p|! \, \sum_{\gs \in S_{p}} (-1)^\gs \sum_{\tau \in S_2} (-1)^\tau \Big[\sum^{|p|}_{i=1} (-1)^{n-|i|} \ga^{\gs(1)}(X^{1})  \cdots  \ga^{\gs(i-1)}(X^{i-1}) \\
& \qqquad \qquad
\gve\Big(X^{i}_+ \gve\big(\pi^{\tau(1)}_+ (\ga^{\gs(i)} \pr)_{(1)}(\pi^{\tau(1)}_-)\big) \, \gve\big(\pi^{\tau(2)}_+ (\ga^{\gs(i+1)} \pr)_{(1)}(\pi^{\tau(2)}_-)\big) \,
\\
& \qqquad \qqquad
(\ga^{\gs(i)} \pr)_{(2)}(X^{i}_{-(1)}) \, (\ga^{\gs(i+1)} \pr)_{(2)} (X^{i}_{-(2)})\Big) \,
\ga^{\gs(i+2)}(X^{i+1}) \cdots \ga^{\gs(p)}(X^{p-1}) \\
& \quad - 
\gve\big(X^{1}_+ (\ga^{\gs(1)} \pr)_{(1)}(X^{1}_-)\big)  \cdots \gve\big(X^{p-1}_+ (\ga^{\gs(p-1)} \pr)_{(1)}(X^{p-1}_-)\big)  
\gve\big(\pi^{\tau(1)}_+(\ga^{\gs(p)} \pr)_{(1)}(\pi^{\tau(1)}_-)\big) 
\\
& \qqquad\quad 
\big((\ga^{\gs(p)} \pr)_{(2)} \cdots (\ga^{\gs(1)} \pr)_{(2)}\big)(\pi^{\tau(2)}) \,  
 \\
&\quad +(-1)^{|p|} 
\gve\big(X^{1}_+ (\ga^{\gs(2)} \pr)_{(1)}(X^{1}_-)\big)  \cdots \gve\big(X^{p-1}_+ (\ga^{\gs(p)} \pr)_{(1)}(X^{p-1}_-)\big)  \\
& \qqquad \quad 
\big((\ga^{\gs(p)} \pr)_{(2)} \cdots (\ga^{\gs(1)} \pr)_{(2)}\big)(\pi^{\tau(2)}) \, 
\gve\big(\pi^{\tau(1)}_+(\ga^{\gs(1)} \pr)_{(1)}(\pi^{\tau(1)}_-)\big) 
\Big] 
\end{split}
\end{equation*}
\begin{equation*}
\begin{split}
& = 
2 \, |p|! \, \sum_{\gs \in S_{p}} (-1)^\gs \sum_{\tau \in S_2} (-1)^\tau \Big[\sum^{|p|}_{i=1} (-1)^{n-|i|} \ga^{\gs(1)}(X^{1})  \cdots  \ga^{\gs(i-1)}(X^{i-1}) \\
& \qqquad \qquad
\gve\Big(X^{i}_+ \gve\big(\pi^{\tau(1)}_+ (\ga^{\gs(i)} \pr)(\pi^{\tau(1)}_- X^{i}_{-(1)})\big) \, 
\gve\big(\pi^{\tau(2)}_+ (\ga^{\gs(i+1)} \pr)(\pi^{\tau(2)}_- X^{i}_{-(2)})\big)\Big) \,
\\
& \qqquad \qqquad \qqquad
\ga^{\gs(i+2)}(X^{i+1}) \cdots \ga^{\gs(p)}(X^{p-1}) \\
& \quad - 
\gve\big(X^{1}_+ (\ga^{\gs(1)} \pr)(X^{1}_- \pi^{\tau(2)}_{(1)})\big)  \cdots \gve\big(X^{p-1}_+ (\ga^{\gs(p-1)} \pr) 
(X^{p-1}_- \pi^{\tau(2)}_{(p-1)} )\big)  \\
& \qqquad\quad 
\gve\big(\pi^{\tau(1)}_+(\ga^{\gs(p)} \pr)(\pi^{\tau(1)}_- \pi^{\tau(2)}_{(p)})\big) 
 \\
&\quad +(-1)^{|p|} 
\gve\big(X^{1}_+ (\ga^{\gs(2)} \pr)(X^{1}_- \pi^{\tau(2)}_{(2)})\big)  \cdots \gve\big(X^{p-1}_+ (\ga^{\gs(p)} \pr)_{(1)}(X^{p-1}_- \pi^{\tau(2)}_{(p)})\big)  \\
& \qqquad \quad 
 \gve\big(\pi^{\tau(1)}_+(\ga^{\gs(1)} \pr)(\pi^{\tau(1)}_- \pi^{\tau(2)}_{(1)})\big) \,
\Big] \\
& = 
2 \, |p|! \, \big(\mathsf{L}_{\pi^1 \wedge \pi^2}(\ga^1 \wedge \cdots \wedge \ga^p)\big)\big(X^1, \ldots, X^{p-1}\big),
\end{split}
\end{equation*}
\end{footnotesize}
where the last step is a (long but) 
straightforward verification using \rmref{rashomon1} and \rmref{rashomon2}, the fact that vector fields act by $X(a):= \gve(Xa)$ as derivations on $A$, and $\pr(XY - YX) = \pr([X,Y]) = [X,Y]$, 
along with $\pr(1_{\scriptscriptstyle_{V \!L}}) = 0$.
The respective computations for the remaining summands in \rmref{mesopotamia} are similar and therefore skipped.
\end{proof}

\subsection{Lie bialgebroids}
\label{salice}
In this section, we deal with the situation how the general construction of Gerstenhaber brackets
on $\Ext^\bull_U(A,A)$ and $\Tor^U_\bull(A,A)$ is related to the notion of Lie bialgebroids, or Lie-Rinehart bialgebras in its algebraic formulation. 

Recall from \cite{MacXu:LBAPG, Kos:EGAALBA, Hue:DBVAAFTLRA} that a {\em Lie-Rinehart bialgebra} (or {\em Lie bialgebroid}) 
is a pair $(L,K)$ of finitely generated Lie-Rinehart algebras $(A,L)$ and $(A,K)$ over the same base algebra $A$, where $K \simeq L^* := \Hom_A(L,A)$, such that one (hence both) of the following equivalent conditions is true:
\begin{enumerate}
\item
$(\bigwedge^\bull_\ahha \! L, [.,.], \wedge, \mathsf{d}_*)$ is a strong differential Gerstenhaber algebra;
\item
$(\bigwedge^\bull_\ahha \! K, [.,.]_*, \wedge, \mathsf{d})$ is one.
\end{enumerate} 
Here, $[.,.]$ is the Schouten-Nijenhuis bracket \rmref{labicana0} on $\bigwedge^\bull_\ahha \! L$ and $\mathsf{d}$ is the generalised de Rham differential on $\bigwedge^\bull_\ahha \! L^* \simeq \bigwedge^\bull_\ahha \! K$, whereas $[.,.]_*$ and $\mathsf{d}_*$ are the respective structures arising from the Lie-Rinehart algebra structure of $(A,K)$; see [{\em op.~cit.}] for further equivalent formulations. 
As proven in \cite{Kos:EGAALBA, Xu:GAABVAIPG}, when the Lie-Rinehart structure on $K \simeq L^*$ arises from a triangular $r$-matrix $\pi \in \bigwedge^2_\ahha \! L$ with Gerstenhaber bracket $[.,.]_* := [.,.]_\pi$ given as in \rmref{elis}, the Batalin-Vilkovisky algebra $(\bigwedge^\bull_\ahha \! L^*, [.,.]_\pi)$ is strong differential with respect to the de Rham differential $\mathsf{d}$ and hence $(L, L^*)$ is a Lie-Rinehart bialgebra.
Examples include the case of classical Lie bialgebras as introduced by Drinfel'd \cite{Dri:HSOLGLBATGMOCYBE}, hence the terminology.

In the case of a left bialgebroid or left Hopf algebroid $U$, 
the situation appears, of course, to be more general. 
In view of Example \ref{cambridge} and \S\ref{pirano} and what was said above, the right question to ask 
is when $H^\bull(U,A)$ and $H_\bull(U,A)$ are (strong) differential Gerstenhaber algebras, but in contrast to the example coming from Lie-Rinehart bialgebras as above, where  $(\bigwedge^\bull_\ahha \! L, [.,.], \wedge, \mathsf{d}_*)$ and $(\bigwedge^\bull_\ahha \! K, [.,.]_*, \wedge, \mathsf{d})$ are simultaneously strong differential Gerstenhaber algebras, these two structures are not necessarily related. In any case, if a triangular $r$-matrix is given, one proves the following:

\begin{prop}
Let $U^\gt$ be a Poisson bialgebroid with triangular $r$-matrix $\gt$. 
Then $H^\bull(U,A)$ (resp.~ $\Ext^\bull_{U}(A,A)$ when $U_\ract$ is $\Aop$-projective) forms a strong differential Gerstenhaber algebra with respect to the differential $\gb^\gt$. 
In case $U^\gt$ is commutative and carries additionally the structure of a left Hopf algebroid, the Batalin-Vilkovisky algebra $\big(H_\bull(U,A), b^\gt\big)$  (resp.~$\big(\Tor_\bull^{U}(A,A), b^\gt\big)$ when $\due U \blact {}$ 
is $A$-projective) is a strong differential Gerstenhaber algebra as well with respect to the cyclic differential $B$.
\end{prop}
\begin{proof}
The first statement, {\em i.e.}, the fact that $\gb^\gt = \{\gt, .\}$ fulfils the identities in Definition \ref{caffetommaseo} with respect to the Gerstenhaber bracket on $H^\bull(U,A)$ induced by \rmref{maxdudler} follows directly from the Leibniz rule and the graded Jacobi identity of the cup product and the bracket itself.

The second statement follows from the fact that the operator $B: H_\bull(U,A) \to H_{\bull+1}(U,A)$ induced 
 by the cyclic differential (and denoted by the same symbol) from \rmref{extra} 
is a graded derivation of the shuffle product, {\em i.e.}, one has 
$$
B(x \times y) = Bx \times y + (-1)^p x \times By, \qquad x \in H_p(U,A), \ y \in H_\bull(U,A),
$$ 
as can be shown, for example, along the lines for the Hochschild case in \cite[Cor.~4.3.4]{Lod:CH}. One then directly verifies the two identities in Definition \ref{caffetommaseo} for the Gerstenhaber bracket \rmref{freni&frizioni2}, 
which is straightforward using \rmref{trieste} resp.~\rmref{alles2}; for convenience of the reader, we nevertheless show the second one: 
\begin{equation*}
\begin{split}
B\{x,y\}_\gt &= (-1)^{|p|} B b^\gt( x \times y) + (-1)^p B b^\gt x \times y \\
& \qquad - b^\gt x \times B y + Bx \times b^\gt y 
+ (-1)^p x \times B b^\gt y \\
& = (-1)^{p} b^\gt (B x \times y) + (-1)^{|p|} b^\gt B x \times y + Bx \times b^\gt y \\
& \qquad + b^\gt(x \times By) - b^\gt x \times B y + (-1)^{|p|} x \times b^\gt B y \\
& = \{Bx, y\}_\gt + (-1)^{|p|} \{x, By\}
\end{split}
\end{equation*}
 for $x \in H_p(U,A)$ and $y \in H_\bull(U,A)$, and this concludes the proof.
\end{proof}

\providecommand{\bysame}{\leavevmode\hbox to3em{\hrulefill}\thinspace}
\providecommand{\MR}{\relax\ifhmode\unskip\space\fi MR }
\providecommand{\MRhref}[2]{%
  \href{http://www.ams.org/mathscinet-getitem?mr=#1}{#2}
}
\providecommand{\href}[2]{#2}

\end{document}